\documentclass[12pt]{article}

\usepackage{amsrefs}

\usepackage{amsfonts, amsmath, amssymb, amsthm, color, hyperref, enumitem, fancyhdr, relsize, tikz, graphicx, mathrsfs}
\usepackage[utf8]{inputenc}
\usepackage{lipsum}

\setlist[description]{leftmargin=1cm,labelindent=0.5cm}

\usepackage{caption}
\usepackage{subcaption}

\title{\scshape 
A Projection Characterization and Symmetry Bootstrap for Elements of a von Neumann Algebra that are Nearby Commuting Elements
\sffamily }
\author{\scshape \sffamily David Herrera
}

\newtheorem{prop}{Proposition}[section]
\newtheorem{corollary}[prop]{Corollary}
\newtheorem{thm}[prop]{Theorem}
\newtheorem{lemma}[prop]{Lemma}

\theoremstyle{definition}
\newtheorem{remark}[prop]{Remark}
\newtheorem{example}[prop]{Example}

\newtheorem{defn}[prop]{Definition}

\newcommand{\R}{\mathbb{R}}
\newcommand{\C}{\mathbb{C}}
\newcommand{\N}{\mathbb{N}}

\newcommand{\Z}{\mathbb{Z}}

\newcommand{\bS}{\mathbb{S}}

\renewcommand{\ll}{\langle}
\newcommand{\rl}{\rangle}

\renewcommand{\Im}{\operatorname{Im}}
\renewcommand{\Re}{\operatorname{Re}}

\newcommand{\id}{\operatorname{id}}

\renewcommand{\H}{{\mathcal H}^{N-1}}
\renewcommand{\L}{{\mathscr L}}

\newcommand{\diag}{\operatorname{diag}}

\newcommand{\diam}{\operatorname{diam}}
\newcommand{\dist}{\operatorname{dist}}

\newcommand{\tr}{\operatorname{tr}}
\newcommand{\bp}{\begin{pmatrix}}
\newcommand{\ep}{\end{pmatrix}}

\newcommand{\A}{{\mathcal A}}
\newcommand{\B}{{\mathcal B}}
\newcommand{\vp}{\varphi}
\renewcommand{\H}{{\mathcal H}}

\newcommand{\dark}{}

\begin{document}
\dark
\large

\maketitle

\abstract{We define a symmetry map $\varphi$ on a unital $C^\ast$-algebra $\mathcal A$ 
to be an $\R$-linear map on $\mathcal A$ that generalizes transformations on matrices like: transpose, adjoint, complex-conjugation, conjugation by a unitary matrix, and their compositions. We include an overview of such symmetry maps on unital $C^\ast$-algebras.
We say that $A\in\A$ is $\varphi$-symmetric if $\varphi(A)=A$, $A$ is $\varphi$-antisymmetric if $\varphi(A)=-A$, and $A$ has a $\zeta=e^{i\theta}$ $\varphi$-phase symmetry if $\varphi(A)=\zeta A$. 

Our main result is a new projection characterization of two operators $U$ (unitary), $B$ that have nearby commuting operators $U'$ (unitary), $B'$.  
This can be used to ``bootstrap'' symmetry from operators $U, B$ that are nearby some commuting operators $U', B'$ to prove the existence of nearby commuting operators $U'', B''$ which satisfy the same symmetries/antisymmetries/phase symmetries as $U, B$, provided that the symmetry maps and symmetries/antisymmetries/phase symmetries satisfy some mild conditions.
We also prove a version of this for $X=U$ self-adjoint instead of unitary.

As a consequence of the prior literature and the results of this paper, we prove Lin's theorem with symmetries: If a $\varphi$-symmetric matrix $A$ is almost normal ($\|[A^\ast, A]\|$ is small), then it is nearby a $\varphi$-symmetric normal matrix $A'$. We also extend this further to include rotational and dihedral symmetries.

We also obtain bootstrap symmetry results for two and three almost commuting self-adjoint operators. 
As a corollary, we resolve a conjecture of \cite{loring2015k} for two almost commuting self-adjoint matrices in the Atland-Zirnbauer symmetry classes related to topological insulators.
}

\section{Introduction}
Lin's theorem (\cite{lin1996almost}) for matrices states that there is a function $\epsilon(\delta)$ such that $\epsilon(\delta)\to 0$ as $\delta\to 0^+$ so that for any matrix $A\in M_n(\C)$ with operator norm satisfying $\|A\| \leq 1$ (alias ``contraction''), there exists a normal matrix $A' \in M_n(\C)$ so that $\|A'-A\| \leq \epsilon(\|[A^\ast , A]\|)$, holding independently of $n$.  
Kachkovskiy and Safarov gave an explicit form for the function $\epsilon(\delta)$ which for finite dimensional $C^\ast$-algebras is:
\begin{thm}(\cite{kachkovskiy2016distance})
There is a universal constant $C_{KS}$ with the following property:
Let $A \in \A$, where $\A$ is a finite dimensional $C^\ast$-algebra. This can be expressed as $A = X+iY$ for unique $X, Y\in \A$ self-adjoint. Then there is a normal $A' \in \A$ expressible as $A'=X'+iY'$ for $X', Y'\in\A$ commuting self-adjoint operators such that 
\[\|A'-A\| \leq C_{KS}\|[A^\ast,A]\|^{1/2},\]
\[\|X'-X\|, \|Y'-Y\| \leq C_{KS}\|[X,Y]\|^{1/2}.\]
\end{thm}
The two inequalities are equivalent (up to a multiplicative factor).   
A natural question to ask is if we know certain things about $A$ then what can be known about $A'$. It was proven by Loring and S{\o}rensen (\cite{loring2016almost}) that if $A$ is a transpose-symmetric matrix then $A'$ can be chosen to be transpose-symmetric. This corresponds to $X', Y'$ being real if $X, Y$ are real. In \cite{loring2014almost}, they showed that if $A$ is a real matrix then $A'$ can be chosen to be real. This corresponds to $X, X'$ being real and $Y, Y'$ being purely imaginary.

More generalizations of symmetries for almost normal matrices and also other types of almost commuting matrices have been considered for their applications to symmetries in physics, in particular for applications to topological insulators. See \cites{loring2015k, loring2014quantitative, loring2016almost, loring2014almost, loring2013almost, hastings2010almost} and the examples later in this section.

\vspace{0.1in}

This paper is composed of three main parts. In Section \ref{Symmetry Maps section}, we will discuss what we call symmetry maps of unital $C^\ast$-algebras in great depth. Symmetry maps are $\R$-linear transformations that generalize conjugation by a unitary,  complex-conjugation, taking the transpose, taking the adjoint, and compositions of these matrix operations. 

We end the first part of the paper with Section \ref{Lin's Theorem with Linear Symmetries} by proving Lin's theorem on almost normal elements with linear symmetries. That is, let $\A$ be a unital $C^\ast$-algebra and ${S}$ be a collection of linear symmetry maps on $\A$. Suppose that ${S}$ satisfies certain commutativity conditions (which we call being admissible) and suppose that $\A$ has a symmetry-preserving form of stable rank 1: that $S$-symmetric elements of $\mathcal A$ are nearby invertible $S$-symmetric elements of $\mathcal A$. We show that Lin's theorem with symmetries holds for this case as a consequence of the known symmetry results of Lin's theorem.

This is done by reducing the problem to either Lin's theorem for unital $C^\ast$-algebras with stable rank 1 as proven by Friis and R{\o}rdam (\cite{friis1996almost}) or Lin's theorem with a single order 2 anti-multiplicative linear symmetry map when the unital $C^\ast$-algebra has a symmetry-respecting version of stable rank 1 as proven by Loring and S{\o}rensen (\cite{loring2016almost}). Symmetry reductions are discussed in Section \ref{Symmetry Reductions}. 

In the second part of this paper, we prove a projection characterization of two operators in a von Neumann algebra that are nearby commuting operators of the same type, where one of the matrices is normal with spectrum in the unit circle or the real line. Section \ref{Spectral Projection Inequalities and Localization Operators} discusses some technical results concerning symmetry-respecting spectral projections and localizer operators. Section \ref{Almost Reducing Projections and Nearly Commuting Operators} discusses the construction of and use of the projections $F$ in both directions of the bootstrap.

In the third part of this paper, we use the results from Section \ref{Almost Reducing Projections and Nearly Commuting Operators} to derive new and rather general symmetry results for almost commuting operators and for almost normal operators.
In Section \ref{SymmBootstrap}, we obtain a symmetry bootstrap for two and three almost commuting self-adjoint operators, Lin's theorem with linear and conjugate-linear symmetries (which provide a reflection symmetry), and a symmetry bootstrap for a unitary operator which includes reflection, rotational (by an angle that is a rational multiple of $2\pi$), and dihedral phase symmetries. 
In Section \ref{rot-dih Lin section} we use the polar decomposition and the symmetry bootstrap for a unitary matrix to prove Lin's theorem with rotational and with dihedral phase symmetries in Theorem \ref{Rotational-Dihedral Lin}.

\vspace{0.1in}

We now discuss the main concepts underlying the bootstrap construction.
The idea underlying the projection characterization is simple: if we know that some operators $U$ (unitary), $B$ are nearby commuting operators $U'$ (unitary), $B'$, then we can extract certain spectral projections $F'$ from $U'$, which then necessarily commute with $B'$. By rotating the range of $F'$ into the range of some spectral projections of $U$, we obtain projections $F$ that contain and are contained in some spectral projections of $U$ and have the property that they almost commute with $B$. (They also have some other technical properties needed to make the rest of the construction work properly.) 
See Lemma \ref{proj requirement} and Lemma \ref{proj requirement sa}.

So, using the fact that there are some commuting operators $U', B'$ nearby, we obtain projections $F$ that almost commute with $B$ and that serve as substitutes for the spectral projections of $U$ over arcs in the unit circle. This is non-trivial because although $B$ almost commutes with $U$, it is a basic fact that $B$ rarely then almost commutes with the spectral projections of $U$. 

These constructed projections $F$ (or any projections satisfying certain technical properties) can be used to construct nearby commuting matrices $U'', B''$.
So, we obtain the result that the existence of certain projections is a characterization of operators that are nearby commuting operators of the same type. This  construction relies on the fact that the spectrum of the normal operator $U$ belongs to a $1$-dimensional set.

If the type of the first operators is unitary we denote them as $U, U'$. If the type is instead self-adjoint, we refer to them as $X, X'$. There is not a requirement on the type of $B$. See Lemma \ref{proj-converse} and Lemma \ref{proj-converse sa}. 

For the reader interested in matrices, one should view all unital $C^\ast$-algebras and von Neumann algebras as finite dimensional and hence isomorphic to direct sums of matrix algebras.

\begin{example}
We now illustrate how we can bootstrap a reflection symmetry of a unitary. For simplicity, we assume that the operators are matrices. Suppose that $U, B$ are both real matrices with $U$ unitary. This means that if $\phi(A)=\overline{A}$ is the conjugation map then $\phi(U)=U, \phi(B)=B$. We write the spectral decomposition \[U = \sum_{z\in \sigma(U)} zE_z,\] 
where $E_z=E_{\{z\}}(U)$ is the spectral projection of $U$ onto its $z$-eigenspace. Then because $\phi$ is a positive, conjugate-linear, and multiplicative map, 
\[\phi(U) = \sum_{z\in \sigma(U)} \overline{z}\phi(E_z)\]
and the $\phi(E_z)$ are projections. Further, $\phi(I)=I$ so the projections $\phi(E_z)$ form a resolution of the identity. Since the spectral decomposition is unique, we see that $U$ being real is equivalent to $\phi(E_z) = E_{\overline{z}}$.
What this tells us is that in order for $U''$ to be $\phi$-symmetric, what we need is that there is a certain relationship determined by the conjugate-linear symmetry map $\phi$ between the spectral projections for $U''$ of conjugate eigenvalue. 

So, brushing aside some technicalities, by the projection characterization of nearly commuting matrices, let $F$ be the projections obtained that are substitutes for the spectral projections of $U$ whose ranges are localized with respect to the spectrum of $U$ in the portion of the unit circle that is in the upper half-plane. 
We, however, do not use the projection characterization to obtain the certain projections $F$ for the portion of the spectrum in the lower half-plane. We instead use 
the images of the projections $F$ under $\phi$. 
We need to construct certain projections $G$ to ``fill in the gaps'' of the spectrum to obtain a spectral resolution of $U''$.
Then because $\phi$ has order 2, $\phi$ interchanges $F$ and $\phi(F)$ so that the constructed matrix $U''$ will be real.

By perturbing $B$ to a certain matrix that is localized with respect to the spectrum of $U$ in a way that commutes with $U''$, we can obtain a matrix $B_0$ that commutes with $U''$ and is nearby $B$ and hence almost real. Then defining $B''=\frac12(B_0+\overline{B_0})$ provides a real matrix that commutes with $U''$ and is near $B$. 
In our proof, however, we do not need to symmetrize because the map that transforms $B$ into $B''=B_0$ commutes with $\phi$ so  that $B''$ is automatically $\phi$-symmetric because $B$ is.

Note that the type of symmetry of $B$ is not really important because the map $B \mapsto B''=B_0$ is independent of $B$ and the type of symmetry of $B$. For instance, if $B$ is purely imaginary so that $\phi(B)=-B$ then as in the previous case, our map $B \mapsto B''=B_0$ will commute with $\phi$ so that $B''$ is automatically $\phi$-antisymmetric when $B$ is.
\end{example}

\begin{example}
The Symmetry Bootstrap method from the previous example works when the symmetry does not ``descend to the individual spectral projections'' but instead generically only permutes far-away spectral projections due to acting as a non-trivial orthogonal transformation of $\R$ or $\C=\R^2$. 

For instance, consider self-adjoint matrices $X, B$ with $\varphi$ a linear symmetry of order $2$ (e.g. the transpose) such that $\varphi(X)=-X$ and $\varphi(B) = B$. Writing $X = \sum_{x\in \sigma(X)} xE_x$ with $\sigma(X) \subset \R$, we see that $\varphi(X) = -X$ is equivalent to $\varphi(E_x) = E_{-x}$.

Thus, we can use the commuting matrices $X', B'$ provided by Lin's theorem and construct the requisite projections $F$ from spectral projections on $(0, \infty)$. Then the projections used to construct $X''$ will be the $F$ and $\varphi(F)$. We then use the rest of the Symmetry Bootstrap method as outlined in the previous example to prove the existence of commuting $X'', B''$ satisfying  $\varphi(X'')=-X''$, $\varphi(B)=B$. 

However, if $\varphi(X)=X$, $\varphi(B)=B$ then our method cannot take any nearby commuting $X', B'$ and construct $\varphi$-symmetric commuting $X'', B''$. This is because the requirement on $X$ is equivalent to $\varphi(E_x)=E_x$. So, the property of $X$ being $\varphi$-symmetric descended to its spectral projections. We are not aware of a way to make the projections $F$ be $\varphi$-symmetric in this case. 
\end{example}

As corollaries of our Symmetry Bootstrap method, we obtain various new results concerning structured almost commuting matrices.
We state our main symmetry bootstrap results in the general context of von Neumann algebras satisfying certain symmetric stable rank 1 properties. However, we will state simplified forms here in the introduction for matrices. Theorem \ref{Symmetry Lin} implies:
\begin{thm}\label{Symmetry LinMatrix}
There exists a function $\epsilon(\delta)$ independent of $n$ with $\lim_{\delta \to 0}\epsilon(\delta) =0$ 
that satisfies the following property: If $X, Y \in M_n(\C)$ are self-adjoint contractions, then there are commuting self-adjoint $X'', Y''\in M_n(\C)$ that satisfy
\[\|X''-X\|, \|Y''-Y\| \leq \epsilon(\|[X, Y]\|)\]
where the $X'', Y''$ can be chosen with additional structure if $X, Y$ have additional structure as in one of the following cases:
\begin{description} 
\item[\underline{$(+,+)$}:] If there are unitaries $W_\alpha\in M_n(\C)$ that commute with $X, Y$ then $X'', Y''$ can be chosen to commute with all the $W_\alpha$ simultaneously.
\item[\underline{$(+,-)$}:] If a single unitary matrix $W\in M_n(\C)$ commutes with $X$ and anti-commutes with $Y$ then $X'', Y''$ can be chosen to commute and anti-commute with this unitary $W$, respectively. 
\item[\underline{$(-,-)$}:] If a single unitary matrix $W\in M_n(\C)$ anti-commutes with $X, Y$ then $X'', Y''$ can be chosen to anti-commute with this unitary $W$.
\item[\underline{$(+,+)$T}:] If a single unitary matrix $W\in M_n(\C)$ satisfies $ W X W^\ast=X^T$, $W Y W^\ast=Y^T$ then $X'', Y''$ can be chosen to also satisfy $W X'' W^\ast=(X'')^T$, $W Y'' W^\ast=(Y'')^T$.
\item[\underline{$(+,-)$T}:] If a single unitary matrix $W\in M_n(\C)$ satisfies $W X W^\ast=X^T$, $ W Y W^\ast=-Y^T$ then $X'', Y''$ can be chosen to also satisfy $W X'' W^\ast=(X'')^T$, $W Y'' W^\ast=-(Y'')^T$.
\item[\underline{$(-,-)$T}:] If a single unitary matrix $W\in M_n(\C)$ satisfies $ W X W^\ast=-X^T$, $W Y W^\ast=-Y^T$ then $X'', Y''$ can be chosen to also satisfy $W X'' W^\ast=-(X'')^T$, $W Y'' W^\ast=-(Y'')^T$.
\end{description}
Moreover, one can simultaneously satisfy case $(+,+)$, and/or multiple instances of case $(+,+)$\textbf{T}, and/or one instance of one of the other cases ($(+,-)$, $(-,-)$, $(+,-)$\textbf{T}, $(-,-)$\textbf{T}) provided that the associated symmetry maps of each case commute with each associated symmetry map of the other case(s).
\end{thm}
\begin{remark}
Consider the case of simultaneously having the symmetry $(+,+)$ with $W_\alpha$, one instance of $(+,+)$\textbf{T} with unitary $W_1$, and one instance of $(+,-)$ with $W_2$. Define the symmetry maps $\varphi_\alpha(A) = W_\alpha^\ast A W_\alpha$, $\varphi(A) = W_1^\ast A^T W_1$, and $\phi(A) = W_2^\ast A^\ast W_2$. The simultaneous symmetries corresponds to $A$ being $\{\varphi_\alpha, \varphi\}\cup\{\phi\}$-symmetric.

We  require $\varphi, \phi$ to commute with each $\varphi_\alpha$ and for $\varphi, \phi$ to commute with each other. Theorem \ref{Symmetry Lin} implies that Lin's theorem holds in this case. See Proposition \ref{translation} and the following remark for more about converting between the symmetries of two almost commuting self-adjoint matrices and a single almost normal matrix.
\end{remark}

\begin{remark}
Let each of $U, V$ be a unitary or an anti-unitary operator on a Hilbert space $\mathcal H$. For $W = U, V$, let $f_W(A) = W^{-1} A W$ and $g_W(A) = W^{-1} A^\ast W$ define symmetry maps on $\mathcal A = B(\mathcal H)$. Then $f_U, f_V$ commute when for all $A \in B(\mathcal H)$,
\[U^{-1} \left(V^{-1}AV\right) U = V^{-1} \left(U^{-1}AU\right)V,\]
\[\Leftrightarrow \;\; \left(VUV^{-1}U^{-1}\right)^{-1}A\left(VUV^{-1}U^{-1}\right)  = A.\]
Hence, the multiplicative commutator: $VUV^{-1}U^{-1}$ is a unitary that commutes with every operator in $B(\mathcal H)$. Therefore, it is some multiple $\zeta\in \C$ with $|\zeta|=1$ of the identity operator and hence $UV = \zeta VU$. That is, $U, V$ commute up to a phase $\zeta$.

By Example \ref{examples}(\ref{herm conj lin}), $f_W(A^\ast)=f_W(A)^\ast$ and $g_W(A^\ast)=g_W(A)^\ast$ so this same reasoning shows that the symmetry maps defined above commute if and only if the respective unitary/anti-unitary operators commute up to a phase.
\end{remark}

\begin{remark}
The separate cases of $(+,+)$, $(+,+)$\textbf{T}, $(+,-)$\textbf{T} are already known in the cases of $W = I$:
Lin's theorem is $(+,+)$ with $W = I$. Loring and S{\o}rensen's result (\cite[Theorem 6.10]{loring2016almost}) that two almost commuting real self-adjoint matrices are nearby commuting real self-adjoint matrices is $(+,+)$\textbf{T} with $W = I$. They further showed $(+,+)$\textbf{T} with $W^2=I$. Loring and S{\o}rensen's result (\cite[Theorem 1.2]{loring2014almost}) that a real almost normal matrix is nearby a real normal matrix is equivalent to $(+,-)$\textbf{T} with $W = I$. 

Versions of Lin's theorem with the self-dual symmetry correspond to different cases of our above result for $W = \Gamma$ defined in Example \ref{examples} below.
\end{remark}

\begin{remark}
The general case of $(+,+)$ can be easily reduced to Lin's theorem on the unital $C^\ast$-subalgebra of matrices in $M_n(\C)$ that commute with each of the $W_\alpha$.

Likewise, the general case of $(+,+)$\textbf{T }can be similarly reduced to the result of Loring and S{\o}rensen in \cite{loring2016almost} as follows: Consider the linear anti-multiplicative symmetry map $\varphi(A)=  W^\ast A^T W$. Then $\varphi(A) = A$ if and only if $A^T = WAW^\ast$. We need to show that if $X, Y$ are $\varphi$-symmetric then $X'', Y''$ can be chosen to be $\varphi$-symmetric as well. Note that 
\[\varphi^2(A) = W^\ast (W^\ast A^T W)^T W = W^\ast W^T A (W^\ast)^T W.\]
So, defining $W_0 = (W^\ast)^T W = \overline{W}W$, we see that $\varphi^2(A) = W_0^\ast A W_0$ is a linear multiplicative symmetry map.
Moreover, if $\varphi(A)=A$ then $\varphi^2(A)=A$ so we can restrict to finding a nearby commuting matrix in the unital $C^\ast$-subalgebra $\mathcal B$ of $\varphi^2$-symmetric matrices which is an invariant subspace for $\varphi$ on which $\varphi|_{\mathcal B}$ has order 2.
(Section \ref{Symmetry Reductions} contains a more general version of this reduction of symmetry maps and unital $C^\ast$-subalgebras.)

This means that solving Lin's theorem on $M_n(\C)$ with the linear anti-multiplicative symmetry $X^T = WXW^\ast, Y^T = WYW^\ast$ is equivalent to solving Lin's theorem on a certain 
finite dimensional $C^\ast$-subalgebra $\mathcal B$ with the linear anti-multiplicative symmetry map $\varphi|_{\mathcal B}$ of order $2$. This is a case addressed by \cite{loring2016almost}. 

See our Theorem \ref{Linear Lin} in Section \ref{Lin's Theorem with Linear Symmetries} below for the details of the reduction of Lin's theorem with linear symmetries to the already known results.
\end{remark}

For the sake of the application that we discuss next, Theorem \ref{Symmetry Lin} also implies:
\begin{thm}\label{Lin sym class}
There exists a function $\epsilon(\delta)$ independent of $n$ with $\lim_{\delta \to 0}\epsilon(\delta) =0$ 
that satisfies the following property: If $X, H \in M_n(\C)$ are self-adjoint contractions then there are commuting self-adjoint $X'', H''\in M_n(\C)$ that satisfy
\[\|X''-X\|, \|H''-H\| \leq \epsilon(\|[X, H]\|)\]
where the $X'', H''$ can be chosen with additional structure if $X, H$ have additional structure as in one of the following cases:
\begin{description} 
\item[\underline{$\{\mathcal W_\alpha, \mathcal T_\beta\}$}:] Suppose that $\{\mathcal W_\alpha\}_\alpha$, $\{\mathcal T_\beta\}_\beta$ (which may be empty) are unitaries, anti-unitaries (resp.) on $\C^n$ such that each $\mathcal T_\beta$ commutes with all the maps in $\{\mathcal W_\alpha\}_\alpha$, $\{\mathcal T_\beta\}_\beta$ up to a phase. 
Suppose that all the maps in $\{\mathcal W_\alpha\}_\alpha$, $\{\mathcal T_\beta\}_\beta$ commute with both $X, H$. Then the $X'', H''$ can be chosen to commute with all the $\mathcal W_\alpha, \mathcal T_\beta$ simultaneously.
\item[\underline{$\mathcal C$}:] In addition to the assumptions of case \textbf{\underline{$\{\mathcal W_\alpha, \mathcal T_\beta\}$}},
let $\mathcal C$ be an anti-unitary map on $\C^n$ that commutes with each of the $\mathcal W_\alpha, \mathcal T_\beta$ up to a phase. 
If $\mathcal C$ commutes with $X$ and anti-commutes with $H$ then $X'', H''$ can be chosen so that $\mathcal C$ commutes with $X''$ and $\mathcal C$ anti-commutes with $H''$ in addition to the results in \underline{$\{\mathcal W_\alpha, \mathcal T_\beta\}$}.
\item[\underline{$\mathcal S$}:] In addition to the assumptions of case \textbf{\underline{$\{\mathcal W_\alpha, \mathcal T_\beta\}$}}, let $\mathcal S$ be a unitary map on $\C^n$ that commutes with each of the $\mathcal W_\alpha, \mathcal T_\beta$ up to a phase. If $\mathcal S$ commutes with $X$ and anti-commutes with $H$ then $X'', H''$ can be chosen so that $\mathcal S$ commutes with $X''$ and $\mathcal S$ anti-commutes with $H''$ in addition to the results in \underline{$\{\mathcal W_\alpha, \mathcal T_\beta\}$}.
\end{description}
\end{thm}
\begin{proof}
Let $\varphi_{\alpha}(A) = {\mathcal W_\alpha}^{-1}A {\mathcal W_\alpha}$, $\varphi_{\beta}(A) = {\mathcal T_\beta}^{-1}A^\ast {\mathcal T_\beta}$. Then the $\varphi_{\alpha}$ are linear multiplicative symmetry maps such that $\varphi_{\alpha}(X)=X, \varphi_{\alpha}(H)=H$ and the $\varphi_{\beta}$ are linear anti-multiplicative symmetry maps such that $\varphi_{\beta}(X)=X, \varphi_{\beta}(H)=H$. Hence, $\varphi_\alpha(X+iH) = X+iH, \varphi_\beta(X+iH)=X+iH$. 
For the \underline{$\{\mathcal W_\alpha, \mathcal T_\beta\}$} case, the result follows by Lin's theorem with linear symmetries because $S = \{\varphi_\alpha, \varphi_\beta\}_{\alpha, \beta}$ define an admissible set of linear symmetry maps on $M_n(\C)$.

In the \underline{$\mathcal C$} case, let $\phi_{\mathcal C}(A) = {\mathcal C}^{-1}A {\mathcal C}$ so $\phi_{\mathcal C}$ is an conjugate-linear multiplicative symmetry map with
$\phi_{\mathcal C}(X)=X$, 
$\phi_{\mathcal C}(H)=-H$. 
Hence, $\phi_{\mathcal C}(X+iH) = X+iH$.

In the \underline{$\mathcal S$} case, let
$\phi_{\mathcal S}(A) = {\mathcal S}^{-1}A^\ast {\mathcal S}$ so $\phi_{\mathcal S}$ is a conjugate-linear anti-multiplicative symmetry map with
$\phi_{\mathcal S}(X)=X, 
\phi_{\mathcal S}(H)=-H$.
Hence, $\phi_{\mathcal S}(X+iH) = X+iH$.

The desired result in the other two cases then follows from Theorem \ref{Symmetry Lin}.
\end{proof}

\begin{remark}\label{sym class application}
Theorem \ref{Lin sym class} implies the validity of Conjecture 2.2 of \cite{loring2015k} related to topological insulators. This concerns whether an almost commuting system that has certain symmetries related to the physics of the system can have such symmetries preserved ``in the atomic limit'' as described by \cite{loring2015k}. That is to say, if certain symmetries and/or antisymmetries induced by conjugation by certain unitaries and/or anti-unitaries for almost commuting matrices can be preserved when finding the nearby commuting matrices. These specific examples of conjugation symmetries are also discussed in more detail in \cite{loring2013almost}.

Note that unlike \cite{loring2015k}, we will not refer to the unitaries/anti-unitaries $\mathcal T, \mathcal C, \mathcal S$ as ``symmetries''. We will, however, use the term ``symmetry'' to refer to the map on operators which are induced by conjugation by these maps.  So, for instance, when \cite{loring2015k} speaks of the symmetries having order $2$, it is meant that each of these $\R$-linear maps are their own inverses. However for our setting, if we say that a conjugation symmetry map has order $2$ on $M_n(\C)$ then this implies that the square of the conjugating map is a multiple of the identity.

Theorem \ref{Lin sym class} provides a proof that Lin's theorem with symmetries applies to the position operator $X$ and Hamiltonian $H$ being in one of the ten Atland-Zirnbauer symmetry classes because each symmetry class of Table 1 of \cite{loring2015k} can be expressed in terms of our terminology as:
\begin{itemize}
\item Class A in 1D is simply Lin's theorem with no symmetries imposed so it is the case of \underline{$\{\mathcal W_\alpha, \mathcal T_\beta\}$} with $\{\mathcal W_\alpha\}_\alpha=\emptyset$, $\{\mathcal T_\beta\}_\beta=\emptyset$.

\item Class AIII in 1D is the case of \underline{$\mathcal S$} with $\{\mathcal W_\alpha\}_\alpha=\emptyset$, $\{\mathcal T_\beta\}_\beta=\emptyset$ and $\mathcal S^2 = I$.

\item Class AI in 1D is the case of \underline{$\{\mathcal W_\alpha, \mathcal T_\beta\}$} with $\{\mathcal W_\alpha\}_\alpha=\emptyset$, $\{\mathcal T_\beta\}_\beta=\{\mathcal T\}$ and $\mathcal T^2=I$. 

\item Class BDI in 1D is the case of \underline{$\mathcal C$} with $\{\mathcal W_\alpha\}_\alpha=\emptyset$, $\{\mathcal T_\beta\}_\beta=\{\mathcal T\}$ and $\mathcal C \mathcal T = \mathcal T \mathcal C$, $\mathcal C^2=\mathcal T^2 = I$.

Note that because $\mathcal C^{-1} = \mathcal C$ commutes with $\mathcal T$, we could have instead applied case \underline{$\mathcal S$} with $\mathcal S = \mathcal C \mathcal T$.

\item Class D in 1D is the case of \underline{$\mathcal C$} with $\{\mathcal W_\alpha\}_\alpha=\emptyset$, $\{\mathcal T_\beta\}_\beta=\emptyset$ and $\mathcal C^2=I$.

\item Class DIII in 1D is the case of \underline{$\mathcal C$} with $\{\mathcal W_\alpha\}_\alpha=\emptyset$, $\{\mathcal T_\beta\}_\beta=\{\mathcal T\}$ and $\mathcal C \mathcal T = \mathcal T \mathcal C$, $\mathcal C^2=I$, $\mathcal T^2 = -I$.

\item Class AII in 1D is the case of \underline{$\{\mathcal W_\alpha, \mathcal T_\beta\}$} with $\{\mathcal W_\alpha\}_\alpha=\emptyset$, $\{\mathcal T_\beta\}_\beta=\{\mathcal T\}$ and $\mathcal T^2=-I$.

\item Class CII in 1D is the case of \underline{$\mathcal C$} with $\{\mathcal W_\alpha\}_\alpha=\emptyset$, $\{\mathcal T_\beta\}_\beta=\{\mathcal T\}$ and $\mathcal C \mathcal T = \mathcal T \mathcal C$, $\mathcal C^2=\mathcal T^2 = -I$.

\item Class C in 1D is the case of \underline{$\mathcal C$} with $\{\mathcal W_\alpha\}_\alpha=\emptyset$, $\{\mathcal T_\beta\}_\beta=\emptyset$ and $\mathcal C^2=-I$.

\item Class CI in 1D is the case of \underline{$\mathcal C$} with $\{\mathcal W_\alpha\}_\alpha=\emptyset$, $\{\mathcal T_\beta\}_\beta=\{\mathcal T\}$ and $\mathcal C \mathcal T = \mathcal T \mathcal C$, $\mathcal C^2=-I$, $\mathcal T^2 = I$.
\end{itemize}

Work toward this conjecture has been done through the symmetry results in several of the papers mentioned above. Our theorem resolves the conjecture by solving the problem for the remaining classes AIII, BDI, DIII, CII, and CI. (Note that we actually solve the problem for all cases that have an antisymmetry: which are all the cases except A, AI, and AII.)

Our result also implies other symmetry results that are not the subject of the aforementioned conjecture. A notable example is that Lin's theorem for two purely imaginary self-adjoint matrices is $(-,-)$\textbf{T} with $W = I_n$ which appears to be new. 
\end{remark}

Much of the prior work toward this conjecture takes a ``lifting'' $C^\ast$-algebraic approach. Our approach is more in the spirit of work of this subject by Davidson (\cite{davidson1985almost}) and Hastings (\cite{hastings2009making}) devoted to Lin's theorem (without symmetries). 
Our symmetry bootstrap is inspired by work done by Davidson (\cite{davidson1985almost}) and Voiculescu (\cite{voiculescu1983asymptotically}) but as far as the author knows their arguments have been used only to show that certain counter-examples were not actually nearby commuting matrices. (The index approach of Choi and by Exel and Loring have been the more recent approach to these type of counter-examples.) We apply the projection approach to matrices that are not counter-examples and hence obtain the framework of the bootstrap. 

Our Corollary \ref{sa normal sym} implies the following automatic symmetry result for three almost commuting self-adjoint matrices. Note that the wording reflects the important fact that three almost commuting self-adjoint matrices are not necessarily nearby three commuting self-adjoint matrices (\cites{voiculescu1981remarks, davidson1985almost}).

\begin{corollary}\label{bootstrap 3 matrices}
If $X_1, X_2, X_3 \in M_n(\C)$ are self-adjoint contractions such that there are commuting self-adjoint $X_1', X_2', X_3'\in M_n(\C)$ that satisfy
\[\varepsilon = \max\left(\|X_1'-X_1\|, \|X_2'-X_2\|, \|X_3'-X_3\|\right)\]
then there are commuting self-adjoint $X_1'', X_2'', X_3''$ that satisfy
\begin{align*}
\|X_1''-X_1\|&, \|X_2''-X_2\|, \|X_3''-X_3\|\\
&\leq 65\varepsilon^{1/4} + C_{KS}\max\left(\|[X_1,X_2]\|, \|[X_1,X_3]\|, \|[X_2,X_3]\|\right)^{1/2},
\end{align*}
and can be chosen with additional structure as in one of the following cases:
\begin{description} 
\item[\underline{$(+,+,-)$}:] If a unitary matrix $W \in M_n(\C)$ satisfies \[WX_1W^\ast=X_1, \,\, WX_2W^\ast=X_2, \;\; WX_3W^\ast=-X_3\] then $X_1'', X_2'', X_3''$ can be chosen so that $WX_1''W^\ast=X_1''$, $WX_2''W^\ast=X_2''$,\\ $WX_3''W^\ast=-X_3''$.
\item[\underline{$(+,-,-)$}:] If a unitary matrix $W \in M_n(\C)$ satisfies \[WX_1W^\ast=X_1,\;\; WX_2W^\ast=-X_2,\;\; WX_3W^\ast=-X_3\]
then $X_1'', X_2'', X_3''$ can be chosen so that
$WX_1''W^\ast=X_1''$, $WX_2''W^\ast=-X_2''$,\\ $WX_3''W^\ast=-X_3''$.
\item[\underline{$(-,-,-)$}:] If a unitary matrix $W \in M_n(\C)$ satisfies \[WX_1W^\ast=-X_1,\;\; WX_2W^\ast=-X_2,\;\; WX_3W^\ast=-X_3\] 
then $X_1'', X_2'', X_3''$ can be chosen so that
$WX_1''W^\ast=-X_1''$, $WX_2''W^\ast=-X_2''$, \\$WX_3''W^\ast=-X_3''$.
\item[\underline{$(+,+,-)$T}:] If a unitary matrix $W \in M_n(\C)$ satisfies \[WX_1W^\ast=X_1^T,\;\; WX_2W^\ast=X_2^T,\;\; WX_3W^\ast=-X_3^T\] 
then $X_1'', X_2'', X_3''$ can be chosen so that
$WX_1''W^\ast=(X_1'')^T$, $WX_2''W^\ast=(X_2'')^T$,\\ $WX_3''W^\ast=-(X_3'')^T$.
\item[\underline{$(+,-,-)$T}:] If a unitary matrix $W \in M_n(\C)$ satisfies \[WX_1W^\ast=X_1^T, \;\; WX_2W^\ast=-X_2^T,\;\; WX_3W^\ast=-X_3^T\] 
then $X_1'', X_2'', X_3''$ can be chosen so that
$WX_1''W^\ast=(X_1'')^T$, $WX_2''W^\ast=-(X_2'')^T$, \\$WX_3''W^\ast=-(X_3'')^T$.
\item[\underline{$(-,-,-)$T}:] If a unitary matrix $W \in M_n(\C)$ satisfies \[WX_1W^\ast=-X_1^T,\;\; WX_2W^\ast=-X_2^T,\;\; WX_3W^\ast=-X_3^T\]
then $X_1'', X_2'', X_3''$ can be chosen so that
$WX_1''W^\ast=-(X_1'')^T$, $WX_2''W^\ast=-(X_2'')^T$,\\ $WX_3''W^\ast=-(X_3'')^T$.
\end{description}
\end{corollary}
\begin{remark}\label{additional sym}
Suppose that three real self-adjoint matrices are nearby commuting  self-adjoint matrices. Our result is not able to answer whether or not these matrices are nearby commuting self-adjoint matrices that are real. 

However if instead at least one of the self-adjoint matrices is purely imaginary instead of real then there are nearby commuting matrices that also are real / purely imaginary (resp.). 

This shows that the symmetry in the construction in \cite{herrera2022constructing} of three commuting observables nearby the macroscopic observables for total $x$, $y$, and $z$ spin is actually a consequence of Ogata's original result. 
That is, the extra symmetry is automatic and not requiring any additional index to vanish. 

This type of phenomenon was observed for real unitaries in \cite{loring2014almost}. 
This automatic symmetry is not the case for linear symmetries in general, as is seen with three $\sharp$-symmetric self-adjoint matrices or two $\sharp$-symmetric unitaries in Theorems 1.2 and 1.6 of \cite{loring2013almost}.
\end{remark}

We also obtain another corollary:
\begin{corollary}
There exists a function $\epsilon(\delta)$ independent of $n$ with $\lim_{\delta \to 0}\epsilon(\delta) =0$ 
that satisfies the following property:

Suppose $X_1, X_2, H \in M_n(\C)$ are self-adjoint contractions such that there are commuting self-adjoint $X_1', X_2', H'\in M_n(\C)$ that satisfy
\[\varepsilon = \max\left(\|X_1'-X_1\|, \|X_2'-X_2\|, \|H'-H\|\right).\]
Suppose there are $\{\mathcal W_\alpha\}_\alpha$, $\{\mathcal T_\beta\}_\beta$ (which may be empty) of unitaries, anti-unitaries (resp.) on $\C^n$ 
such that each $\mathcal T_\beta$ commutes with all the maps in $\{\mathcal W_\alpha\}_\alpha$, $\{\mathcal T_\beta\}_\beta$ up to a phase and
such that each $\mathcal W_\alpha, \mathcal T_\beta$ commute with $X_1, X_1', X_2, X_2', H, H'$.

Then there are commuting self-adjoint $X_1'', X_2'', H''$ that satisfy
\begin{align*}
\|X_1''-X_1\|&, \|X_2''-X_2\|, \|H''-H\|\\
&\leq 45\sqrt{\varepsilon} + \epsilon\left(\max\left(\|[X_1,X_2]\|, \|[X_1,X_3]\|, \|[X_2,X_3]\|\right)+90\sqrt{\varepsilon}\right),
\end{align*}
that commute with each element of $\{\mathcal W_\alpha\}_\alpha$, $\{\mathcal T_\beta\}_\beta$, and that    can be chosen with an additional structure as in one of the following cases:
\begin{description} 
\item[\underline{$\mathcal C$}:] Let $\mathcal C$ be an anti-unitary map on $\C^n$ such that $\mathcal C$ commutes with each of the $\mathcal W_\alpha, \mathcal T_\beta$ up to a phase. If $\mathcal C$ commutes with $X_1, X_2$ and anti-commutes with $H$ then $X_1'', X_2'', H''$ can be chosen so that $\mathcal C$ commutes with $X_1'', X_2''$ and $\mathcal C$ anti-commutes with $H''$.
\item[\underline{$\mathcal S$}:] Let $\mathcal S$ be a unitary map on $\C^n$ such that $\mathcal S$ commutes with each of the $\mathcal W_\alpha, \mathcal T_\beta$ up to a phase. If $\mathcal S$ commutes with $X_1, X_2$ and anti-commutes with $H$ then $X_1'', X_2'', H''$ can be chosen so that $\mathcal S$ commutes with $X_1'', X_2''$ and $\mathcal S$ anti-commutes with $H''$.
\end{description}
\end{corollary}
\begin{remark}
This result allows us to observe relationships between the obstructions to three almost commuting self-adjoint matrices being nearby commuting self-adjoint matrices in the the Atland-Zirnbauer symmetry classes in 2D from Remark \ref{sym class application} as discussed above. We outline what this results says for each symmetry class:
\begin{itemize}
\item It is already known that three almost commuting self-adjoint matrices are not necessarily nearby commuting self-adjoint matrices. This corresponds to almost commuting matrices in Class A in 2D. 

Our result does not guarantee the existence of nearby commuting matrices in this symmetry class.

\item Class AIII in 2D is the case of \underline{$\mathcal S$} with $\{\mathcal W_\alpha\}_\alpha=\emptyset$, $\{\mathcal T_\beta\}_\beta=\emptyset$ and $\mathcal S^2 = I$.

Our result implies that the only possible obstructions for almost commuting matrices to have nearby commuting matrices in Class AIII in 2D are those of having nearby commuting matrices in Class A in 2D.

\item Class AI in 2D is the case of three almost commuting self-adjoint matrices that commute with an anti-unitary $\mathcal T$ that satisfies $\mathcal T^2=I$. It is not known that there are necessarily nearby commuting self-adjoint matrices for matrices in this symmetry class, let alone with the nearby commuting matrices being in this symmetry class.

Our result does not guarantee the existence of any nearby commuting matrices in this symmetry class.

\item Class BDI in 2D is the case of \underline{$\mathcal C$} with $\{\mathcal W_\alpha\}_\alpha=\emptyset$, $\{\mathcal T_\beta\}_\beta=\{\mathcal T\}$ and $\mathcal C \mathcal T = \mathcal T \mathcal C$, $\mathcal C^2=\mathcal T^2 = I$.

Our result implies that the only possible obstructions for almost commuting matrices to have nearby commuting matrices in Class BDI in 2D are those of having nearby commuting matrices in Class AI in 2D.

\item Class D in 2D is the case of \underline{$\mathcal C$} with $\{\mathcal W_\alpha\}_\alpha=\emptyset$, $\{\mathcal T_\beta\}_\beta=\emptyset$ and $\mathcal C^2=I$.

Our result implies that the only possible obstructions for almost commuting matrices to have nearby commuting matrices in Class D in 2D are those of having nearby commuting matrices in Class A in 2D.

\item Class DIII in 2D is the case of \underline{$\mathcal C$} with $\{\mathcal W_\alpha\}_\alpha=\emptyset$, $\{\mathcal T_\beta\}_\beta=\{\mathcal T\}$ and $\mathcal C \mathcal T = \mathcal T \mathcal C$, $\mathcal C^2=I$, $\mathcal T^2 = -I$.

Our result implies that the only possible obstructions for almost commuting matrices to have nearby commuting matrices in Class DIII in 2D are those of having nearby commuting matrices in Class AII in 2D (see below).

\item Class AII in 2D is the case of \underline{$\{\mathcal W_\alpha, \mathcal T_\beta\}$} with $\{\mathcal W_\alpha\}_\alpha=\emptyset$, $\{\mathcal T_\beta\}_\beta=\{\mathcal T\}$ and $\mathcal T^2=-I$.
It is not known that there are necessarily nearby commuting self-adjoint matrices for matrices in this symmetry class, let alone with the nearby commuting matrices being in this symmetry class.

Our result does not guarantee the existence of any nearby commuting matrices in this symmetry class.

\item Class CII in 2D is the case of \underline{$\mathcal C$} with $\{\mathcal W_\alpha\}_\alpha=\emptyset$, $\{\mathcal T_\beta\}_\beta=\{\mathcal T\}$ and $\mathcal C \mathcal T = \mathcal T \mathcal C$, $\mathcal C^2=\mathcal T^2 = -I$.

Our result implies that the only possible obstructions for almost commuting matrices to have nearby commuting matrices in Class CII in 2D are those of having nearby commuting matrices in Class AII in 2D.

\item Class C in 2D is the case of \underline{$\mathcal C$} with $\{\mathcal W_\alpha\}_\alpha=\emptyset$, $\{\mathcal T_\beta\}_\beta=\emptyset$ and $\mathcal C^2=-I$.

Our result implies that the only possible obstructions for almost commuting matrices to have nearby commuting matrices in Class C in 2D are those of having nearby commuting matrices in Class A in 2D.

\item Class CI in 2D is the case of \underline{$\mathcal C$} with $\{\mathcal W_\alpha\}_\alpha=\emptyset$, $\{\mathcal T_\beta\}_\beta=\{\mathcal T\}$ and $\mathcal C \mathcal T = \mathcal T \mathcal C$, $\mathcal C^2=-I$, $\mathcal T^2 = I$.

Our result implies that the only possible obstructions for almost commuting matrices to have nearby commuting matrices in Class CI in 2D are those of having nearby commuting matrices in Class AI in 2D.
\end{itemize}
\end{remark}

\section{Symmetry Maps}\label{Symmetry Maps section}
In this section we define and give basic definitions of what we refer to as ``symmetry maps.'' These maps will be used to provide generalizations of properties such as a matrix in $M_n(\C)$ being ``symmetric,'' ``antisymmetric,'' ``real,'' and other structure for matrices.
For this entire section $\A$ will be a unital $C^\ast$-algebra. In our applications and results in later sections, it will often be finite dimensional or a von Neumann algebra. 
\begin{defn}
Let $\varphi:\A\to \A$ be an $\R$-linear map. We say $\varphi$ is \emph{multiplicative} if
\[\varphi(AB) = \varphi(A)\varphi(B), \;\; \mbox{ for all } A, B\in \A.\]
We say $\varphi$ is \emph{anti-multiplicative} if
\[\varphi(AB) = \varphi(B)\varphi(A), \;\; \mbox{ for all } A, B\in \A.\]
We say $\varphi$ is \emph{conjugate-linear} if
\[\varphi(cA) = \overline{c}\varphi(A), \;\; \mbox{ for all } A\in \A, c \in \C.\]
We reserve the term ``linear'' to solely mean $\C$-linear.
We say that $\varphi$ is \emph{hermitian} if \[\varphi(A^\ast) = \varphi(A)^\ast, \;\; \mbox{ for all } A\in \A.\]

We say that $\varphi$ has order $n$ if $\varphi^k$, the $k$-fold composition of $\varphi$, equals $\id$ for $k = n$ and is not equal to $\id$ for any smaller $k \geq 1$. Otherwise, we say that $\varphi$ has infinite order. If $\zeta \in \C$ satisfies $|\zeta| = 1$, we define the order of $\zeta$ similarly.
\end{defn}
\begin{defn}
Let $\varphi:\A\to\A$ be a hermitian $\R$-linear map that maps non-invertible elements to non-invertible elements. We say that $\varphi$ is a \emph{symmetry} map if $\varphi$ also satisfies the following three conditions:
\begin{enumerate}
\item $\varphi(I) =I$.
\item $\varphi$ is linear or conjugate-linear.
\item $\varphi$ is multiplicative or anti-multiplicative.
\end{enumerate}

If $\varphi$ is a symmetry and $A \in \A$ satisfies $\varphi(A) = A$ then we say that $A$ is $\varphi$-symmetric. If $\varphi(A) = -A$ then we say that $A$ is $\varphi$-antisymmetric. If $\varphi(A) = \zeta A$ for $|\zeta| = 1$, then we say that $A$ is $\varphi$-$\zeta$-symmetric or that $A$ has a $\varphi$-phase symmetry with phase $\zeta$.

If ${S}$ is a set of symmetry maps and $A$ is $\varphi$-symmetric for all $\varphi \in {S}$, we say that $A$ is ${S}$-symmetric. Let $\A_{S} = \{A\in \A: \varphi(A) = A\mbox{ for all }\varphi \in {S}\}$ be the set of elements of $\A$ that are fixed by every symmetry of ${S}$.

For general statements, we will refer to the symmetries of an element as its unital phase symmetries, symmetries/antisymmetries as real phase symmetries, and all other types of phase symmetries as non-real phase symmetries.
\end{defn}

\begin{remark}
A symmetry map $\varphi$ always maps invertible elements to invertible elements. So, the condition that a symmetry maps non-invertible elements to non-invertible elements is equivalent to ``$\varphi(A)$ is invertible if and only if $A$ is invertible'', which is automatically satisfied if $\varphi$ has finite order.
\end{remark}

\begin{example}\label{examples}
We now discuss various examples. A few of these examples are from \cite{loring2015k}. The examples that are we are interested in are typically for $\A = M_n(\C)$. However, even for this case, in order to recursively apply symmetry constructions, we may have to work in a class of unital $C^\ast$-algebras that are closed under subalgebras that are closed in the suitable topology.  Thus the reader interested in symmetries on $M_n(\C)$ can assume that $\mathcal A$ is a finite dimensional $C^\ast$-algebra.

\begin{enumerate}
\item The identity map $\vp= \id$ is a symmetry map. All elements of $\A$ are $\id$-symmetric. Only $0\in \A$ is $\id$-antisymmetric.
\item The transpose map $\vp(A) = A^T$ is a linear anti-multiplicative symmetry map. It has order 2, meaning that $\vp^2(A)=\vp(\vp(A)) = A$ for all $A \in M_n(\C)$. Transpose-symmetric matrices usually are referred to as ``symmetric matrices'' and transpose-antisymmetric matrices usually are referred to as ``antisymmetric matrices''. To avoid confusion, we will always refer to an operator being symmetric or antisymmetric with respect to a specific given symmetry map.
\item For $\A = M_{n}(\C)$ for $n$ even, the dual map $\sharp$ defined by 
\[\bp A &B \\ C&D\ep^\sharp =  \bp D^T &-B^T \\ -C^T&A^T\ep\] is also a linear anti-multiplicative symmetry map with order 2. $\sharp$-symmetric matrices are referred to as self-dual.
\item The adjoint map $\phi(A) = A^\ast$ is a conjugate-linear anti-multiplicative symmetry map. $\ast$-symmetric matrices are self-adjoint and $\ast$-antisymmetric are the skew-adjoint matrices.
\item The conjugation map $\phi(A) = \overline{A}$ is a conjugate-linear multiplicative symmetry map. The conjugation-symmetric matrices are the real matrices and the conjugation-antisymmetric matrices are purely imaginary.
\item For any unitary $U \in \A,$ conjugation by $U$: $\vp(A) = U^{-1}AU$ is a linear multiplicative symmetry. The $\vp$-symmetric elements are those that commute with $U$ and the $\vp$-antisymmetric elements are those that anti-commute with $U$. If $U$ belongs to a von Neumann algebra $\mathcal A$, one can show that $\varphi$ is weakly continuous on $\mathcal A$.
\item \label{herm conj lin} An anti-unitary map $U$ on a Hilbert space $\H$ is an invertible, conjugate-linear map such that for any $v, w \in \H$,
\[\ll Uv, Uw\rl = \overline{\ll v, w \rl} = \ll w, v\rl.\]
With $v = w$, this gives $\|Uv\| = \|v\|$ for all $v$. Note then that $U^{-1}$ is also anti-unitary. 

Let $\A= B(\H)$ be the space of bounded (linear) operators on $\H$. One can show that for any $A \in \A$ it is true that $U^{-1}AU \in \A$. 
So, conjugation by $U$: $\phi(A) = U^{-1}AU$ is a weakly continuous, conjugate-linear multiplicative symmetry.
\begin{proof}
To show adjoint-invariance, We need that $U^{-1}A^\ast U = (U^{-1}AU)^\ast$. Let $v, w \in \H$. Then 
\begin{align*}
\ll (U^{-1}AU)^\ast v, w\rl &= \ll v, U^{-1}AUw\rl = \overline{\ll Uv, AUw\rl}\\
&= \overline{\ll A^\ast Uv, Uw\rl}= \ll U^{-1}A^\ast Uv, w\rl
\end{align*}
which implies what we needed to show.

This symmetry map is weakly continuous because
\[\ll v, U^{-1}AU w\rl = \overline{\ll Uv, A(Uw)\rl} = \ll A(Uw), Uv\rl.\]
\end{proof}
\item If $\Gamma = \bp 0 & I_{n/2} \\ -I_{n/2} & 0\ep$ then a matrix $A$ is (complex) symplectic if $\Gamma^{-1} A^T \Gamma = A^{-1}$. If $A$ is unitary this means that $\Gamma^{-1} \overline{U} \Gamma = U$. So, if we define the symmetry map by $\phi(A) = \Gamma^{-1} \overline{A} \Gamma$ on $M_{n}(\C)$ then $\phi$ is conjugate-linear multiplicative and a unitary matrix is symplectic if and only if it is $\phi$-symmetric.

Note that conjugation by $\Gamma$ is equal to the composition of taking the transpose and $\sharp$, which are commuting symmetry maps as observed in the proof of Lemma 4.2 of \cite{loring2014almost}.
See \cite{hall2015lie} for more about symplectic matrices.
\item Any composition of symmetry maps is a symmetry map.  In particular, $\id$, transpose, conjugation, and the adjoint form an abelian group of order four where each non-trivial symmetry has order $2$.
\end{enumerate}
\end{example}

We also mention the following symmetry map of order $N$.
\begin{example}\label{Ogata symmetry}
Consider the symmetry relevant for Ogata's theorem (\cite{ogata2013approximating}) defined first on pure tensors of $N$ matrices in $M_d(\C)$. Let $\sigma$ be a permutation of the set $\{1, \dots, N\}$, so $\sigma$ belongs to the group $S_N$. Define $S_\sigma(A_1 \otimes A_2 \otimes \cdots \otimes A_N) = A_{\sigma^{-1}(1)}\otimes \cdots, \otimes A_{\sigma^{-1}(N)}$ which permutes the factors of the pure tensor product. We extend $S_\sigma$ to a linear map on $M_d(\C)^{\otimes N} = M_{d^N}(\C)$ by linearity.

We see that this map is well-defined as follows. Consider the Hilbert-Schmidt inner product on $M_{d}(\C)$ given by $\ll A, B \rl_{HS} = \tr(A^\ast B)$. One can show that the Hilbert-Schmidt inner product on $M_{d^N}(\C)$ satisfies:
\[\ll \bigotimes_i A_i, \bigotimes_i B_i \rl_{HS} = \prod_i\ll A_i, B_i \rl_{HS}.\] 
Choose an orthonormal basis $C^i$ of matrices in $M_d(\C)$ with respect to the Hilbert-Schmidt inner product. 
We then see that all possible tensor products of matrices $C^i$ will form an orthonormal subset of $M_{d^N}(\C)$. These are also spanning, so we see that $S_\sigma$ is well-defined.

The linear maps $S_\sigma$ are induced by conjugation by a unitary. 
More specifically, let $v_i \in \C^d$ be some collection of vectors and consider the unitary $U_\sigma$ that maps $\bigotimes_i v_i$ to $\bigotimes_i v_{\sigma^{-1}(i)}$.
This map is well-defined with inverse $U_{\sigma^{-1}}$. We then see that $S_\sigma(A) = U_{\sigma}AU_{\sigma}^{-1}$. So, $S_\sigma$ is a linear multiplicative symmetry that has the same order as $\sigma$.

Now, given a matrix $A \in M_d(\C)$ define $T_N(A)$ to be the average of the $N$-fold possible tensor products of $A$ with the identity $I_d$:
\[T_N(A)=\frac{1}{N}\left(A\otimes I_d^{\otimes (N-1)} + I_d\otimes A\otimes I_d^{\otimes (N-2)}+\cdots +I_d^{\otimes (N-1)}\otimes A\right).\]
For any $\sigma \in S_N$, $T_N(A)$ is $S_\sigma$-symmetric. Note that for a matrix to satisfy such a property, it need only be $S_\sigma$ symmetric for $\sigma$ being the transposition $(1\,2)$ and the cycle $(1\,2\,\dots,\,n)$, which have order $2$ and $n$, respectively. These symmetry maps do not commute.

Let $A_1, \dots, A_m \in M_d(\C)$ and define $H_i = T_N(A_i)$.
We note that the $H_i$ are almost commuting and $S_\sigma$-symmetric for every $\sigma \in S_N$.
Using the reduction in \cite{herrera2022constructing}, Ogata's theorem implies that the $H_i$ are nearby commuting matrices which are self-adjoint when the $A_i$ are self-adjoint.

A natural question that one can ask is if a symmetry version of Ogata's theorem holds. Namely, when $H_i$ satisfy certain symmetries and antisymmetries, is it true that the $H_i$ are nearby commuting matrices that also satisfy these symmetries.

How much symmetry that is preserved for $m \geq 3$ is still not known.
One can readily show from Lin's theorem with linear symmetries (Theorem \ref{Linear Lin}) that the nearby commuting matrices can be chosen to be $S_\sigma$-symmetric when $m=2$.

In \cite{herrera2022constructing}, we showed that when $m=3$ and $d=2$ then complex conjugation-symmetry can be preserved. By Remark \ref{additional sym}, our Symmetry Bootstrap also shows that conjugation symmetry can also be preserved when $m=3$ and $d \geq 2$ as a direct consequence of Ogata's original proof if one of the matrices $H_i$ is purely imaginary and the others are real. See also \underline{$(+,+,-)$\textbf{T}}, \underline{$(+,-,-)$\textbf{T}}, and \underline{$(-,-,-)$\textbf{T}} of Corollary \ref{bootstrap 3 matrices}.
\end{example}

It is a fact that for $\mathcal A = M_n(\C)$, these are all the types of symmetry maps:
\begin{prop}
Let $\varphi$ be a symmetry map on $M_n(\C)$.
\begin{description}
\item[LM:] If $\varphi$ is a linear multiplicative symmetry map then there exists a unitary matrix $U\in M_n(\C)$ such that $\varphi(A) = U^\ast AU$ for every $A \in M_n(\C)$.
\item[LAm:] If $\varphi$ is a linear anti-multiplicative symmetry map then there exists a unitary matrix $U\in M_n(\C)$ such that $\varphi(A) = U^\ast A^TU$ for every $A \in M_n(\C)$.
\item[ClM:] If $\phi$ is a conjugate-linear multiplicative symmetry map then there exists a unitary matrix $U\in M_n(\C)$ such that $\phi(A) = U^\ast \overline{A}U$ for every $A \in M_n(\C)$.
\item[ClAm:] If $\phi$ is a conjugate-linear anti-multiplicative symmetry map then there exists a unitary matrix $U\in M_n(\C)$ such that $\phi(A) = U^\ast A^\ast U$ for every $A \in M_n(\C)$.
\end{description}
\end{prop}
\begin{proof}
By Proposition \ref{properties}\ref{isometry} below, a linear multiplicative symmetry map is a $\ast$-isomorphism of the unital $C^\ast$-algebra $M_n(\C)$. Then \textbf{LM} above is a well-known fact.

The remainder of the proof follows immediately since every map the various cases satisfies the required properties and composing the map with the appropriate symmetry map (without the unitary) produces a symmetry map in the case of \textbf{LM}. The details are left to the reader.
\end{proof}
\begin{remark}
This classification of symmetry maps is the reason that we expressed a simplified version of our results as we did in Theorem \ref{Symmetry LinMatrix}.
\end{remark}
\begin{remark}\label{Group Isomorphism}
Composing symmetry maps produces another symmetry map whose type is determined by the types of the composed maps.

Identity the symmetry map cases with elements of the abelian multiplicative group $\Z_2\times \Z_2 = \{(\pm1, \pm 1)\}$:
\[\textbf{\mbox{LM}} \leftrightarrow (1,1), \;\; \textbf{\mbox{LAm}} \leftrightarrow (1,-1) \;\; \textbf{\mbox{ClM}} \leftrightarrow (-1,1) \;\; \textbf{\mbox{ClAm}} \leftrightarrow (-1,-1).\]
Then we see that this provides an isomorphism of sorts between symmetry equivalence classes $\{ \mathbf{LM}, \mathbf{LAm}, \mathbf{ClM}, \mathbf{ClAm} \}$ and $\Z_2^2$.

Consider a set ${S}$ of commuting symmetry maps on $\A$ containing symmetry maps of all types. If we think of the symmetry maps modulo the equivalence relation of the type of symmetry, this group is roughly speaking isomorphic to $\Z_2^2 = \{ (\pm1, \pm 1)\}$. The isomorphism maps $\varphi \mapsto (a, b)$, where the $a=1$ if and only if $\varphi$ is linear and $b=1$ if and only if $\varphi$ is multiplicative.

If ${S}$ does not have symmetry maps of all types, then under the equivalence relation it is a subset of $\Z_2 \times \Z_2$. 
\end{remark}
\begin{remark}
This classification can also be expressed in terms of anti-unitary operators because the transpose can be expressed as:
$A^T=U^{-1} A^\ast U$, where $U v = \overline{v}$ performs complex conjugation of the entries of $v \in \C^n$. We obtain the alternative classification of two of the cases as:
\begin{description}
\item[LAm:] If $\varphi$ is a linear anti-multiplicative symmetry map then there exists an anti-unitary operator $U$ such that $\varphi(A) = U^{-1} A^\ast U$ for every $A \in M_n(\C)$.
\item[ClM:] If $\phi$ is a conjugate-linear multiplicative symmetry map then there exists an anti-unitary operator $U$ such that $\phi(A) = U^{-1} A U$ for every $A \in M_n(\C)$.
\end{description}
\end{remark}

We now describe two operations on symmetry maps.
\begin{defn}
Let $\varphi$ be a symmetry map on $\A$. Define $\varphi_\ast$ by $\varphi_\ast(A) = \varphi(A^\ast)$. So, $\varphi_\ast$ is the composition of taking the adjoint and $\varphi$. Because taking the adjoint is a conjugate-linear anti-multiplicative symmetry map, replacing $\varphi$ with $\varphi_\ast$ inverts its properties as a symmetry map:
\[\mbox{linear}\longleftrightarrow\mbox{conjugate-linear}, \;\;\;\;\;\mbox{multiplicative}\longleftrightarrow\mbox{anti-multiplicative}\]
\end{defn}

We now list some basic properties of symmetry maps that will be used throughout and often without comment.
\begin{prop}\label{properties}
Suppose that $\vp:\A\to\A$ is a symmetry map. Then:
\begin{enumerate}[label=(\roman*)] 
\item\label{triple product} For any $A, B \in \A$, $\vp(ABA) = \vp(A)\vp(B)\vp(A)$.
\item If $A$ and $B$ commute then $\vp(A)$ and $\vp(B)$ commute.
\item The restriction of $\vp$ to an abelian $C^\ast$-subalgebra (with or without unit) is multiplicative.
\item For any polynomial $p(z)$ with real coefficients, $\varphi(p(A)) = p(\varphi(A))$.
\item For any invertible element $A \in \A$, $\varphi(A^{-1}) = \varphi(A)^{-1}$.
\item If $A$ is self-adjoint then $\vp(A)$ is self-adjoint.
\item $\vp$ is a positive map: if $A \geq 0$ then $\vp(A)\geq 0$. 
\item\label{isometry} $\vp$ is an isometry (and hence norm-continuous):  $\|\vp(A)\| = \|A\|$ for all $A\in \A$.
\item\label{orthog projs} If $\{F_n\}$ is a finite collection of (mutually) orthogonal projections then $\vp(F_n)$ is a collection of (mutually) orthogonal projections. 
\item\label{proj inequality} If $F$ and $G$ are two projections with $F \leq G$ then $\varphi(F)$ and $\varphi(G)$ are two projections with $\varphi(F) \leq \varphi(G)$.
\item If $U$ is unitary then $\vp(U)$ is unitary as well.
\item If $N$ is normal then $\vp(N)$ is normal as well.
\item Let $f(z)$ be continuous and $N$ normal. Define $\overline{f}(z) = \overline{f(\overline{z})}$. Then 
\[\varphi(f(N))= \left\{\begin{array}{ll} 
f(\varphi(N)), &\mbox{ if }\varphi \mbox{ is linear}\\
\overline{f}(\varphi(N)), &\mbox{ if }\varphi \mbox{ is conjugate-linear}\end{array}\right..\]
\item\label{spectral proj transformation} Let $\Omega \subset \C$ be a set such that $E_{\Omega}(N)$ can be defined by weak convergence of polynomials using the Borel functional calculus. (For instance, $\Omega$ can be a point, an open or closed disk, or an open/half-closed/closed interval in $\R$ or arc in the unit circle.)

If $\mathcal A$ is a von Neumann algebra and $\varphi$ is weakly continuous then
\[\varphi(E_\Omega(N))= \left\{\begin{array}{ll} 
E_{{\Omega}}(\varphi(N)), &\mbox{ if }\varphi \mbox{ is linear}\\
E_{\overline{\Omega}}(\varphi(N)), & \mbox{ if }\varphi \mbox{ is conjugate-linear}\end{array} \right..\]
\item\label{B subalg} If ${S}$ is a set of linear multiplicative symmetry maps on $\A$ then the space $\A_{S}$ of ${S}$-symmetric elements of $\A$ forms a unital $C^\ast$-subalgebra of $\A$. 
\end{enumerate}
\end{prop}
\begin{proof}
\begin{enumerate}[label=(\roman*)]
\item If $\vp$ is multiplicative then $\vp(ABC)= \vp(A)\vp(B)\vp(C)$ and if $\vp$ is anti-multiplicative then $\vp(ABC)= \vp(C)\vp(B)\vp(A)$. Setting $C=A$ gives the result.
\item If $\vp$ is multiplicative then $\vp(A)\vp(B) = \vp(AB) = \vp(BA) = \vp(B)\vp(A)$. 

Although a similar proof shows that the result holds for $\vp$ anti-multiplicative, instead consider the following argument: \\
Note that $\vp_\ast$ is multiplicative. So if $A$ and $B$ commute then $\vp_\ast(A), \vp_\ast(B)$ commute. The result follows using $\vp(A) = \vp_\ast(A)^\ast$.
\item This follows from the previous statement.
\item Write $p(z) = \sum_n a_n z^n$ for $a_n\in\R$. Then 
\[\varphi(p(A)) = \sum_n \varphi(a_nA^n) = \sum_n a_n\varphi(A)^n = p(\varphi(A)).\]
\item The identity $\varphi(A)\varphi(A^{-1}) = \varphi(1) = 1$ holds when $\varphi$ is multiplicative or anti-multiplicative. So, $\varphi(A^{-1})=\varphi(A)^{-1}$.  
\item If $A$ is self-adjoint, $A^\ast = A$, so $\varphi(A)^\ast = \varphi(A^\ast) = \varphi(A)$. So, $\varphi(A)$ is self-adjoint.
\item If $A$ is positive, there is an element $B\in \A$ so that $A = B^\ast B$. Then $\varphi(A) = \varphi(B)^\ast\varphi(B)$ if $\varphi$ is multiplicative or $\varphi(A) =\varphi(B)\varphi(B)^\ast$ if $\varphi$ is anti-multiplicative. In either case, $\varphi(A)$ is positive.

\item For $A \in \A$,  $\varphi(A-\lambda I)$ is invertible whenever $A-\lambda I$ is. If $\varphi$ is linear, $\varphi(A)$ and $A$ have the same spectrum. If $\varphi$ is conjugate-linear, the spectrum of $\varphi(A)$ is the complex conjugate of the spectrum of $A$. 
So, in either case, $\varphi(A)$ and $A$ always have the same spectral radius. 

If $\varphi$ is multiplicative, we use $\|B^\ast B\| = \|B\|^2$ to obtain
\[\|\varphi(A)\|^2 = \|\varphi(A)^\ast \varphi(A)\|=\|\varphi(A^\ast A)\|, \; \|A\|^2 = \|A^\ast A\|.\]
Then we use the fact that the spectral radius for a self-adjoint element is its norm to conclude that 
\[\|\varphi(A)\|^2 = \|\varphi(A^\ast A)\|  = \|A^\ast A\| = \|A\|^2,\] which is what we wanted to show in the case that $\varphi$ is multiplicative. We can do something similar when $\varphi$ is anti-multiplicative except that we also must use $\|BB^\ast\| = \|B\|^2$ as well.

\item Because the $F_n$ are self-adjoint and $F_n^2=F_n$, $\varphi(F_n)$ are self-adjoint and $\varphi(F_n)=\varphi(F_n^2) = \varphi(F_n)^2$. So, $\varphi(F_n)$ are projections. Orthogonality follows from $F_nF_m = F_mF_n= 0$ when $n \neq m$ so
$\varphi(F_n)\varphi(F_m) = 0$.
\item The fact that $\varphi(F)$ and $\varphi(G)$ are projections follows from the previous statement. If $F \leq G$ then $F$, $(1-G)$ are orthogonal projections. So, $\varphi(F)$, $1-\varphi(G)$ are orthogonal projections. So, $\varphi(F)\leq \varphi(G)$.
\item $\vp(U)^{\ast} = \vp(U^\ast) = \vp(U^{-1}) = \vp(U)^{-1}$. So, $\vp(U)$ is unitary.
\item $\vp(N)$ and $\vp(N^\ast) = \vp(N)^\ast$ commute if $N$ is normal. So, $\vp(N)$ is normal.

\item We first show the result for polynomial $p(z, \overline{z})=\sum_{m,n} a_{m,n}z^m \overline{z}^n$ with generic entries. Note that 
\[\overline{p}(z, \overline{z}) = \overline{\sum_{m,n} a_{m,n}\overline{z}^m z^n} = \sum_{m,n} \overline{a_{m,n}}z^m \overline{z}^n.\] So, because $A$ is normal
\begin{align*}
\varphi&(p(A, A^\ast))= \sum_{m,n}\varphi(a_{m,n}A^mA^{\ast n})\\
&=\left\{\begin{array}{ll} 
\sum_{m,n}a_{m,n}\varphi(A)^m\varphi(A)^{\ast n}= p(\varphi(A)), &\mbox{ if }\varphi \mbox{ is linear}\\\sum_{m,n}\overline{a_{m,n}}\varphi(A)^m\varphi(A)^{\ast n}=
\overline{p}(\varphi(A)), & \mbox{ if }\varphi \mbox{ is conjugate-linear}
\end{array} \right..
\end{align*}
Choosing a sequence of polynomials $p_n(z,\overline{z})$ so that $z\mapsto p_n(z, \overline{z})$ converges to $f(z)$ uniformly on $\sigma(A)$ gives the result because $z \mapsto \overline{p_n}(z, \overline{z})$ is a sequence of polynomials converging uniformly to $\overline{f}(z)$ and $\varphi$ is continuous.
\item Let $f_n(z)$ be a uniformly bounded sequence of continuous functions on $\sigma(A)$ that converge pointwise to $\chi_\Omega(z)$ on $\sigma(A)$. Then $f_n(A)$ converges weakly to $\chi_{\Omega}(A) = E_{\Omega}(A)$. So, the result then follows from the calculation in (xiii) since $\chi_{\Omega}$ is real-valued and $\varphi$ is weakly continuous.
\item Closedness follows because each $\varphi \in {S}$ is continuous. The fact that it is a unital $\ast$-subalgebra of $\A$ follows from $\varphi(I)=I$ and since if $A, B$ are $\varphi$-symmetric and $z \in \C$ then $\varphi(AB) = \varphi(A)\varphi(B)=AB$, $\varphi(zA+B) = z\varphi(A)+\varphi(B)=zA+B$, and $\varphi(A^\ast) = \varphi(A)^\ast = A^\ast$.
\end{enumerate}
\end{proof}

\begin{remark}
Let $A$ be normal, $\varphi$ weakly continuous, and $\Omega$ a ``nice'' set. 
\begin{enumerate}[label=(\roman*)]
\item If $\varphi$ is linear and $A$ is $\varphi$-symmetric then the spectral projections $E_{\Omega}(A)$ of $A$ are $\varphi$-symmetric also.
\item If $\varphi$ is linear and $A$ is $\varphi$-antisymmetric then $\varphi(E_{\Omega}(A)) = E_{-\Omega}(A)$.
\item If $\varphi$ is conjugate-linear and $A$ is $\varphi$-symmetric then $\varphi(E_{\Omega}(A)) = E_{\overline{\Omega}}(A)$.
\item If $\varphi$ is conjugate-linear and $A$ is $\varphi$-antisymmetric then $\varphi(E_{\Omega}(A)) = E_{-\overline{\Omega}}(A)$.
\end{enumerate}
\end{remark}
If $\varphi$ is conjugate-linear then multiplying by $i$ makes a $\varphi$-symmetric matrix into an $\varphi$-antisymmetric matrix and vice-versa.

\begin{remark}\label{symmetrize}
A useful trick for constructing an element that is symmetric or antisymmetric for some symmetry map is averaging. Suppose that we have an operator $H$ that is almost $\varphi$-symmetric (i.e. $\|H - \varphi(H)\|$ is small) where $\varphi$ has order $2$. Then we can form the symmetrized operator $\tilde H = \frac{1}{2}(H + \varphi(H))$ so that $\tilde H$ is $\varphi$-symmetric and $\|\tilde H - H\| \leq \frac12\|H - \varphi(H)\|$.

Similarly, if $\varphi$ has order $n$, then we define $\tilde H = \frac{1}{n}(H + \varphi(H)+ \cdots + \varphi^{n-1}(H))$. We now provide a bound for how close $\tilde H$ is to $H$.
Applying $\varphi$ repeatedly to $\varphi(H) - H$ gives 
$\|\varphi^{k+1}(H) - \varphi^k(H)\|= \|\varphi(H)-H\|$.
So, using a telescoping sum gives
\[\|\varphi^k(H) - H\|\leq k\|\varphi(H)-H\|.\]
So, the bound
\begin{align}\label{symEst}
\|\tilde H - H\| \leq \frac1n\sum_{k=1}^{n-1}\|\varphi^k(H)-H\|\leq \frac1n\sum_{k=0}^{n-1}k\|\varphi(H)-H\| =\frac{n-1}{2}\|\varphi(H)-H\|.
\end{align}

This estimate is sharp as far as the optimal dependence on $n$. 
Consider Voiculescu's unitaries $U_n, V_n$ in Equation (\ref{Vunitaries}). These matrices satisfy the property that $V_n$ is diagonal, $U_n$ is off-diagonal, and $V_n^{-1}U_nV_n = e^{-2\pi i/n}U_n$.
So, if we define the symmetry $\varphi$ as conjugation by $V_n$ then $U_n$ is approximately $\varphi$-symmetric with 
\[\|\varphi(U_n)-U_n\| = |e^{-2\pi i/n}-1|\sim \frac{2\pi}n, \mbox{ as } n \to \infty\] 
and $\varphi$ has order $n$. Now, one can show that the symmetrized operator gotten by averaging $U_n$ over the orbit of $\varphi$ is equal to zero. 
Moreover, in order for an operator to be $\varphi$-symmetric, it needs to be diagonal, so $0$ is a $\varphi$-symmetric matrix with minimal distance from $U_n$. 
This shows that our bound in (\ref{symEst}) is, up to a constant, optimal. 

Suppose that we have that  $H$ almost has a phase $\varphi$-symmetry: $\|\varphi(H) - \zeta H\|$ is small. Suppose also that $\varphi$ is linear with order $m$. In order for us to even consider finding an operator with $\zeta$ phase symmetry, we require that $\zeta^m = 1$ since otherwise $\varphi(\tilde H) = \zeta \tilde H$ is only possible if $\tilde H = 0$. 

We now can average the action of $\zeta^{-1}\varphi$ on $H$ to define \[\tilde H = \frac1m\sum_{j=0}^{m-1} \zeta^{-j}\varphi^j(H).\]
Then 
\[\varphi(\tilde H) = \frac\zeta m\sum_{j=0}^{m-1} \zeta^{-j-1}\varphi^{j+1}(H)= \frac\zeta m\sum_{k=1}^{m} \zeta^{-k}\varphi^{k}(H) = \zeta \tilde H,\]
because $\zeta^{-m}\varphi^m =\zeta^{-0}\varphi^0$. Applying the same reasoning as above but with the isometric linear map $\zeta^{-1}\varphi$ that has order $m$ (but is not a symmetry map because it does not map $I$ to itself), we obtain the estimate
\[\|\tilde H - H\| \leq \frac{m-1}{2}\|\zeta^{-1}\varphi(H)-H\|=\frac{m-1}{2}\|\varphi(H)-\zeta H\|\]
as before.
\end{remark}

We end this section by discussing the standard translation of almost normal elements of a $C^\ast$-algebra $\A$ into almost commuting self-adjoint elements of $\A$ and vice-versa as well as providing estimates. The purpose is to see how symmetries are transformed.

We will primarily focus on how the symmetries and antisymmetries are affected by these conversions. The main ideas for these symmetry conversions for examples such as the transpose, conjugation, and such can be found in the arguments of papers such as \cite{loring2016almost} and \cite{loring2015k}. 

The conversion between an almost normal $A$ and two almost commuting self-adjoint elements $X$ and $Y$ is standard.
For $A \in \A$, let $X = \Re(A) = \frac12(A+A^\ast)$ and $Y = \Im(A)= \frac{1}{2i}(A-A^\ast)$. Because $i[X,Y] = \frac12[A^\ast, A],$ we have that
\[\|[X,Y]\|= \frac12 \|[A^\ast, A]\|.\]
One obtains that $X$ and $Y$ are commuting (resp. almost commuting) if and only if $A$ is normal (resp. almost normal) since if $X', Y'$ are self-adjoint with $A' = X'+
 iY'$ then
 \[\|A'-A\| \leq 2\max(\|X'-X\|, \|Y'-Y\|)\]
 and
 \[\|X'-X\| = \|\Re(A'-A)\| \leq \|A'-A\|,\]
 \[\|Y'-Y\| = \|\Im(A'-A)\| \leq \|A'-A\|.\]

We now state how the symmetries and antisymmetries of $X$ and $Y$ are related to those of $A$. A key observation is that $\varphi(A) = \varphi(X)+\varphi(iY)$ so that if $\varphi$ is linear then 
$\varphi(A) = \varphi(X)+i\varphi(Y)$ is the decomposition of $\varphi(A)$ into its real and imaginary parts. Likewise, if $\varphi$ is conjugate-linear then
$\varphi(A) = \varphi(X)-i\varphi(Y)$.
So, if $A$ has a real phase symmetry then $X$ and $iY$ will each have a real phase symmetry also as the proposition below states. 
Note that non-real phase symmetries do not translate to symmetries of $X$ and $Y$ individually. 
\begin{prop}\label{translation}
Let $\varphi$ be a symmetry map on a unital $C^\ast$-algebra $\A$. Let $A \in \A$ and $X = \Re(A), Y = \Im(A)$. Or equivalently, consider $X, Y\in \A$ self-adjoint with $A = X + iY$. Then the following hold: 
\begin{enumerate}[label=(\roman*)] 
\item Suppose that $\varphi$ is linear. Then $A$ is $\varphi$-symmetric if and only if $X$ and $Y$ are $\varphi$-symmetric.
\item Suppose that $\varphi$ is linear. Then $A$ is $\varphi$-antisymmetric if and only if $X$ and $Y$ are $\varphi$-antisymmetric.
\item Suppose that $\varphi$ is conjugate-linear. Then $A$ is $\varphi$-symmetric if and only if $X$ is $\varphi$-symmetric and $Y$ is $\varphi$-antisymmetric.
\item Suppose that $\varphi$ is conjugate-linear. Then $A$ is $\varphi$-antisymmetric if and only if $X$ is $\varphi$-antisymmetric and $Y$ is $\varphi$-symmetric.
\end{enumerate}
\end{prop}
\begin{remark}
Observe that because $X$ and $Y$ are self-adjoint $\varphi(X) = \varphi_\ast(X)$ and $\varphi(Y) = \varphi_\ast(Y)$. 
So, the statement that $X$ and $Y$ have symmetries or antisymmetries is unchanged when replacing $\phi$ with $\varphi=\phi_\ast$. 
So, if $X$ and $Y$ both are $\phi$-symmetric and $\phi$ is conjugate linear, then the way to apply Proposition \ref{translation} is to use the only case that has both $X$ and $Y$ $\varphi$-symmetric, which is case (i). So, because $\varphi$ needs to be linear in that case, we conclude that $A = X+iY$ is $\phi_\ast$-symmetric.

In this way, we see that in the case that both $X$ and $Y$ are both $\varphi$-symmetric or are both $\varphi$-antisymmetric then $A$ has a symmetry or antisymmetry with respect to a linear symmetry. Also, if one of $X$ and $Y$ is $\varphi$-symmetric and the other is $\varphi$-antisymmetric then $A$ has a symmetry or antisymmetry with respect to a conjugate-linear symmetry map.

Another perspective on this reduction is to use what seems to be the approach in Section 4.1 of \cite{loring2015k} (so called ``Class AIII in 1D''). We calculate $\varphi$ of $A$:
\[\varphi(A) = \varphi(X) - i\varphi(Y) = X - iY = A^\ast.\]
Because symmetry maps are hermitian, this of course is equivalent to $\varphi_\ast(A) = A$. Regardless of approach, we see that the natural perspective for both $X$ and $Y$ having a symmetry is to express it as $A$ having a linear symmetry. 
\end{remark}

\section{TR rank 1 and Symmetry Reductions}\label{Symmetry Reductions}

Loring and S{\o}rensen in \cite{loring2016almost} generalized the standard lifting argument from Friis and R{\o}rdam's proof of Lin's theorem to incorporate a symmetry.  In this section we begin with a definition which extends stable rank 1 and give some of its properties to show that it generalizes properly. Because we wish to use the extension of Lin's theorem in \cite{loring2016almost}, we also wish for it to extend their definition in a way that is compatible with our symmetries.

We begin with stating the symmetry-respecting version of stable rank one given in \cite{loring2016almost}.
\begin{defn}(\cite{loring2016almost})
Let $\A$ be a unital $C^\ast$-algebra and $\tau$ a linear anti-multiplicative symmetry map on $\A$ with order $2$. Let $\A_\tau$ be the set of $\tau$-symmetric elements of $\A$.

We say that $(\A, \tau)$ has TR rank 1 if the  invertible elements of $\A_\tau$ are dense in $\A_\tau$. 
\end{defn}
\begin{remark}
In \cite{loring2016almost}, $\tau$ is called a ``reflection'' on $\A$, as a generalization of the transpose. We choose not to use this terminology to avoid confusion with the fact that $\tau$ (as a result of being linear) preserves the spectral projections of any $\tau$-symmetric normal matrix $A$. 
So, the map $\tau: E_{\{z\}}(A) \mapsto \tau(E_{\{z\}}(A)) = E_{\{z\}}(A)$ is the identity. 
If $\phi$ is a conjugate-linear symmetry for a normal matrix $A$ then $\phi$ induces a reflection on $\C$ about the real axis since $\phi: E_{\{z\}}(A) \mapsto  E_{\{\overline{z}\}}(A)$.
\end{remark}

\begin{defn}
Let $\A$ be a unital $C^\ast$-algebra and ${S}$ be a collection of symmetry maps on $\A$. 

Let $\A_{S} = \{A\in \A: \varphi(A) = A \mbox{ for all }\varphi \in {S}\}$ be the set of ${S}$-symmetric elements of $\A$.
We say that $(\A, S)$ has TR rank 1 if the invertible elements of $\A_{S}$ are dense in $\A_{S}$. 
\end{defn}
We use the same terminology as \cite{loring2016almost} to indicate that we are really discussing the same definition but for a collection of symmetry maps of various types.

We note the following result whose statement and proof are generalizations of those of Lemma 6.1 of \cite{loring2016almost}.
\begin{thm}\label{finite TR rank 1}
Suppose that $\A$ is a finite dimensional $C^\ast$-algebra and ${S}$ is a collection of symmetry maps on $\A$. Then $(\A, S)$ has TR rank 1.
\end{thm}
\begin{proof}
Let $A\in\A_{S}$. Because $\sigma(A)$ is finite, the operator $A_n=A+\frac1nI$ is invertible for $n$ large. Because $\varphi(I)=I$ for all $\varphi\in {S}$, $A_n \in \A_{S}$. This is all we needed to show since $A_n \to A$ as $n \to \infty$.
\end{proof}

We first prove the following result which states that restricting our attention to a von Neumann subalgebra of a von Neumann algebra can be seen as  requiring certain symmetries be satisfied. 
\begin{prop}
Suppose that $\A$ is a von Neumann algebra.

For a collection of 
weakly continuous linear multiplicative symmetry maps ${S}$ on $\A$, the set $\A_{S}$ of fixed points of these symmetry maps is a von Neumann subalgebra of $\A$.

Conversely, for any von Neumann subalgebra $\B$ of $\A$ there is a collection of weakly continuous linear multiplicative symmetry maps $S_{\B}$ so that $\B = \A_{S_{\B}}$. 
\end{prop}
\begin{proof}
By Proposition \ref{properties}\ref{B subalg}, $\A_{S}$ is a unital $C^\ast$-subalgebra. Because each symmetry in ${S}$ is weakly continuous, $\A_{S}$ is weakly closed so it is also a von Neumann subalgebra.

Conversely, suppose that $\B$ is a von Neumann subalgebra of $\A$. Then by the von Neumann double commutant theorem, $\B = \B''$, where if $\mathcal R \subset \A$ is a set then $\mathcal R' = \{A \in A: [A, B] = 0 \mbox{ for all }B \in \mathcal R\}$ is a von Neumann subalgebra of $\A$. Any von Neumann subalgebra is generated by unitaries, so there is a collection of unitaries $\{U_\alpha\}_\alpha \subset \B'$ so that 
$\B = \{U_\alpha\}_\alpha'$. Now, define $\varphi_\alpha(A) = U_\alpha^{\ast}AU_\alpha$. Then 
 the $\varphi_\alpha$ are linear multiplicative symmetry maps and $\B = \A_{\{\varphi_\alpha\}_\alpha}$.
\end{proof}
Note that these symmetry maps may not all commute.
The assumption that the symmetry maps commute is essential when performing symmetrizations or performing constructions based on a particular symmetry (assuming that we want to retain other symmetries). 
If our constructions do not reference the symmetry maps, then we expect there to be no requirement that the symmetry maps commute in order for an element constructed from another object to retain the symmetries when we only multiply elements $A$ and $B$ in the form $ABA$ or $AB + BA$, perform certain applications of the functional calculus, or take weak closures.

Also, if we have a collection of symmetry maps ${S}$ on $\A$ and $\varphi$ is also a symmetry map on $\A$ then a sufficient condition for $\A_{S}$ to be an invariant set under $\varphi$ is for $\varphi$ to commute with each symmetry of ${S}$. In the following result, this is the only place where we use any commutativity of elements of ${S}$.

When ${S}$ contains anti-multiplicative symmetry maps, $\A_{S}$ will generally no longer be closed under multiplication (and hence will not be a subalgebra). 
If ${S}$ contains conjugate-linear symmetry maps then $\A_{S}\neq \{0\}$ is not even closed under multiplication by non-real constants.
We now make the definition:
\begin{defn}Let $\A$ be a unital $C^\ast$-algebra and ${S}$ a collection of symmetry maps on $\A$. We say that ${S}$ is an admissible collection of symmetry maps if for each type of symmetry map in ${S}$ that is not linear multiplicative, there is a symmetry map in ${S}$ of that type which commutes with every symmetry map in ${S}$.

Linear multiplicative symmetry maps are not required to commute with each other.
\end{defn}

Before continuing with an extension of Lin's theorem for linear symmetries, we prove a result showing that having many symmetry maps can be reduced to the case where there are at most two symmetry maps of different types.
\begin{prop}\label{S0 Reduction Prop}
Suppose that ${S}$ is an admissible collection of symmetry maps on a unital $C^\ast$-algebra $\A$. There exist collections $S_0$ and $S_1$ of symmetry maps on $\A$ with the following properties.
\begin{enumerate}[label=(\roman*)]
\item $S_1$ is a finite commutative set of at most 2 symmetry maps with different types, neither of which is linear multiplicative. Each symmetry in $S_1$ commutes with each symmetry of $S_0$. 
\item Each symmetry of $S_1$ either belongs to ${S}$ or is a composition of two symmetry maps of ${S}$ of different types that are not linear multiplicative.
\item $S_0$ consists of only linear multiplicative symmetry maps.
\item Any symmetry in $S_0$ belongs to ${S}$ or is composition of at most three symmetry maps of $S_1$.
\item The restriction of each symmetry in $S_1$ to $\A_{S_0}$ has order $2$.
\item $A \in \A$ is ${S}$-symmetric if and only if $A \in \A_{S_0}$ and $A$ is $S_1$-symmetric. 
\item $(\A, S)$ has TR rank 1 if and only if $(\A_{S_0}, S_1)$ has $TR$ rank 1.
\item \label{vNSymReduction} If $\A$ is a von Neumann algebra and ${S}$ consists of weakly continuous symmetry maps then $\A_{S_0}$ is also a von Neumann subalgebra of $\A$ and $S_1$ consists of weakly continuous symmetry maps.
\end{enumerate}
\end{prop}

\begin{remark}
In the case that $S = \emptyset$, this statement is entirely trivial because being $\emptyset$-symmetric means nothing and $\A_\emptyset = \A$.

In the case that ${S}$ only consists of linear multiplicative symmetry maps, there is an equivalence made between $(\A, S)$ and $\A_{S}$. The statement that $(\A_{S}, \emptyset)$ has TR rank 1 is just that $\A_{S}$ has stable rank one.

In the case that ${S}$ only consists of linear symmetry maps but some are anti-multiplicative, there is a linear anti-multiplicative symmetry map $\tau \in {S}$ so that there is an equivalence made between $(\A, S)$ and $(\A_{S_0}, \tau)$. Also, $\tau$ has order $2$ on $\A_{S_0}$. In fact, any such linear anti-multiplicative symmetry in ${S}$ that satisfies the admissibility condition will do.

In the case that ${S}$ only consists of multiplicative symmetry maps but not all are linear, there is a conjugate-linear multiplicative symmetry map $\phi \in {S}$ such that there is an equivalence between $(\A, S)$ and $(\A_{S_0}, \phi)$. Also, $\phi$ has order $2$ on $\A_{S_0}$.
In fact, any such conjugate-linear multiplicative symmetry in ${S}$ that satisfies the admissibility condition will do.

In the case that ${S}$ consists of linear multiplicative symmetry maps and symmetry maps of at least two other types, then there are two (commuting) symmetry maps $\tau, \phi \in {S}$ of different types other than linear multiplicative so that there is an equivalence made between
$(\A, S)$ and $(\A_{S_0}, \{\tau, \phi\})$. Also, $\tau$ and $\phi$ have order $2$ on $\A_{S_0}$.
We can actually choose $\tau$ and $\phi$ to be any two symmetry maps with different types other than linear multiplicative that satisfy the admissibility condition.

In all these cases the symmetry maps in $S_1$ commute with all the symmetry maps of $S_0 \cup S_1$ and the symmetry maps in $S_1$ have order $2$ on $\mathcal A_{S_0}$.
\end{remark}
\begin{remark}
Compare to Remark \ref{Group Isomorphism}, thinking of ${S}$, modulo type, as a subset of $\Z_2 \times \Z_2$. The construction in this proposition can then be thought of as first forming the group generated by ${S}$, modulo type, in $\Z_2\times \Z_2$. Then we let $S_1$ be representatives of the cosets $S_{(1, -1)}, S_{(-1, 1)}, S_{(-1,-1)}$ but we can omit one if two are present since two cosets can be used to get a representative of the other.
\end{remark}
\begin{proof}
Partition ${S}$ into disjoint subsets $S_{(1,1)}$, $S_{(1,-1)}$, $S_{(-1,1)}$, $S_{(-1,-1)}$ based on type: linear multiplicative, linear anti-multiplicative, conjugate-linear multiplicative, and conjugate-linear anti-multiplicative, respectively.

If $T$ is a set of symmetry maps and $\varphi_0$ is a symmetry, let $T\circ \varphi_0 = \{\varphi \circ \varphi_0 : \varphi \in T\}$.
We construct $S_0$ and $S_1$ in different cases that depend on which of these partitioned subsets are non-empty.

\begin{enumerate}
\item Suppose that $S = S_{(1,1)}$. Then let $S_0 = S$ and $S_1 = \emptyset$. (Note that ${S}$ can be empty.)
\item Suppose that $S = S_{(1,1)} \cup S_\alpha$, where $S_\alpha \neq \emptyset$ and the multi-index $\alpha$ equals one of $(1,-1)$, $(-1,1)$, $(-1,-1)$. 

Let $\varphi_\alpha$ be one element from $S_\alpha$ that commutes with every element of $S$.
Then $S_0 = S_{(1,1)}\cup S_\alpha \circ \varphi_\alpha$ and $S_1 = \{\varphi_\alpha\}$.
\item Suppose that $S = S_{(1,1)} \cup S_\alpha \cup S_\beta$, where $S_\alpha, S_\beta \neq \emptyset$ and the multi-indices $\alpha, \beta = (1,-1), (-1,1), (-1,-1)$ are not equal. 

Let $\varphi_\alpha, \varphi_\beta$ each be an element from $S_\alpha, S_\beta$, respectively, that commute with every element of $S$. We can choose $S_1$ to equal one of $\{\varphi_\alpha, \varphi_\beta\}, \{\varphi_\alpha\circ \varphi_\beta, \varphi_\beta\}, \{\varphi_\alpha, \varphi_\beta \circ\varphi_\alpha\}$.

Let $S_0 = S_{(1,1)}\cup (S_\alpha \circ \varphi_\alpha) \cup (S_\beta\circ\varphi_\beta)$.

\item Suppose that $S = S_{(1,1)} \cup S_{(1,-1)} \cup S_{(-1, 1)} \cup S_{(-1,-1)}$ and $S_{\alpha}\neq \emptyset$ for each $\alpha = (1,-1)$, $(-1,1)$, $(-1,-1)$. 

Let $\alpha, \beta$ be any two distinct multi-indices in $(1,-1)$, $(-1,1)$, $(-1,-1)$. Let $\gamma$ denote the other multi-index. Let $\varphi_\alpha, \varphi_\beta$ each be an element from $S_\alpha, S_\beta$, respectively, that commute with every element of $S$. Set $S_1 = \{\varphi_\alpha, \varphi_\beta\}$.

Let $S_0 = S_{(1,1)} \cup (S_\alpha \circ \varphi_\alpha) \cup (S_\beta \circ \varphi_\beta) \cup (S_\gamma \circ \varphi_\alpha\circ \varphi_\beta)$.
\end{enumerate}

So, we have defined $S_0$ and $S_1$ in each case. We now verify the statements of the Proposition.
\begin{enumerate}[label=(\roman*)]
\item We note that if $\varphi_\alpha$ and $\varphi_\beta$ are two symmetry maps of different types that are not linear multiplicative then $\varphi_\alpha \circ \varphi_\beta$ does not have the same type as either $\varphi_\alpha$ or $\varphi_\beta$. So, this guarantees the types of $S_1$. Because ${S}$ is admissible, the set $S_1$ is well-defined and contains symmetry maps that commute with every symmetry map in $S$ and in $S_0$.
\item This is clear from the construction.
\item This follows as a result of two facts. Note that the composition of two symmetry maps of the same type always has type: linear multiplicative. For Case 4, $\varphi_\alpha \circ \varphi_\beta$ has the same type as the elements of $S_\gamma$.
\item This is clear from the construction.
\item Notice that because all the symmetry maps in $S_1$ commute with the symmetry maps in ${S}$, each symmetry and product of symmetry maps of $S_1$ map elements of $\A_{S_0}$ into itself. So, the restriction of a symmetry in $S_1$ to $\A_{S_0}$ is well-defined.
Case 1 is trivial.

Consider Case 2. Because $\varphi_\alpha \in S_\alpha,$ it follows that $\varphi_\alpha^2 \in S_0$. So, for every $A \in \A_{S_0}$, $\varphi_\alpha^2(A) = A$. 

Consider Case 3. By the argument above, $\varphi_\alpha^2$ and $\varphi_\beta^2$ equal the identity on $\A_{S_0}$. So, $(\varphi_\alpha \circ \varphi_\beta)^2=\varphi_\alpha^2 \circ \varphi_\beta^2$ equals the identity on $\A_{S_0}$ as well.

Consider Case 4. As in Case 3, $\varphi_\alpha^2$ and $\varphi_\beta^2$ equal the identity on $\A_{S_0}$.

So, we have shown that this holds for each case.

\item
We now explain why the elements of $\A_{S_0}$ that are $S_1$-symmetric are exactly the elements of $\A$ that are ${S}$-symmetric. Clearly, any element $A \in \A$ that is ${S}$-symmetric will also belong to $\A_{S_0}$ and be $S_1$-symmetric. 

Conversely, suppose that $A \in \A_{S_0}$ is $S_1$-symmetric. Because $S_{(1,1)}\subset S_0$, we only need to show that $\varphi(A) = A$ for any symmetry $\varphi$ in $S \setminus S_{(1,1)}$. We go through each case. 

\begin{enumerate}[label=\arabic*.]
\item This case  is trivial.

\item  Since $\varphi\circ \varphi_\alpha \in S_0$, it follows that $\varphi(\varphi_\alpha(A)) = A$. However, since $A$ is $S_1$-symmetric, $\varphi_\alpha(A) = A$. So, $\varphi(A) = A$ as desired.

\item The key observation is that regardless of the choice of $S_1$, being $S_1$-symmetric is equivalent to being $\{\varphi_\alpha, \varphi_\beta\}$-symmetric. For instance, if $S_1 = \{\varphi_\alpha\circ \varphi_\beta, \varphi_\beta\}$ then because $A$ is $S_1$-symmetric, $\varphi_\beta(A) = A$ so 
\[\varphi_\alpha(A) = \varphi_\alpha(\varphi_\beta(A)) = (\varphi_\alpha\circ\varphi_\beta)(A)=A.\] 
The other case is similar.
So, it follows that $A$ is $\varphi_\alpha$-symmetric and $\varphi_\beta$-symmetric. 

Now, $\varphi$ must belong to $S_\alpha$ or $S_\beta$, so because $(S_\alpha \circ \varphi_\alpha) \cup (S_\beta \circ \varphi_\beta) \subset S_0,$ by the argument in Case 2, it follows that $A$ is $\varphi$-symmetric. This is what we wanted to show.

\item  Since $A$ is $S_1$-symmetric, it is $\{\varphi_\alpha, \varphi_\beta, \varphi_\alpha\circ\varphi_\beta\}$-symmetric. Now, $\varphi$ must belong to one of $S_\alpha$, $S_\beta$, $S_\gamma$ so because $(S_\alpha \circ \varphi_\alpha) \cup (S_\beta \circ \varphi_\beta) \cup (S_{\gamma}\circ \varphi_\alpha\circ \varphi_\beta)\subset S_0,$ by the argument in Case 2, it follows that $A$ is $\varphi$-symmetric. This is what we wanted to show.
\end{enumerate}

\item By (v) and Proposition\ref{properties}\ref{B subalg}, $\B=\A_{S_0}$ is a unital $C^\ast$-subalgebra. 

Now, suppose that $(\A, S)$ has TR rank 1. Let $B \in \B$ be $S_1$-symmetric. Then $B$ is ${S}$-symmetric so there is a sequence of invertible elements $A_n \in \A$ that are ${S}$-symmetric and are converging to $B$. But the $A_n$  belong to $\B$ and are $S_1$-symmetric. This shows that $B$ has TR rank 1.

Conversely, suppose $(\B, S_1)$ has TR rank 1. Let $A \in \A$ be ${S}$-symmetric. Then $A\in \B$ is $S_1$-symmetric, so there is a sequence of invertible $B_n \in \B$ that are $S_1$-symmetric converging to $A$. But the $B_n$ are ${S}$-symmetric. This shows that $(A,S)$ has TR rank 1.

\item $S_0, S_1$ consist of weakly continuous symmetry maps because the composition of finitely many weakly continuous maps is weakly continuous. Consequently, $\A_{S_0}$ is a von Neumann subalgebra of $\A$. 

This concludes the verification of the last part of the Proposition.
\end{enumerate}
\end{proof}

Note that a conjugate-linear phase symmetry is very similar to just a conjugate-linear symmetry in that they both induce a reflection on $\C$. For instance, suppose that $\phi$ is a conjugate-linear phase symmetry for a normal matrix $A$ with phase $\zeta$. Let $E_z = E_{\{z\}}(A)$. Then
\[\sum_z \overline{z} \phi(E_z)= \phi(\sum_zzE_z) = \zeta A = \sum_z \zeta z E_z.\]
So, \[\sum_z \overline{\zeta z} \phi(E_z) = \sum_z  z E_z = \sum_z  \overline{\zeta z} E_{\overline{\zeta z}}.\]
Thus, $\phi(E_z) = E_{\overline{\zeta z}}$. This corresponds to a rotation followed by a reflection. However that is the same as a reflection about a different angle since: $\overline{\zeta z}=\zeta^{-1/2}\overline{\zeta^{1/2} z}.$

In particular, if $B = \zeta^{1/2}A$ then 
\[\phi(B) = \phi(\zeta^{1/2}A) = \zeta^{-1/2}\phi(A)=\zeta^{1/2}A = B.\] So, by rotating we transformed a conjugate-linear phase symmetry into a conjugate-linear symmetry.

Another thing to note is that one does not always have a generalization of TR rank 1 for non-unital phase symmetries even in the finite dimensional case. For instance, consider the example of a linear antisymmetry for normal $n \times n$ matrix $A$. The non-zero eigenvalues must come in $x \leftrightarrow -x$ pairs with eigenspaces of the same dimension. So, if $n$ is odd, then necessarily $A$ has $0$ as an eigenvalue which is not removable due to the antisymmetry.  This implies that no small perturbation of $A$ can be invertible while maintaining the antisymmetry. 

Similar statements about a phase symmetry $\zeta$ can be made, but the relevant number is not $2$ (so ``even'' and ``odd'') but the order of $\zeta$. So, all this is to say that there are technicalities involved with such phase symmetries. In particular, one cannot translate an element $A$ to $A+cI$ while preserving phase symmetry.

\section{Lin's Theorem with Linear Symmetries}
\label{Lin's Theorem with Linear Symmetries}

In this section, we will prove Lin's theorem with linear symmetries. 

As noted in \cite{loring2015k}, Lin already proved his result with no symmetries. By the reduction in the last section, we will reduce the case of linear multiplicative symmetries to this.
We will reduce the case of linear symmetries where some of them are anti-multiplicative to the following result of Loring and S{\o}rensen. 

The Friis and R{\o}rdam proof of Lin's theorem shows:
\begin{thm}\label{Lin sr0}(\cite{friis1996almost})
Let $\A$ be a unital $C^\ast$-algebra so that $\A$ has stable rank $1$.

There exists a function $\epsilon_1(\delta)$ independent of $\A$ with $\lim_{\delta \to 0}\epsilon_1(\delta) =0$ 
that satisfies the following property: If $A \in \A$ is a contraction, then there is a normal $A'\in \A$ that satisfies
\[\|A'-A\| \leq \epsilon_1(\|[A^\ast,A]\|).\]
\end{thm}
\begin{remark}
\cite{friis1996almost} states this result for unital $C^\ast$-algebras satisfying what they call property (IR), but after Definition 3.1 it is stated that having stable rank implies (IR). So, we state this result in terms of what we will find useful.
\end{remark}

There also is this extension by Kachkovskiy and Safarov, which is a corollary of their main theorem of \cite{kachkovskiy2016distance} and its Remark 1.2.
\begin{thm}\label{Lin KS}(\cite{kachkovskiy2016distance})
Let $\A$ be a von Neumann algebra with stable rank $1$.

There exists a constant $C_{KS}>0$ independent of $\A$ with the following property: If $A \in \A$, then there is a normal $A'\in \A$ that satisfies
\[\|A'-A\| \leq C_{KS}\|[A^\ast,A]\|^{1/2}.\]
\end{thm}

Loring and S{\o}rensen extended Lin's theorem when there is a single linear anti-multiplicative symmetry of order 2:
\begin{thm}\label{Lin reflection}(\cite{loring2016almost})
Let $\A$ be a unital $C^\ast$-algebra and $\tau$ a linear anti-multiplicative symmetry map on $\A$ with order $2$ so that $(\A, \tau)$ has TR rank $1$.

There exists a function $\epsilon_2(\delta)$ independent of $\A$ and $\tau$ with $\lim_{\delta \to 0}\epsilon_2(\delta) =0$  
that satisfies the following property: If $A \in \A$ is a $\tau$-symmetric contraction, then there is a normal $\tau$-symmetric $A'\in \A$ that satisfies
\[\|A'-A\| \leq \epsilon_2(\|[A^\ast,A]\|).\]
\end{thm}

We extend this to the following.
\begin{thm}\label{Linear Lin}
Suppose that $\A$ is a unital $C^\ast$-algebra and ${S}$ is an admissible collection of linear symmetry maps on $\A$  so that $(\A,S)$ has TR rank 1. 

Let $\epsilon_\ell(\delta) = \max(\epsilon_1(\delta), \epsilon_2(\delta))$, where $\epsilon_1$ and  $\epsilon_2$ are functions that satisfy the statement of Theorem \ref{Lin sr0} and Theorem \ref{Lin reflection}, respectively. 
We can assume that they are both monotonically increasing by the argument in Theorem 2.12 of \cite{filonov2011relation}.

Then if $A \in \A$ is an ${S}$-symmetric contraction, there is a normal ${S}$-symmetric $A' \in \A$ that satisfies
\[\|A'-A\| \leq \epsilon_\ell(\|[A^\ast,A]\|).\]
\end{thm}
\begin{proof}
Suppose that all the $\varphi\in {S}$ are linear multiplicative symmetries. Then by Proposition \ref{S0 Reduction Prop}, $\A_{S} = \{A\in \A: \varphi(A)=A\mbox{ for all }\varphi\in {S}\}$ is a unital $C^\ast$-subalgebra of $\A$ that has stable rank one. 

Let $A\in\A$ be an ${S}$-symmetric contraction. It then belongs to $\A_{S}$. So, by Theorem \ref{Lin sr0} for $\A_{S}$, there is a normal $A'\in \B$ with
\[\|A'-A\| \leq \epsilon_1(\|[A^\ast,A]\|).\]
Then $A'$ is also ${S}$-symmetric.

\vspace{0.05in}

Suppose now that some $\tau \in {S}$ is anti-multiplicative. Let $S_0$ and $S_1=\{\tau\}$ be as in Proposition \ref{S0 Reduction Prop}. Then $(\A_{S_0}, \tau)$ is a unital $C^\ast$-subalgebra of $\A$ that has TR rank 1 and $\tau^2=\id$ on $\A_{S_0}$. 

Let $A\in\A$ be an ${S}$-symmetric contraction. It then belongs to $\A_{S_0}$ and is $\tau$-symmetric. So, by Theorem \ref{Lin reflection} for $(\A_{S_0}, \tau)$ there exists a normal $A' \in \A_{S_0}$ that is $\tau$-symmetric with
\[\|A'-A\| \leq \epsilon_2(\|[A^\ast,A]\|).\]
Then $A'$ is also ${S}$-symmetric.
\end{proof}
\begin{remark}
By Proposition \ref{translation}(i), we see that this also provides a result for almost commuting self-adjoint elements of a finite dimensional $C^\ast$-algebra with linear symmetries.
\end{remark}

Using Theorem \ref{Lin KS} we also have
\begin{thm}\label{Sym Lin KS}
Let $\A$ be a von Neumann algebra and ${S}$ a collection of weakly continuous linear multiplicative symmetries on $\A$. Suppose that $(\A, S)$ has TR rank 1.

Let $C_{KS}>0$ be the universal constant in Theorem \ref{Lin KS}. If $A \in \A$ is ${S}$-symmetric, then there is a normal $A'\in \A$ that is ${S}$-symmetric and satisfies
\[\|A'-A\| \leq C_{KS}\|[A^\ast,A]\|^{1/2}.\]
\end{thm}
\begin{proof}
By Proposition \ref{S0 Reduction Prop}\ref{vNSymReduction}, $\A_{S}$ is a von Neumann algebra with stable rank 1. So, the result follows from Theorem \ref{Lin KS}.
\end{proof}

\section{Spectral Projection Inequalities and Localization Operators}
\label{Spectral Projection Inequalities and Localization Operators}

In this section we review some results for projections that will be useful in the next section. We also develop symmetry-respecting localization operators which are an important technical part of forming nearby commuting operators from projections.

Before beginning our results for this section, we state the Davis-Khan theorem (\cite{davis1970rotation}) for perturbation of spectral projections as stated and proved in \cite{bhatia1983perturbation} and nicely summarized in Section 11 of \cite{bhatia1997and}. Naturally, this result applies when $A, B$ are in any von Neumann algebra.
\begin{thm}\label{DK theorem}
(\cite{bhatia1983perturbation})
Suppose that $A, B$ are normal elements of $B(\H)$, where $\H$ is some Hilbert space. Let $K_A, K_B \subset \C$ be sets with distance $\delta$. Let $E_A = E_{K_A}(A), E_B = E_{K_B}(B)$. Then there is a universal constant $c > 0$ so that
\[\|E_AE_B\| \leq \frac{c}\delta\|A-B\|.\]
If $K_A$ and $K_B$ are separated by a strip of width $\delta$, then we can use $c = 1$. In particular, this holds if $K_A, K_B$ are intervals in $\R$ separated by distance $\delta$.
\end{thm}
We now give the following useful  consequence:
\begin{lemma}\label{three projs lemma}
Suppose that $A, B$ are normal elements of $B(\H)$, where $\H$ is some Hilbert space. Let $K_1, K_2, K_3 \subset \C$ be such that $K_1$ and $K_3$ have distance $\delta$. Then for the universal constant $c > 0$ as in Theorem \ref{DK theorem},  
\[\|E_{K_1}(A)E_{K_2}(B)E_{K_3}(A)\| \leq \frac{4c}\delta\|A-B\|.\]
If $K_1, K_2, K_3$ are intervals in $\R$ so that $\inf K_2\in K_1$ and $\sup K_2 \in K_3$ then $c$ can be chosen to be $1$.
\end{lemma}
\begin{proof}
Let $K_{2,1} = \{z \in K_2: d(z, K_1) < \delta/2\}$ and $K_{2,3} = \{z \in K_2: d(z, K_1) \geq \delta/2\}$ be a decomposition of $K_2$ into a piece that is close to $K_1$ and a piece not close to $K_1$. Note that $d(K_1, K_{2,3})\geq \delta/2$ and $d(K_1, K_{2,1}) \leq \delta/2$ so that $d(K_3, K_{2,1}) \geq \delta/2$.

So, by Theorem \ref{DK theorem} applied to $K_{2,1}, K_3$ and $K_{1}, K_{2,3}$,
\begin{align*}
\|E_{K_1}(A)&E_{K_2}(B)E_{K_3}(A)\| \\
&\leq  \|E_{K_1}(A)E_{K_{2,1}}(B)E_{K_3}(A)\| + \|E_{K_1}(A)E_{K_{2,3}}(B)E_{K_3}(A)\|\\
&\leq\|E_{K_{2,1}}(B)E_{K_3}(A)\| + \|E_{K_1}(A)E_{K_{2,3}}(B)\|\\
&\leq \frac{2c}{\delta/2}\|A-B\|.
\end{align*}

Suppose that $K_1, K_2, K_3$ are intervals as specified.
then  $\overline{K_1} = [a,b]$, $\overline{K_3} = [b+\delta,c]$  but $\overline{K_2}=[\alpha, \beta]$ where $\alpha \in K_1$ and $\beta \in K_3$. Then $\overline{K_{2,1}}=[\alpha,b+\delta/2]$ and $\overline{K_{2,3}}=[b+\delta/2,\beta]$. So, $K_1, K_{2,3}$ and $K_3, K_{2,1}$ are separated by a strip of width $\delta/2$ so we can use $c=1$.
\end{proof}

We state the standard lemma concerning perturbing an almost projection to a projection, such as in Lemma 8.3.7 in \cite{strung2021introduction}. We stress the way that the projection is gotten by how we state this result.
\begin{lemma}\label{Strung lemma}
(\cite{strung2021introduction})
Let $\A$ be a unital $C^\ast$-algebra and let $X \in \A$ be self-adjoint such that $\|X\| \geq 1/2$ and $\|X^2 - X\| < 1/4$. Then there is a projection $F$ in $\A$, expressed as a real function $f$ of $X$ with $f(0)=0, f(1) = 1$, such that $\|F-X\| \leq 2\|X^2-X\|$.
\end{lemma}

We also state the following projection perturbation lemma from \cite{davidson1985almost}. Because we wish to make use of symmetries, we show that the construction respects symmetries.
\begin{lemma}\label{almost proj}
Let $\A$ be a unital $C^\ast$-algebra and let ${S}$ be a collection of symmetry maps on $\A$. Suppose that $E, F', G$ are projections in $\A$ that are ${S}$-symmetric. Suppose that $E \leq G$ and $\|(1-F')E\|, \|(1-G)F'\| \leq \varepsilon$ so that the range of $E$ is approximately contained in the range of $F'$ and the range of $F'$ is approximately contained in the range of $G$.

Then there exists a projection $F \in \A$ that is ${S}$-symmetric so that $E \leq F \leq G$ and $\|F'-F\| \leq 5\varepsilon$.
\end{lemma}
\begin{proof}
We outline the construction in the proof of \cite{davidson1985almost} to show that it respects symmetries. The other statements are shown by Davidson. Note that because $E, F', G$ are projections in $\A$, we can view them as projections on some Hilbert space $\H$. 

Davidson assumes that $\varepsilon < 1/5$ then writes $F'$ as a block matrix $(F_{ij})$ with respect to $E\H \oplus (G-E)\H \oplus (1-G)\H$. 
Then he notes that the self-adjoint operator $F_{22}$ is almost a projection with spectrum in $[0, 4\varepsilon^2]\cup [1-4\varepsilon^2, 1]$. 
Then it is stated that there is a projection $P$ on $(G-E)\H$ such that $\|F_{22} - P\| < 4\varepsilon^2<\varepsilon.$ 

Comparing the estimate in the proof of Lemma \ref{Strung lemma}, one see that $P$ can be gotten as a real, continuous function $f$ of $F_{22}$. Then Davidson sets $F = E + P$.

We note that if $E, F', G$ are ${S}$-symmetric then so is $\tilde F_{22} = (G-E)F'(G-E)\in\A$ and hence $P = f(\tilde F_{22})$ as well. So, $F=E+P$ is ${S}$-symmetric.
\end{proof}

\vspace{0.05in}

We now proceed to discussing localization operators.
We utilize some of the notation from \cite{hastings2009making}. The idea for constructing nearby commuting operators from $A$ and $B$ given projections $F_k$ is to first ``localize'' $B$ with respect to the spectrum of $A$ then to project $B$ onto the certain projections by ``pinching''. This is the map $B \mapsto B''$. 

We first begin with the following localization lemma from \cite{kachkovskiy2016distance}. How we will use it will be similar to the localization lemmas of Section 5 of  \cite{herrera2020hastings}. In particular, we use that same method of localization for a normal element as in Corollary 5.4 of \cite{herrera2020hastings}.

\begin{lemma}\label{KS localization}(\cite{kachkovskiy2016distance}) Let $\rho\geq 0$ be a $C^2(\R)$ function supported in $[-1,1]$ such that if  $\rho_n(x) = \rho(x-n)$ then $\sum_n \rho_n^2 = 1.$ Let $X, Y \in B(\H)$ be self-adjoint operators on a Hilbert space $\H$. 

Define $Y' = \sum_n \rho_n(X)Y\rho_n(X)$. Then $Y'$ is self-adjoint and satisfies the following properties:
\begin{enumerate}[label=(\roman*)] 
\item $\|EY'E\| \leq \|EYE\|$ for any spectral projection $E$ of $X$.
\item $\|[X, Y']\| \leq \|[X,Y]\|$
\item $\|Y'-Y\|\leq C_{\rho}\|[X,Y]\|$, where $C_\rho$ is a constant only depending on $\rho$. 
\end{enumerate}
\end{lemma}

We will provide an estimate for $C_\rho$ for $\rho$ even. Note that the proof only uses that $\rho$ is smooth enough that it is continuous and 
\[\|Y'-Y\|\leq\|\sum_n[\rho_n(X),Y]\rho_n(X)\| \leq C_\rho \|[X,Y]\|.\]
As we will see, this inequality does not require $Y$ to be self-adjoint.
Inspired by the method in \cite{kachkovskiy2016distance}, let \[\rho_n(x) =\frac1{\sqrt{2\pi}}\int\widehat{\rho}(k) e^{ik(x-n)}dk.\] Then 
\begin{align*}
\sum_n[\rho_n(X),Y]\rho_n(X) = \frac1{\sqrt{2\pi}}\int\widehat{\rho}(k)[e^{ikX}, Y]\left(\sum_n\rho_n(X)e^{-ink}\right)dk
\end{align*}
Recall that for each $x \in \R$, there are at most two values of $n$ so that $\rho_n(x)> 0$ and the sum of their squares equals 1. Now, recall that if $a, b \in \C$ with $|a|^2+|b|^2 = 1$ then
\[|a+b|\leq |a|+|b|\leq \sqrt{2(|a|^2+|b|^2)}=\sqrt2.\]
So, applying this to $a, b$ being the non-zero $\rho_n(x)e^{-ink}$, we obtain a bound for the norm of the sum in the parenthesis of $\sqrt2$.
Also, by Example 5 in Section 1.1 and Theorem 1.3.1 of \cite{aleksandrov2016operator}, we have \[\|[e^{ikX}, Y]\|\leq |k|\|[X,Y]\|,\]
which does not require $Y$ to be self-adjoint.
So, altogether
\[\|\sum_n[\rho_n(X),Y]\rho_n(X)\| \leq \frac{\sqrt2}{\sqrt{2\pi}}\int|k\widehat{\rho}(k)|dk\cdot\|[X,Y]\| \leq \sqrt{2}C_1\,\|[X,Y]\|.\]

We now choose $\rho$ to be piecewise $C^2$ and $\rho'' \in L^{\infty}$. By the proof of Corollary 2 in Section 3.2.23 of \cite{bratteli1979operator}, 
\[C_1 \leq \left(\frac\pi2\int|\rho''(x)-\rho'(x)|^2dx\right)^{1/2}.\]
Note that the statement of the cited result requires $\rho\in C^2$, however our conditions are sufficient so that we can apply Fourier transform identities to the Fourier transform of $\rho''$, which is what is used in its proof.

So, we construct $\rho$ as follows. Let \[p(x) = 1 - 35x^4 + 84x^5 - 70x^6 + 20x^7 = (x-1)^4(1 + 4x + 10x^2 + 20x^3).\] Then $p$ is a polynomial with $p(0) = 1, p(1)=0$ and $p^{(k)}(0)=0$, $p^{(k)}(1)=0$ for $1 \leq k \leq 3$. Also, $0 \leq p\leq 1$ on $[0,1]$, the only zero of $p(x)$ on $[0, \infty)$ is at $x=1$, and 
\begin{equation}\label{p-identity}
p(x) = 1-p(-x+1).
\end{equation}
Let 
\begin{align}\label{rho def}
\rho(x) = \left\{\begin{array}{ll} \sqrt{p(x)}, & 0 \leq x \leq 1\\
\sqrt{1-p(x+1)}, & -1 \leq x \leq 0\\
0, & |x|>1 \\\end{array}\right.
\end{align}
Then $\rho$ is supported on $[-1,1]$, has values in $[0,1]$, is piecewise $C^2$ with $\rho'' \in L^\infty$, and satisfies $\rho^2(x-1)+\rho^2(x)=1$ on $[0,1]$. So, $\sum_n\rho^2(x-n) = 1$ for all $x\in \R$.
Note that $\rho$ is even because if $|x|>1$, $\rho(x)=0$ and by (\ref{p-identity}) if for $x\in(-1,0]$, \[\rho(-x) =\sqrt{p(-x)}= \sqrt{1-p(x+1)} = \rho(x).\]

Note that 
\[\sqrt{p(x)} = (x-1)^2\sqrt{1 + 4x + 10x^2 + 20x^3},\] so that $\sqrt{p(x)}$ is smooth in a neighborhood of $[0,1]$ because the radicand is at least $1$ on $[0, \infty)$. 
Since $\sqrt{p(1)} = 0$ and $\sqrt{p(x)}'|_{x=1} = 0$, $\rho$ is $C^1$ on $(0, \infty)$.

Because $p(x)=1-35x^4+\cdots,$ we have that $\sqrt{p(x)}$ is smooth in a neighborhood of the origin and has first and second derivatives at the origin equal to zero.
Hence, by evenness,  $\rho'' \in L^\infty$ and $\rho''$ is piecewise smooth.

By a numerical calculation 
, 
we get $C_1< 8.314$ so 
\begin{align}\label{C_rho}
C_\rho < 11.758.
\end{align}

We will need the following version/application of that lemma:
\begin{lemma}\label{localization lemma} Let $\rho\geq 0$ be the function  defined in Equation (\ref{rho def}). Let $\A$ be a unital $C^\ast$-algebra and $A, B \in \A$ with $A$ self-adjoint. Let $\Delta > 0$.

Define $\L_A^\Delta(B) = \sum_n \rho_n\left(\frac2\Delta A\right)B\rho_n\left(\frac2\Delta A\right)$. Then $\L_A^\Delta$ is a completely positive map with norm $1$ that satisfies the following properties:
\begin{enumerate}[label=(\roman*)] 
\item $\|\L_A^\Delta(B)\| \leq \|B\|, \L_A(I) = I$
\item For any $C \in \A$ that commutes with $A$, $[C, \L_A^\Delta(B)] = \L_A^\Delta([C,B])$. Consequently, $\|[C,\L_A^\Delta(B)]\| \leq \|[C,B]\|$.
\item $\|\L_A^\Delta(B)-B\|\leq \frac{2C_{\rho}}{\Delta}\|[A,B]\|$, where $C_\rho<12$. 
\item\label{localization} If $\Omega_1, \Omega_2 \subset \R$ are sets such that $\dist(\Omega_1, \Omega_2) \geq \Delta$ then \[E_{\Omega_2}(A)\L_A^\Delta(B)E_{\Omega_1}(A) = 0.\]
\item If $X_1$ and $X_2$ are commuting self-adjoint elements of $\A$ and $\Delta_1, \Delta_2 > 0$ then $\L_{X_1}^{\Delta_1}$ and $\L_{X_2}^{\Delta_2}$ commute.
\item\label{localiation symmetric property} If $\varphi$ is a symmetry map on $\A$ then $\varphi(\L_A^\Delta(B)) = \L_{\varphi(A)}^\Delta(\varphi(B))$. Also, $\L_{-A}^\Delta = \L_A^\Delta$.
\end{enumerate}
\end{lemma}
\begin{proof}
Because we can embed a $W^\ast$-algebra generated by $\A$ into some $B(\H)$ in an isometric and weakly continuous way, this result then applies to its natural setting of $A, B$ belonging to $\A$ since $\L_A(B)\in \A$ and (iv) makes sense when we view the the spectral projections of $A$ belong to the extension of $\A$. We also note again that (iii) of  Lemma \ref{KS localization} does not require that $Y$ be self-adjoint.

Applying Lemma \ref{KS localization} to $X=\frac2{\Delta}A$ and $Y=B$ gives
\[\L_A^\Delta(B) = \sum_n \rho_n\left(\frac2\Delta A\right)B\rho_n\left(\frac2\Delta A\right)= \sum_n \rho_n(X)Y\rho_n(X) =\L_X^2(Y)\in\A\] satisfying the following properties.
Note that because all operators are bounded, the sum defining $\L$ is always a finite sum.
\begin{enumerate}[label=(\roman*)] 
\item $\L_A(I) = I$ follows directly because $\sum_n \rho_n^2 = 1$. The norm being $1$ then follows from Theorem 1.3.3 of \cite{stormer2013positive} from noting that $\L_A^\Delta$ is positive so $\|\L_A\|=\|\L_A(I)\|.$
\item The commutator identity is similar to the proof of (ii) in \cite{kachkovskiy2016distance}, which is just a simple computation. The commutator inequality then uses (i).
\item $\|\L_A^\Delta(B)-B\|= \|\L_X^2(Y)-Y\|\leq C_{\rho}\|[X,Y]\| = \frac{2C_\rho}{\Delta}\|[A,B]\|$, where $C_\rho<12$  by (\ref{C_rho}).
\item The function $r_n(x)=\rho_n\left(\frac2\Delta x\right)$ is non-zero only in the open interval $((n-1)\Delta/2, (n+1)\Delta/2)$. This is an open interval of length $\Delta$. Because there is a distance of at least $\Delta$ between $\Omega_1$ and $\Omega_2$, it follows that $\chi_{\Omega_2}r_n$ and $\chi_{\Omega_1}r_n$ are never simultaneously non-zero. So,
\[E_{\Omega_2}(A)\L_A^\Delta(B)E_{\Omega_1}(A) = \sum_n (\chi_{\Omega_2}r_n)(A)B(\chi_{\Omega_1}r_n)(A)=0.\]
\item Because the sums defining $\L_{X_1}^{\Delta_1}$ and $\L_{X_2}^{\Delta_2}$ are finite, we easily see that these operators commute because $\rho_n\left(\frac2{\Delta_i}X_i\right)$ commute.
\item By Proposition \ref{properties}\ref{triple product} and Proposition \ref{properties}\ref{isometry},
\begin{align*}
\varphi(\L_A^\Delta(B)) &= \sum_n\varphi(r_n(A))\varphi(B)\varphi(r_n(A))= \sum_nr_n(\varphi(Y))\varphi(X) r_n(\varphi(Y)) \\
&= \L_{\varphi(A)}^\Delta(\varphi(B)).
\end{align*}

We now show that $\L_{-A}^\Delta = \L_{A}^\Delta$. We calculate
\[r_n(-A) = \rho\left(-\frac2\Delta A-n\right) = \rho\left(\frac2\Delta A+n\right) = r_{-n}(A).\]
Thus, by re-indexing, \[\L_{-A}^\Delta(B) = \sum_nr_{n}(-A)B r_{n}(-A) = \sum_nr_{-n}(A)B r_{-n}(A) = \L_{A}^\Delta(B)\]
as desired.

\end{enumerate}
\end{proof}

We will be directly using the following. The statement \ref{localization} is the localization property.
\begin{lemma}\label{normal localization lemma} Let $\A$ be a unital $C^\ast$-algebra and $A, B \in \A$ with $A$ normal. Let $\Delta > 0$ and $\L^\Delta$ as in Lemma \ref{localization lemma}.

Define $\L(B) = \L_{\Re(A)}^{\Delta/\sqrt2}\circ \L_{\Im(A)}^{\Delta/\sqrt2}(B)$. Then $\L$ is a completely positive map with norm $1$ that satisfies the following properties:
\begin{enumerate}[label=(\roman*)]
\item $\|\L(B)\| \leq \|B\|$, $\L(I) = I$.
\item For any $C\in \A$ that commutes with $A$, $\|[C,\L(B)]\| \leq \|[C,B]\|$.
\item $\|\L(B)-B\|\leq\frac{2\sqrt2C_{\rho}}{\Delta}(\|[A,B]\|+\|[A^\ast,B]\|)$, where $C_\rho<12$. 
\item If $\Omega_1, \Omega_2 \subset \C$ are sets such that $\dist(\Omega_1, \Omega_2) > \Delta$ then \[E_{\Omega_2}(A)\L(B)E_{\Omega_1}(A) = 0.\]
\item\label{localization symmetry} If $\varphi$ is a symmetry map on $\A$ so that $A$ is $\varphi$-symmetric or $\varphi$-antisymmetric then $\varphi(\L(B)) = \L(\varphi(B))$. 
\end{enumerate}
\end{lemma}
\begin{proof}
We abbreviate $\L_{\Re(A)}^{\Delta/\sqrt2}$ as $\L_{\Re}$ and $\L_{\Im(A)}^{\Delta/\sqrt2}$ as $\L_{\Im}$ when there is no need to recall the localization distance $\Delta/\sqrt 2$.
Recall that $\L_{\Re}$ and $\L_{\Im}$ commute because $A$ is normal. 
\begin{enumerate}[label=(\roman*)]
\item This follows directly from the same properties for $\L_{\Re}$ and $\L_{\Im}$. 
\item By the Fuglede-Putnam theorem, $[C,A] = 0$ implies that $[C,A^\ast]=0$. So, $[C, -]$ commutes with $\L_{\Re}$ and $\L_{\Im}$. So,
\[\|[C,\L(B)]\| = \|\L([C,B])\|\leq \|[C,B]\|.\]
\item 
\begin{align*}
\|\L(B)-B\|&\leq \|\L_{\Re(A)}^{\Delta/\sqrt2}(\L_{\Im}(B)) - \L_{\Im}(B)\| + \|\L_{\Im(A)}^{\Delta/\sqrt2}(B) - B\|\\
&\leq \frac{2 C_{\rho}}{\Delta/\sqrt2}\|[\Re(A), \L_{\Im}(B)]\|+\frac{2 C_{\rho}}{\Delta/\sqrt2}\|[\Im(A), B]\|\\
&\leq \frac{2\sqrt2C_{\rho}}{\Delta}(\|[A,B]\|+\|[A^\ast,B]\|),
\end{align*} 
\item Let $z_1 \in \Omega_1, z_2 \in \Omega_2$.
It suffices to show that $E_{R_1}(A)\L(B)E_{R_2}(A) = 0$ for rectangles $R_1, R_2$ with small enough diameters containing $z_1, z_2$, respectively, in their interior.

Because $|\Re(z_1)-\Re(z_2)|^2+|\Im(z_1)-\Im(z_2)|^2 > \Delta^2$, we see that \[\max(|\Re(z_1)-\Re(z_2)|^2,|\Im(z_1)-\Im(z_2)|^2) > \Delta^2/2.\] 
Suppose that $|\Re(z_1)-\Re(z_2)| > \Delta/\sqrt2$. 
Consider the projections $\Re(R_1), \Re(R_2)$ of $R_1, R_2$ onto the real axis $\R$. For $\diam(R_1), \diam(R_2)$ small enough, $\Re(R_1)$ and $\Re(R_2)$ are a distance at least $\Delta/\sqrt2$ away. So,
\begin{align*}
&E_{R_1}(A)\L(B)E_{R_2}(A) \\
&=E_{R_1}(A)\left(E_{\Re(R_1)}(\Re(A))\L_{\Re(A)}^{\Delta/\sqrt2}(\L_{\Im}(B))E_{\Re(R_2)}(\Re(A))\right)E_{R_2}(A) \\
&= 0
\end{align*}
by (iv) of Lemma \ref{localization lemma}.
By a similar calculation, $E_{R_1}(A)\L(B)E_{R_2}(A) =0$ in the case that $|\Im(z_1)-\Im(z_2)| > \Delta/\sqrt2$.
This shows that $E_{\Omega_2}(A)\L(B)E_{\Omega_1}(A)=0.$

\item The cases of symmetry or antisymmetry of $A$ with respect to $\varphi$ break down into cases of some combination of $\Re(A)$ and $\Im(A)$ being $\varphi$-symmetry or $\varphi$-antisymmetric by Proposition \ref{translation}. 

So, because $\L_{-X}^{\Delta/\sqrt2}=\L_{X}^{\Delta/\sqrt2}$ for any self-adjoint $X$, it follows that
\[\varphi(\L_{\Im(A)}\circ \L_{\Re(A)}(B)) = \L_{\pm \Im(A)}\circ \L_{\pm \Re(A)}(\varphi(B)) = \L(\varphi(B)).\]

\end{enumerate}
\end{proof}

So, we have a localization operator that respects symmetry and antisymmetries. However, it will not respect non-real phase symmetries in general. Suppose that $\zeta = \cos\theta + i\sin\theta$ and $\varphi(A) = \zeta A$ for some linear symmetry $\varphi$. Then \[\varphi(A) = (\cos\theta 
\Re(A)-\sin\theta \Im(A))+i(\sin\theta \Re(A) +\cos\theta \Im(A)).\] Hence,
\[\varphi(\Re(A)) = \Re(\varphi(A)) = \cos\theta 
\Re(A)-\sin\theta \Im(A)\]
and because $\varphi$ is linear,
\[\varphi(\Im(A)) = \Im(\varphi(A)) = \sin\theta \Re(A) +\cos\theta \Im(A).\]
So, we can express this using a rotation matrix:
\[\begin{pmatrix} \varphi(\Re(A)) \\ \varphi(\Im(A))\end{pmatrix} =\begin{pmatrix} (\cos\theta) I & (-\sin\theta) I \\ (\sin\theta) I & (\cos\theta) I \end{pmatrix}\begin{pmatrix} \Re(A) \\ \Im(A)\end{pmatrix}.\]

Suppose that $\zeta$ has order $n\neq 2.$ Then the orbit of $A$ under $\varphi$ is $A$, $\zeta A$, $\zeta^2 A$, $\dots$, $\zeta^{n-1}A$ and the real and imaginary parts of $\zeta^k A$ are gotten by using the rotation expression above.

With these remarks, we have the following localization lemma that respects the phase symmetries of $A$.
\begin{lemma}\label{phase normal localization lemma} Let $\A$ be a unital $C^\ast$-algebra and $A, B \in \A$ with $A$ normal. Let $\Delta > 0$ and $\L^\Delta$ as in Lemma \ref{localization lemma}. Let $n \in \N$ and define $\omega_n = e^{2\pi i/n}$. 

Define $\Re_k = \Re(\omega_n^k A), \Im_k = \Im(\omega_n^k A)$, indexed cyclically in $\{0, 1, \dots, n-1\}$. 
Define $\L_n$ as the composition: $\L_n = \Pi_{k=0}^{n-1}\L_{\Re_k}^{\Delta/\sqrt2}\circ\L_{\Im_k}^{\Delta/\sqrt2}$. Then $\L_n$ is a completely positive map with norm $1$ that satisfies the following properties:
\begin{enumerate}[label=(\roman*)]
\item $\|\L_n(B)\| \leq \|B\|$, $\L_n(I) = I$.
\item For any $C\in \A$ that commutes with $A$, $\|[C,\L_n(B)]\| \leq \|[C,B]\|$.
\item $\|\L_n(B)-B\|\leq \frac{2n\sqrt2C_{\rho}}{\Delta}(\|[A,B]\|+\|[A^\ast,B]\|)$, where $C_\rho<12$. 
\item If $\Omega_1, \Omega_2 \subset \C$ are sets such that $\dist(\Omega_1, \Omega_2) > \Delta$ then \[E_{\Omega_2}(A)\L_n(B)E_{\Omega_1}(A) = 0.\]
\item If $\varphi$ is a symmetry map on $\A$ so that $A$ has a phase symmetry with phase $\zeta$ so that the order of $\zeta$ divides $n$ then $\varphi(\L_n(B)) = \L_n(\varphi(B))$. 
\end{enumerate}
\end{lemma}
\begin{proof}
Note that because $A$ is normal, $\Re_k, \Im_k$ all commute. So, (i), (ii), and (iii) follow just as in Lemma \ref{normal localization lemma}. 

By Lemma \ref{localization lemma}\ref{localiation symmetric property} for $\L$ defined in that lemma, $\L_n = \L \circ \Pi_{k=1}^{n-1}\L_{\Re_k}^{\Delta/\sqrt2}$ so the localization property (iv) holds for $\L_n$ as well.

We now address (v). Suppose that $\varphi(A) = \zeta A$ and $\zeta^n = 1$. There is an $s\in \N$ so that $\zeta = \omega_n^s.$ Suppose that $\varphi$ is linear. Then
\[\varphi(\Re_k) = \varphi(\Re(\omega^k A))= \Re(\omega^k\varphi(A)) = \Re_{k+s}\]
\[\varphi(\Im_k) = \varphi(\Im(\omega^k A))= \Im(\omega^k\varphi(A)) = \Im_{k+s},\]
hence $\varphi\circ\L_{\Re_k}\circ \L_{\Im_k}=\L_{\Re_{k+s}}\circ \L_{\Im_{k+s}}\circ\varphi$. So, when moving $\varphi$ past all the terms $\L_{\Re_k}\circ \L_{\Im_k}$ in $\varphi(\L_n(B)),$ we see that it merely permutes the commuting operators whose composition is $\L_n$. So, $\varphi\circ \L_n = \L_n\circ \varphi$ as desired.

If $\varphi$ is conjugate-linear then $\varphi_\ast$ is linear. Because $\L$ is positive, it then is hermitian. Therefore, 

\[\varphi(\L(B))=\varphi_\ast(\L(B)^\ast) =\varphi_\ast(\L(B^\ast))=\L(\varphi_\ast(B^\ast))=\L(\varphi(B)).\] 
\end{proof}

\section{Almost Reducing Projections and Nearly Commuting Operators}\label{Almost Reducing Projections and Nearly Commuting Operators}

We first discuss the two counter-examples that play a key role in the historical precedent of the results discussed in this section.
Voiculescu showed that for $\omega_n = e^{2\pi i/n}$, the unitary matrices (referred to as ``Voiculescu's unitaries'' henceforth) defined by
\begin{align}\label{Vunitaries}
U_n = \bp 
0 & & & &  1 \\
1 & 0 & & & \\
 & 1 & \ddots & & \\
& & \ddots & \ddots &  \\
& & & 1 & 0  \\
\ep, \;\;\; V_n = \bp 
\omega_n & & & & \\
& \omega_n^2 & & & \\
& & \ddots & & \\
& & &  \omega_n^{n-1} &  \\
& & & &  \omega_n^n  \\
\ep\end{align}
are almost commuting but are not nearly commuting. That is to say,  $[U_n, V_n] = U_nV_n - V_nU_n$ has small operator norm but there are no nearby commuting unitary matrices $U_n', V_n'$. 
The result that $U_n, V_n$ are not nearby \emph{any} commuting matrices was obtained later by Exel and Loring (\cite{exel1989almost}).

Recall that the (finite dimensional) unilateral shift on $\C^n$ refers to the lower triangular portion of $U_n$. Focusing on the spectral subspace of $V_n$ corresponding to the portion of the unit circle in the upper half-plane, Voiculescu's argument proceeds by contradiction. On this subspace, $U_n$ acts as the unilateral shift, so the proof constructs a projection whose range is almost invariant under the finite dimensional unilateral shift.
Then framing this in terms of Theorem 3 from \cite{halmos1968quasitriangular}, Voiculescu concludes.

This cited result from Halmos shows that the unilateral shift $U$ on $\ell^2(\N)$ defined by $Ue_n = e_{n+1}$ is not quasitriangular, 
meaning there does not exist a sequence $\{E_n\}$ of finite dimensional projections that such that $E_n \to 1$ strongly and the $E_n$ are almost invariant under $U$: $\|(1-E_n)UE_n\| \to 0$.

Davidson in 1985 (\cite{davidson1985almost}), among other advances, gave the example of almost commuting self-adjoint and almost normal matrices that do not have nearby commuting matrices. 
Define $(n^2+1)\times(n^2+1)$ matrices by
\[A_n = 
\bp
0&&&&\\
&\frac1{n^2}&&&\\
&&\frac2{n^2}&&\\
&&&\ddots&\\
&&&&1
\ep, \;
B_n = 
\bp
0&&&&&&&&&\\
\frac1n&0&&&&&&&&\\
&\ddots&\ddots&&&&&&&\\
&&\frac{n-1}{n}&0&&&&&&\\
&&&1&0&&&&&\\
&&&&\ddots&\ddots&&&&\\
&&&&&1&0&&&\\
&&&&&&\frac{n-1}{n}&0&&\\
&&&&&&&\ddots&\ddots&\\
&&&&&&&&\frac1n&0
\ep. 
\]
Note that by using Berg's result  that almost normal weighted shift matrices are close to normal matrices, Davidson proved that this counter-example showed that there is a pair of almost commuting self-adjoint and normal matrices that are not nearby commuting matrices of which the first is self-adjoint.

Notice from the form of Davidson's counter-example that there is a large subspace on which $B_n$ acts as the unilateral shift. 
With this observation, Davidson's argument is a  streamlining and simplification of Voiculescu's argument and the cited result from Halmos, although some of the details in \cite{halmos1968quasitriangular} are omitted. 
A thorough treatment of Voiculescu's unitaries counter-example (but with Exel and Loring's proof), the relevant material in Halmos's \cite{halmos1968quasitriangular}, and Davidson's counter-example can be found in S{\o}rensen's Masters Thesis: \cite{sorensen2012almost}.

Davidson's simplification is fitting particularly because Voiculescu frames an estimate concerning an almost invariant subspace for $U_n$ as an almost invariant subspace of the unilateral shift $U$ on $\ell^2(\N)$ in order to apply Halmos's result. However, Halmos' result is proved by examining finite dimensional projections.

Davidson's argument is also simpler because it can be. 
The fact that the spectrum of $A_n$ lies on the line allows one to estimate the norm of commutator of $B_n$ and a relevant projection $F$ and not just the norm of $(1-F)B_nF$ because the spectrum of $A_n$ can be bisected into two subsets that are only close on one frontier. 
In particular, Davidson does not need the additional decomposition of a projection in (ii) of Theorem \ref{Fpmrange} below. He only uses (i) and (iii) for a set like $\Omega = (-\infty, a]$.

\vspace{0.05in}

We now begin discussing almost invariant / almost reducing subspaces as generalizations of Voiculescu's and Davidson's arguments. The first lemma below shows that if $A, B$ are nearby structured commuting elements $A', B'$ then there is a nearby almost-reducing (and hence almost invariant) subspace $R(F)$. 

The building up a commuting elements using projections is in line with the spirit of Davidson's reformulation ($Q'$) of Lin's theorem. However, it is not a direct corollary of it because given a normal element $A$, we cannot perform the ``pinching'' by a projection $P$: $PAP+(1-P)A(1-P)$ and retain the spectral structure of $A$. For a self-adjoint element this will remain self-adjoint, which is a central part of Davidson's formulation of Lin's theorem, which has been used in \cite{hastings2009making}, \cite{herrera2020hastings}, \cite{herrera2022constructing}. However, for other normal operators, like unitaries, this can fundamentally affect what is being studied. This is a part of why Hastings's bootstrapping of Lin's theorem in \cite{hastings2009making} does not work for almost commuting unitaries, for instance.

The proof given below follows the proof in Theorem 2.3 of \cite{davidson1985almost} with the additional structure of $F$ from \cite{voiculescu1983asymptotically}. 
Note that we prove these results for von Neumann algebras. This is the most natural context given our dependence on spectral projections.

\begin{lemma}\label{proj requirement} Let $\A$ be a von Neumann algebra and ${S}$ some collection of weakly continuous linear symmetry maps on $\A$.
Suppose that $U, B \in \A$ with $U$ unitary and ${S}$-symmetric. 

Let
$\Omega_0\subset\subset \Omega$ be arcs in the unit circle $\bS$, not equalling the entire unit circle.
Let $\Omega_-, \Omega_+$ be the left and right disjoint intervals whose union is $\Omega\setminus \Omega_0$, where the ``right''/``left'' directions are defined to be ``clockwise''/``counterclockwise'', respectively. For two sets $S_1, S_2 \subset \C$, let $\dist(S_1, S_2)$ denote the standard Euclidean distance between the two sets. Let $c>0$ be as in Theorem \ref{DK theorem}.

Suppose that there exist $U', B' \in \A$ commuting with $U'$ unitary and ${S}$-symmetric.
Then there is an ${S}$-symmetric projection $F=F_{\Omega_0, \Omega}\in \A$ with the following properties:
\begin{enumerate}[label=(\roman*)]
\item $
E_{\Omega_0}(U) \leq F \leq E_{\Omega}(U),
$ 
\item $F$ has the decomposition $F=F_{-} + E_{\Omega_0}(U) + F_+$, where $F_-, F_+$ are ${S}$-symmetric projections with \begin{align}\label{Fpmrange}
F_- \leq  E_{\Omega_-}(U),\; F_+ \leq E_{\Omega_+}(U).
\end{align}
\item 
$\displaystyle\|\,[F,B]\,\| \leq \left(\frac{24c}{d(\bS\setminus \Omega, \Omega_0)}+\max\left(\frac{16c}{d(\bS\setminus \Omega, \Omega_0)}, \frac{32c}{d(\Omega_-, \Omega_+)}\right)\right)\|U'-U\|\|B\|$\\ 
\textcolor{white}{\_} \hspace{4in} $+ \|B'-B\|.$
\end{enumerate}
Also, the projections $F, F_-, F_+$ do not depend on $B$ or $B'$.
\end{lemma}

\begin{proof}
Let $U', B'$ be a pair of commuting elements with $U'$ unitary and ${S}$-symmetric. Let $\Omega'$ be the $\frac12\dist(\bS \setminus \Omega, \Omega_0)$ neighborhood of $\Omega_0$ in $\bS$.
So, $\Omega_0 \subset \Omega' \subset \Omega$ and the distance of each of these sets to the complement of the next set in the chain is $\frac12d(\bS \setminus \Omega, \Omega_0)$.

Now set $F' = E_{\Omega'}(U')$. We have that $F'$ is ${S}$-symmetric because the symmetries in ${S}$ are weakly continuous and linear.
We now proceed to perturb $F'$ to obtain $F$.

By Theorem \ref{DK theorem}, \[\|E_{\bS\setminus \Omega'}(U')E_{\Omega_0}(U)\|, \|E_{\bS \setminus \Omega}(U)E_{\Omega'}(U')\| \leq \frac{2c}{d(\bS\setminus \Omega, \Omega_0)}\|U'-U\|\]
so
\begin{align}\label{F ineqs}
\|(1-F')E_{\Omega_0}(U)\|, \|(1-E_{\Omega}(U))F'\| \leq \frac{2c}{d(\bS\setminus \Omega, \Omega_0)}\|U'-U\|.
\end{align}
Recall that for any two projections $F_1, F_2$, their product $F_1F_2$ and its adjoint $F_2F_1$ have the same norm. So, (\ref{F ineqs}) is equivalent to
\begin{align*}
\|E_{\Omega_0}(U)(1-F')\|, \|F'(1-E_{\Omega}(U))\| \leq \frac{2c}{d(\bS\setminus \Omega, \Omega_0)}\|U'-U\|.
\end{align*}

In order to obtain (ii), we filter $F'$ as in \cite{voiculescu1983asymptotically} before applying Davidson's lemma.
Observe that $\Omega_-, \Omega_0, \Omega_+$ form a partition of $\Omega$. We note that
\[E_{\Omega_0}(U') \leq F' \mbox{ and } F' \leq E_{\Omega}(U')= E_{\Omega_-}(U') + E_{\Omega_0}(U')+ E_{\Omega_+}(U'),\]
however we want a similar requirement to hold for $F$ with respect to $U$ and for the projection $F-E_{\Omega_0}(U)$ to decompose as the sum of two projections majorized, respectively, by the two spectral projections of $U$:  $E_{\Omega_-}(U), E_{\Omega_+}(U)$. 

Note that by the second inequality of Equation (\ref{F ineqs}), 
\begin{align}
\|F'-E_{\Omega}(U)F'E_{\Omega}(U)\| &\leq \|F'-E_{\Omega}(U)F'\|+\|E_{\Omega}(U)(F'-F'E_{\Omega}(U))\| \nonumber\\
&\leq \frac{4c}{d(\bS\setminus \Omega, \Omega_0)}\|U'-U\|.\label{F Omega}
\end{align}
Let \[X = E_{\Omega_-}(U)F'E_{\Omega_-}(U)+E_{\Omega_0}(U)+E_{\Omega_+}(U)F'E_{\Omega_+}(U).\] 
We obtain a few estimates to show that $X$ is close to $F'$ and construct from it the projection $F$ having the decomposition  $F_-+E_{\Omega}(U)+F_+$.

The norms of $E_{\Omega_0}(U)F'-E_{\Omega_0}(U)$ and its adjoint $F'E_{\Omega_0}(U)-E_{\Omega_0}(U)$ are controlled by Equation (\ref{F ineqs}). So because $\Omega_-, \Omega_0$, and $ \Omega_0, \Omega_+$ are disjoint,
\begin{align}
\nonumber \|E_{\Omega_-}(U)F'E_{\Omega_0}(U)\| &=
\|E_{\Omega_0}(U)F'E_{\Omega_-}(U)\|\leq \frac{2c}{d(\bS\setminus \Omega, \Omega_0)}\|U'-U\|,\\
\|E_{\Omega_+}(U)F'E_{\Omega_0}(U)\|&=\|E_{\Omega_0}(U)F'E_{\Omega_+}(U)\|
\leq \frac{2c}{d(\bS\setminus \Omega, \Omega_0)}\|U'-U\|. \label{Omega0 terms}
\end{align} 

We now calculate
\begin{align}
\nonumber E_{\Omega}&(U)F'E_{\Omega}(U) \\
\nonumber &= \left(E_{\Omega_-}(U)+E_{\Omega_0}(U)+E_{\Omega_+}(U)\right)F'\left(E_{\Omega_-}(U)+E_{\Omega_0}(U)+E_{\Omega_+}(U)\right)
\\
\nonumber&= X +E_{\Omega_0}(U)(F'-1)E_{\Omega_0}(U)
+E_{\Omega_-}(U)F'E_{\Omega_0}(U) +E_{\Omega_0}(U)F'E_{\Omega_-}(U)\\
\nonumber&\;\;\;\;\;\;\;\;+E_{\Omega_0}(U)F'E_{\Omega_+}(U)+E_{\Omega_+}(U)F'E_{\Omega_0}(U)\\
&\;\;\;\;\;\;\;\;
+E_{\Omega_-}(U)F'E_{\Omega_+}(U)
+E_{\Omega_+}(U)F'E_{\Omega_-}(U). \label{F', X}
\end{align}
We now show that the norms of all the terms except $X$ are small. Observe that $E_{\Omega_+}(U)F'E_{\Omega_-}(U)$ and $E_{\Omega_-}(U)F'E_{\Omega_+}(U)$ have norm controlled by Lemma \ref{three projs lemma} by
\[\|E_{\Omega_+}(U)E_{\Omega'}(U')E_{\Omega_-}(U)\| \leq \frac{4c}{d(\Omega_-, \Omega_+)}\|U'-U\|.\] 

Recall the fact that if $G_1, G_2, G_3 \in \A$ are orthogonal projections then for $S \in \A$,  
\begin{equation}\label{blockNorm}
\|G_1SG_2+G_2SG_3+G_3SG_1\| = \max(\|G_1SG_2\|, \|G_2SG_3\|, \|G_3SG_1\|)
\end{equation}
which can be seen by using a block operator matrix for ${S}$ and noting that each of the terms in the sum are in their own separate row and column. So, applying the identities in (\ref{F', X}) and (\ref{blockNorm}) provides
\begin{align}
\|\nonumber E_{\Omega}&(U)F'E_{\Omega}(U) -X-E_{\Omega_0}(U)(F'-1)E_{\Omega_0}(U)\|\\
&\leq\nonumber
\|E_{\Omega_-}(U)F'E_{\Omega_0}(U)+E_{\Omega_0}(U)F'E_{\Omega_+}(U)+E_{\Omega_+}(U)F'E_{\Omega_-}(U)\|\\
&\;\;\;+\nonumber
\|E_{\Omega_0}(U)F'E_{\Omega_-}(U)+E_{\Omega_-}(U)F'E_{\Omega_+}(U)+E_{\Omega_+}(U)F'E_{\Omega_0}(U)\|\\
&=2\max\left(\|E_{\Omega_0}(U)F'E_{\Omega_-}(U)\|, \|E_{\Omega_-}(U)F'E_{\Omega_+}(U)\|, \|E_{\Omega_+}(U)F'E_{\Omega_0}(U)\|\right).
\label{normTrick}
\end{align}
So, using this with the the estimates in (\ref{F Omega}), (\ref{Omega0 terms}), and the first part of (\ref{F ineqs}) we obtain
\begin{align}\nonumber
\|F'-X\| &\leq \left(\frac{6c}{d(\bS\setminus \Omega, \Omega_0)}+\max\left(\frac{4c}{d(\bS\setminus \Omega, \Omega_0)}, \frac{8c}{d(\Omega_-, \Omega_+)}\right)\right)\|U'-U\| \\
&= c_{\Omega,\Omega_0}\|U'-U\|.\label{F'XDiff}
\end{align}

Using $0 \leq X \leq 1$, $(F')^2-F=0$, and the identity $x^2-x= (x-1/2)^2-1/4$, we obtain
\begin{align}\nonumber
\|X^2-X\|&= \left\|\left(X-\frac12I\right)^2-\left(F'-\frac12I\right)^2\right\|\\
&\leq \left\|X-\frac12I\right\|\|X-F'\|
+\|X-F'\|\left\|F'-\frac12I\right\|\nonumber
\\
&\leq \|X-F'\|. \nonumber
\end{align}
We momentarily assume that $c_{\Omega',\Omega}\|U'-U\| < 1/4$ so $\|X^2-X\| < 1/4$ and $\|X\| \geq 1-\|X-F'\|> 1/2$.
Then by Lemma \ref{Strung lemma}, there is a projection $F$ that is a continuous real function $f$ of $X$ with $f(0)=0$ and $f(1)=1$ and \[\|F-X\| \leq 2\|X^2-X\| \leq 2\|F'-X\| \leq 2c_{\Omega,\Omega'}\|U'-U\|.\] 

Because $E_{\Omega_-}(U), E_{\Omega_0}(U), E_{\Omega
_+}(U)$ commute with $X$, they commute with $F$. 
We also have that $E_{\{1\}}(X) \leq F \perp E_{\{0\}}(X)$. Also, $R(E_{\Omega_0}(U))$ belongs to  $R(E_{\{1\}}(X))$, so $E_{\Omega_0}(U) \leq F$. Because $E_{\Omega}(U)X = X$ and $f(0)=0$, we have that $F \leq E_{\Omega}(U)$.

If we define $F_a = FE_{\Omega_a}(U)$ for $a = \pm$, then the $F_a \leq E_{\Omega_a}(U)$ are projections so that $F$ has the desired decomposition $F = F_- + E_{\Omega_0}(U) + F_+$.
Also,
\begin{align*} 
\|F'-F\| &\leq \|F'-X\|+\|X-F\| \leq 3c_{\Omega,\Omega'}\|U'-U\|.
\end{align*}

Recall that our derivation of this inequality only is valid when the assumption $c_{\Omega,\Omega
'}\|U'-U\| < 1/4$ held. 
For the case that $c_{\Omega,\Omega
'}\|U'-U\| \geq 1/4$, we instead choose $F = E_{\Omega_0}(U)$, $F_-=F_+=0$ so 
\[\|F'-F\| \leq 1 \leq 4c_{\Omega,\Omega'}\|U'-U\|.\] 
So, we have a projection $F$ with the the desired decomposition so that the estimate of $\|F'-F\|\leq 4c_{\Omega,\Omega'}\|U'-U\|$ holds without any assumption.

Since $[F',B']=0$, we see that
\begin{align*}\
|[F,B]\| &= \|[F-F',B]+[F',B-B']+[F',B']\| \\
&\leq \|[F-F',B]\|+\|[F'-1/2,B-B']\|\\
&\leq2\|F-F'\|\|B\|+\|B'-B\|
\end{align*}
which is the desired commutator inequality.

As for the symmetries, because $U$ and $U'$ are ${S}$-symmetric and ${S}$ is composed of weakly continuous linear symmetry maps, all spectral projections of $U$ and $U'$ are ${S}$-symmetric. Then by Proposition \ref{properties}\ref{triple product}, $X$ is ${S}$-symmetric as well. Then because $F$ is a real-continuous function of $X$, so is $F$. Because $F$ and $E_{\Omega_a}(U)$ are ${S}$-symmetric and commuting, $F_a$ are ${S}$-symmetric as well.
\end{proof}
\begin{remark}
This result can be extended to any normal operator $A$ with spectrum in some finite union of smooth curves in $\C$ along the lines of Section 16 of \cite{herrera2020hastings}. 

If we are interested in the case that $A$ has spectrum in some smooth simple closed or open curve in $\C$, then we can take a function on $\C$ that maps the curve into the unit circle so that it and its inverse are smooth enough. For instance, there are rational functions (like that of the Cayley transform) that work well when the spectrum is in a line that map into a compact proper subarc of the unit circle.
\end{remark}
\begin{remark}
Note that the condition (ii) is essential for the converse of this result to be true. (This is why Voiculescu included the condition (ii) in his construction.) That is, it is not enough for $F$ to be localized with respect to the spectrum of $U$ (i.e. (i)) and the range of $F$ to be an almost reducing subspace for $B$ (i.e. (iii)).

Roughly speaking, (ii) allows us to ``cut'' the spectrum of $U$ since the orthogonal complement of $F$ breaks with a complement due to $F_-$ and due to $F_+$ in an unentangled way. Otherwise, $FBF+(1-F)B(1-F)$ will commute with $F$ and be close to $B$, but may map one end of the complement of the spectrum of $U$ to the other end. We want our perturbation of $B$ to be localized in the spectrum of $U$ so that we can form $U'$ by combining eigenspaces, so this is not what we want.

Such a projection for Voiculescu's unitaries that satisfies (i) and (iii) can be constructed using Berg's gradual exchange lemma (\cite{berg1975approximation}) in the simplified form of \cite{herrera2022constructing}.
\end{remark}

We now discuss the converse of this result. The converse will provide a construction of nearby commuting operators given the existence of certain projections.
\begin{lemma}
\label{proj-converse}
Let $\A$ be a unital $C^\ast$-algebra and ${S}$ some collection of weakly continuous linear symmetry maps on $\A$.  Let $U, B \in \A$ with $U$ being unitary and ${S}$-symmetric.

Suppose that there is a constant $\epsilon(L)$ (explicitly depending on some parameter $L$ and implicitly depending on $U, B$)  so that for all arcs $\Omega_0 \subset\subset \Omega \subsetneq \bS$ with the left and right subarcs $\Omega_-, \Omega_+$ of $\Omega\setminus \Omega_0$, there is an ${S}$-symmetric projection $F=F_{\Omega_0, \Omega}\in \A$ with the following properties. 
 \begin{enumerate}
 \item $E_{\Omega_0}(U) \leq F\leq E_{\Omega}(U)$.
 \item $F$ can be decomposed as $F = F_- + E_{\Omega_0}(U), + F_+$, where $F_- \leq E_{\Omega_-}(U), F_+ \leq E_{\Omega_+}(U)$ are ${S}$-symmetric projections.
 \item \label{commutator-leqStatement}
 $F$ satisfies the commutator estimate
\begin{align}\label{commutator-leqconv}
\|\,[F,B]\,\| \leq \epsilon(L),
\end{align}
whenever $\dist(\bS\setminus \Omega, \Omega_0)$ and $\dist(\Omega_-, \Omega_+)$ are between $L$ and $2L$, for $L$ small enough.
\end{enumerate}

Then there is a $L_0>0$ so that for all $L\in (0, L_0)$, there are commuting $U'', B''\in \A$ with $U''$ unitary and ${S}$-symmetric satisfying
\begin{align}\label{ABestimates}
\|U''-U\| \leq \frac{11}2L, \; \|B''-B\| \leq \frac{17C_\rho}{L}\|[U,B]\| +2\epsilon(L).
\end{align}
The construction of $U''$ only depends on $U, L,$ and the projections. The map defining $B \mapsto B''$ is a completely positive linear map of norm $1$ that commutes with each symmetry in ${S}$.
\end{lemma}
\begin{remark}
We prove this result in a non-exact form to avoid the difficulties of determining the precise needed restrictions on the size of $L$ in order for the construction to work. The reoccurring issue is that arc length on the circle and Euclidean distance are equivalent for arcs of length $\ell$ not too large and are asymptotically equal as $\ell \to 0^+$. Optimizing the estimate in this result would needlessly complicate the main idea of the proof.
\end{remark}
\begin{proof} 
In order to construct nearby commuting elements $U'', B''$ we will construct a spectral resolution of the identity to be the spectral measure for $U''$ and we will perturb $B$ so that the eigenspaces of $U''$ are invariant subspaces of $B''$. So, $U''$ and $B''$ will commute. Let $\delta = \|[U,B]\|$. Note that $\delta = \|U^\ast[U,B]U^\ast\| = \|[U^\ast, B]\|$.

In order for this construction to work, we localize $B$. 
Let $L\in (0,2/3)$ be further restricted later. Let $\Delta \in (0, L)$ be arbitrary. 
Define $\tilde B = \L(B)$ where $\L$ is the localization operator of Lemma \ref{normal localization lemma}. For later reference, note that using Lemma \ref{phase normal localization lemma} would work equally as well with an increased constant. 
Now, $B \mapsto \tilde B$ is a completely positive map with $\|\tilde B\| \leq \|B\|, \|\tilde B - B\| \leq \frac{4\sqrt2C_\rho}{\Delta}\delta,$ $\|[U, \tilde B]\| \leq \delta$, $E_{K_2}(U)\tilde B E_{K_1}(U) = 0$ if $\dist(K_1, K_2) > \Delta$, and 
if $\varphi \in {S}$ then
\[\varphi(\L (B)) =\L(\varphi(B)).\]
We chose $\Delta < L$ so that when we view $\tilde B$ as a block matrix with respect to the projections $E_j$ that we construct below from some $F_{\Omega_0, \Omega}$, $\tilde B$ will be bilateral tridiagonal with off-diagonal blocks having small norm.
This was the purpose of constructing $\tilde B$ from $B$ and this is how we will use that the projections $E_j$ are almost commuting with $B$.

\vspace{0.05in} 

``Length'' will always refer to arc length and ``distance'' will always refer to the Euclidean distance in $\C$. 
Consider $n_0\in \N$ (divisible by $4$) many closed arcs $\Omega^k$ in $\bS$ of length $4L$, indexed cyclically, with the index increasing moving counter-clockwise on the unit circle. 
We require that the sets $\Omega^k$ each have a fixed distance between them, distance of at most $3L$ and at least $2L$. 
We call the open arc between $\Omega^{k}$ and $\Omega^{k+1}$ a ``gap interval'' and denote it by $\omega^k$ which we require have length at most $3L$.  

For consistency with the construction in a later theorem, we choose the points $\pm 1, \pm i$ to belong to the center of some gap intervals and we choose $1 \in \omega^0$.
In order for this construction to work, we will need $2\pi$ to be  $n_0$ times the length of $\Omega^k\cup\omega^k$, which is at least $6L$. 
We also need the distance between consecutive $\Omega^k$ and the distance between consecutive $\omega^k$ to be at least $L$. These assumptions can all be made for $L$ small enough.

Note that $\omega^0, \Omega^1, \omega^1, \Omega^2, \dots, \omega^{n_0-1}, \Omega^{n_0}$ form a partition of $\bS$ into arcs, each having length at least $3L$ and at most $4L$. We  partition $\Omega^k$ into arcs $\Omega_-^k, \Omega_0^k, \Omega_+^k$, with $\Omega_-^k, \Omega_+^k$ having length $L$ and $\Omega^k_0$ being an open interval of length $2L$.
Later in the proof we will make use of the arcs
 \[\Omega^{n_0}_+\cup\omega^0\cup \Omega^1_-,\; \Omega^1_0, \;\Omega^1_{+}\cup \omega^1\cup\Omega^2_-, \;\Omega^2_0,\;\dots,\; \Omega^{n_1-1}_+\cup \omega^{n_1-1}\cup\Omega^{n_0}_-,\; \Omega^{n_0}_0\]
which form an alternative partition of $\bS$ where each $\omega^k$ interval has grown in length $L$ on each end. So, these intervals have length at most $3L+2L = 5L$.

By definition, $d(\Omega_-^k, \Omega_+^k) \leq 2L$ and if $L$ is sufficiently small that this distance can be made to be at least $L$. Note that $\dist(\omega^k, \omega^{k+1})> L$ and $\dist(\Omega^k, \Omega^{k+1}) > L$ as well. Consequently,  $\dist(\Omega^{k_1}_+\cup \omega^{k_1} \cup \Omega^{k_1+1}_-, \Omega^{k_2}) > L$ if $|k_2-k_1|\geq 1$ and $|k_2-(k_1+1)|\geq 1$ for $L$ small.

We choose $L_0>0$ so that all the smallness requirements on $L$ stated thus far in this proof in addition to the smallness restriction on $L$ of (\ref{commutator-leqStatement}.) in the statement of the lemma is equivalent to $L \in (0, L_0)$.

Apply the assumption of this lemma for $\Omega^k_0, \Omega^k$ to obtain ${S}$-symmetric projections $F_k = F_{\Omega^k_0, \Omega^k}$ which have the decomposition $F_{k,-}+E_{\Omega^k_0}(U) + F_{k,+}$, where $F_{k,-} \leq E_{\Omega^k_{-}}(U), F_{k,+} \leq E_{\Omega^k_+}(U)$ are also ${S}$-symmetric projections and 
\[ \|\,[F_k, B]\;\| \leq \epsilon(L).\]
Setting $F^c_{k,-} = E_{\Omega^k_{-}}(U) - F_{k,-}$ and $F^c_{k,+} = E_{\Omega^k_{+}}(U) - F_{k,+}$, we obtain the decomposition of $E_{\Omega^k}(U)$ as the sum $F^c_{k,-} + F_k + F^c_{k,+}$ of projections with orthogonal ranges. 

Because each $\varphi\in {S}$ is a weakly continuous linear symmetry and $U$ is ${S}$-symmetric, we see that all these projections are ${S}$-symmetric as well. Define $G_k = F_{k-1,+}^c + E_{\omega^k}(U) + F_{k,-}^c$. Then the projections $F_1, G_1, F_2, \dots, G_{n_0-1}, F_{n_0}, G_0$ form a resolution of the identity of ${S}$-symmetric projections so that 
\begin{align}\label{F and G leq}
E_{\omega^k}(U) \leq G_k \leq E_{\Omega^{k}_+\cup \omega^k \cup \Omega^{k+1}_-}(U), \;\;\; E_{\Omega^k_0}(U) \leq F_k \leq E_{\Omega^k}(U).
\end{align}
So, $F_1, G_1, \dots, F_{n_0}, G_0$ are:
\begin{align*}
&F_1,\;\; F_{1,+}^c + E_{\omega^1}(U) + F_{2,-}^c,\;\; F_2,\;\; F_{2,+}^c + E_{\omega^2}(U) + F_{3,-}^c,\;\; \dots, \\
&F_{n_0-1,+}^c + E_{\omega^{n_0}}(U) + F_{n_0,-}^c,\;\; F_{n_0},\;\; F_{n_0,+}^c + E_{\omega^0}(U) + F_{1,-}^c.
\end{align*}
For simplicity of notation, we relabel these projections as $E_j$ for $j = 1,\dots, 2n_0$.

If $E_j = F_k$ for some $k$, let $a_j$ be the center point of $\Omega^k_0$. If $E_j = G_k$ for some $k$, let $a_j$ be the center point of $\omega^k$.
Define \[U'' = \sum_j a_j E_j, \;\; B'' = \sum_j E_j \tilde B E_j.\] 

Then $U'', B''$ commute and the map $B \mapsto B''$ is completely positive which maps $I$ to itself, so it has norm $1$. Because the projections $E_j$ are ${S}$-symmetric, so is $U''$. Also, if  $\varphi \in {S}$ then
\begin{align*}
\varphi(B'') &= \varphi\left(\sum_k E_k \L(B)E_k\right)=\sum_k\varphi(E_k)\varphi(\L(B))\varphi(E_k) \\
&=\sum_kE_k\L(\varphi(B))E_k.
\end{align*}

We now need to estimate the norms of $U''-U, B''-B$.  We start with $U''-U$.
For an arc $I$, let $a_I$ denote its center.
Note that \[U'' = \sum_k \left(a_{\Omega^k}F_k+a_{\omega^k}(F_{k,
+}^c+E_{\omega^k}(U)+F_{k+1,-}^c)\right).\] 
However, because $E_{\Omega^k_-}(U) = F_{k,-}+F_{k,-}^c$ and $E_{\Omega^k_+}(U) = F_{k,+}+F_{k,+}^c$, $F_k = F_{k,-}+E_{\Omega_0^k}(U)+F_{k,+}$, and $\Omega^k$ is the disjoint union of $\Omega^k_-, \Omega^k_0, \Omega^k_+$ we have that
\begin{align*}
U'' &= \sum_k \left( a_{\Omega^k}(F_{k,-}+E_{\Omega^k_0}(U)+F_{k,+})+a_{\omega^k}E_{\omega^k}(U) + a_{\omega^k}(F^c_{k,+}+F^c_{k+1,-}) \right)\\
&= \sum_k \left( a_{\Omega^k}E_{\Omega^k}(U)+a_{\omega^k}E_{\omega^k}(U) -a_{\Omega^k}F^c_{k,-}+ (a_{\omega^k}-a_{\Omega^k})F^c_{k,+}+a_{\omega^k}F^c_{k+1,-} \right).
\end{align*}
So, because $\{\Omega^k, \omega^k\}_k$ form a partition of $\bS$ and because the distance of the center of each of intervals $\omega^k, \Omega^k$ it its end points is at most $2L$, we have
\[\|\sum_k \left( a_{\Omega^k}E_{\Omega^k}(U)+a_{\omega^k}E_{\omega^k}(U)\right)-U\|\leq 2L.\]
Also,
\begin{align*}
\sum_k& \left(-a_{\Omega^k}F^c_{k,-}+ (a_{\omega^k}-a_{\Omega^k})F^c_{k,+}+a_{\omega^k}F^c_{k+1,-} \right)\\
&=\sum_k \left(-a_{\Omega^k}F^c_{k,-}+ (a_{\omega^k}-a_{\Omega^k})F^c_{k,+}+a_{\omega^{k-1}}F^c_{k,-} \right)
\end{align*}
so
\begin{align*}
\|U''-U\| &\leq 2L + \max_k\max(| a_{\omega^k}-a_{\Omega^k}|, |a_{\omega^{k-1}}-a_{\Omega^k}|)\\
&< 2L + \left(\frac{3L}{2}+\frac{4L}{2}\right) = \frac{11}{2}L.
\end{align*}

We now estimate the norm of $B''-B$:
\[\|B''-B\|\leq \|B''-\tilde B\| + \|\tilde B - B\| \leq \|B''-\tilde B\| + \frac{4\sqrt2C_\rho}{\Delta}\delta.\]
We will now show that $\tilde B$ has the following bilateral block tridiagonal structure with respect to the projections $E_j$:
\[\bp 
\ast   &\ast& & & &\ast \\
\ast&  \ast &\ast&&&    \\
&\ast&\ast&\ddots&& \\
&&\ddots&\ddots&\ast& \\
& & &\ast&  \ast & \ast \\
\ast & & & & \ast& \ast
\ep\]
and that the off-diagonal terms are small in norm. We first will show the tridiagonal property.

By construction, $\tilde B$ satisfies the property that $E_{K_2}(U) \tilde B E_{K_1}(U) = 0$ if $\dist(K_1, K_2) > \Delta$. We assumed that $\Delta \leq L$ and that $L$ was small enough that $\dist(\Omega^{k_1}_+\cup \omega^{k_1} \cup \Omega^{k_1+1}_-, \Omega^{k_2}) > L$ if $|k_2-k_1|\geq 1$ and $|k_2-(k_1+1)|\geq 1$. So, we see by Equation (\ref{F and G leq}) that if $|j_2 - j_1| \geq 2$ then $E_{j_2}\tilde B E_{j_1} = 0.$ This shows the tridiagonal property.

We now show that the off-diagonal terms are small in norm. By assumption, the projections $F_k \geq 0$ satisfy
\begin{align}\label{invsub4H}\nonumber
\|(1-F_k)\tilde BF_k\|,\,\|F_k\tilde B(1-F_k)\| &\leq \|[F_k,\tilde B]\|\leq \|\tilde B-B\|+\|[F_k,B]\| \\
&\leq \frac{4\sqrt2C_\rho}{\Delta}\delta + \epsilon(L)= c_F.
\end{align}

Consider the projections $E_{j_1}, E_{j_2}$ for $|j_2-j_1| = 1$. Necessarily, one of $j_1, j_2$ will be even and hence the projection with that index will equal some $F_k$.
If $E_{j_1} = F_k$ for some $k$ then for $j_2 = j_1\pm 1$,
\[\|E_{j_2}\tilde BE_{j_1}\| = \|E_{j_2}(1-E_{j_1})\tilde BE_{j_1}\| \leq c_F.\] 
If $E_{j_2} = F_k$ for some $k$ then for $j_1 = j_2\pm 1$,
\[\|E_{j_2}\tilde BE_{j_1}\| = \|E_{j_2}\tilde B(1-E_{j_2})E_{j_1}\| \leq c_F.\] 

Now, because $B'' = \sum_k E_k \tilde B E_k$, we see that $\tilde B - B''$ can be expressed as the block matrix with respect to the $E_j$:
\[ \bp 
0   &\ast& & & & \\
&  0 &\ast&&&    \\
&&0&\ddots&& \\
&&&\ddots&\ast& \\
& & &&  0 & \ast \\
\ast & & & & & 0
\ep+\bp 
0   && & & &\ast \\
\ast&  0 &&&&    \\
&\ast&0&&& \\
&&\ddots&\ddots&& \\
& & &\ast&  0 &  \\
 & & & & \ast& 0
\ep\] with $\|\ast\| \leq c_F$.
Bounding each of these matrices separately shows that $\|B'' - \tilde B\| \leq 2c_F.$

Putting these estimates together, we obtain
\[\|U''-U\| < \frac{11}2L, \; \|B''-B\| \leq \frac{12\sqrt2C_\rho}{\Delta}\delta +2\epsilon(L)\]
under the assumption that $\Delta < L$ was arbitrary. So, we pick $\Delta$ close enough to $L$ to obtain
\begin{align}
\|U''-U\| \leq \frac{11}2L, \; \|B''-B\| \leq \frac{17C_\rho}{L}\delta +2\epsilon(L).
\end{align}
Note also that if $\varphi \in {S}$ then $\varphi$ commutes with $\L$ and preserves the $E_j$. So, $\varphi(B'') = \sum_j E_j \L(\varphi(B))E_j$. 
\end{proof}
We note that the previous lemma can be extended to include additional non-unital phase symmetries or conjugate-linear phase symmetries. The idea for the construction is based in the following observation: Suppose that we have two almost commuting real matrices $U, B$. Then if we were able to construct $U''$ so that it is real, then $[U'', B'']=0$ and $[U'', \overline{B''}] = \overline{[U'', B'']}=0$. So, $U''$ commutes with $B''$ and $\overline{B''}$. It then follows that $U''$ commutes with the matrix of the real entries of $B'': (B''+\overline{B''})/2$ gotten by symmetrization.  So, we can obtain $U'', B''$ both real. If $B$ were instead imaginary, then we would instead symmetrize as $(B''-\overline{B''})/2$. 

However, this symmetrization is actually unnecessary because the map that transforms $B$ into $B''$ is a positive linear map that commutes with the respective symmetries of $U$. So, if $B$ is real or imaginary then $B''$ is automatically real or imaginary.

We now extend this idea into the construction of $U''$ and $B''$.
The reader should think of this result as an addendum to the previous result.
\begin{lemma}\label{proj-converse extra sym}
Let $\A, U, B$ satisfy the assumptions of Lemma \ref{proj-converse}.
Suppose that one of the following cases of additional phase symmetries for $U$ and $B$ hold. Then we will have additional phase symmetries for $U''$ and $B''$ as well as described in each case. 
\begin{enumerate}[label=(\roman*)]
\item\label{conj-lin sym addition} Suppose that $\phi$ is a weakly continuous conjugate-linear symmetry map on $\A$ with order $2$ that commutes with each symmetry of ${S}$. Suppose that $\phi(U)=\zeta_0 U$ for $\zeta_0$ in the unit circle. Suppose that $\phi(B)=\eta_0 B$ (resp. $\phi_{\ast}(B)=\eta_0 B$) for $\eta_0\in \C$.

Then $U'', B''$ in Lemma \ref{proj-converse} can be chosen to satisfy $\phi(U'')=\zeta_0 U''$ and $\phi(B'')=\eta_0 B''$ (resp. $\phi_{\ast}(B'')=\eta_0 B''$).

\item\label{lin sym addition} Suppose that $\varphi$ is a weakly continuous linear symmetry map on $\A$ with finite order $n\geq 2$ that commutes with each symmetry of ${S}$. Suppose that $\varphi(U)=\zeta_1 U$ for $\zeta_1$ in the unit circle having order $n$. Suppose that $\varphi(B)=\eta_1 B$ (resp. $\varphi_{\ast}(B)=\eta_1 B$) for $\eta_1\in \C$.

Then $U'', B''$ in Lemma \ref{proj-converse} can be chosen so that $\varphi(U'')=\zeta_1 U''$ and $\varphi(B'')=\eta_1 B''$ (resp. $\varphi_{\ast}(B'')=\eta_1 B''$) satisfying the same inequality but when $n \geq 3$, $C_\rho$ will be replaced with $nC_\rho$. 

We additionally require $L_0$ to be smaller and
$\|[U,B]\|$ smaller in the second estimate in a way that depends on $n$.

\item Suppose that there are $\phi, \varphi$ as in both \ref{conj-lin sym addition} and \ref{lin sym addition} that additionally have the properties that they satisfy the dihedral commutation relation: 
$\phi\circ\varphi\circ \phi = \varphi^{-1}$. 

Then $U'', B''$ in Lemma \ref{proj-converse} can be chosen to satisfy each of the phase symmetries of  \ref{conj-lin sym addition} and \ref{lin sym addition}.

\end{enumerate}
\end{lemma}

\begin{remark}
Note that although additional symmetries of $B$ are required, the construction of $U''$ depends only on $U$, $L$ (resp. $L_\ast$), the projections, and the symmetries $\phi$ and/or $\varphi$ (depending on the case). We need $B$ to satisfy a phase symmetry with respect to the symmetries so that the projection commutator inequality in (\ref{commutator-leqStatement}.) hold for our projections constructed from the symmetries.  The map $B \mapsto B''$ now commutes with the symmetries of ${S}$ and the additional symmetries considered in this statement.

Also, observe that because the construction of $B \mapsto B''$ in the previous scenario did not involve the symmetries of ${S}$, there was no need to require them to commute. However, the projections formed in this lemma are constructed from the symmetries $\phi$ and/or $\varphi$ and so we require these symmetries to satisfy certain commutation relations (when they are both present) and that they commute with all the former symmetries. The projections formed are ${S}$-symmetric but the symmetries $\phi, \varphi$ permute the constructed projections.
\end{remark}

\begin{proof}
We constructed the sets $\Omega^k, \Omega^k_-, \Omega^k_0, \Omega^k_+$ so that they are invariant under the symmetries $z\mapsto \overline {z}, z\mapsto -\overline {z}, z\mapsto -z$ up to a relabeling to account for changes in orientation. However, our definitions of the projections $E_j$ will need to change so that $U''$ will have the required phase symmetries. Besides that, the construction will not change. For Case (ii), we will need the center points of some additional $\omega^k$ to be fixed which will require $L$ to be smaller, but we will get to that later. Let $m = n_0/4\in \N$.

Observe that $1 \in \omega^0$, $i \in \omega^{m}$, $-1 \in \omega^{2m}$, and $-i \in \omega^{3m}$. These are the only sets that do not lie in a quadrant of $\R^2=\C$. The intervals $\Omega^k$ for $k \in (rm, (r+1)m]$ and $\omega^k$ for $k \in (rm, (r+1)m)$ lie in the $(r+1)$th quadrant.

With these remarks we proceed to the cases.

\vspace{0.05in}
\noindent \underline{Case (i)}: Without loss of generality, we can assume that $\zeta_0=1$ by replacing $U$ with $\zeta_0^{1/2}U$ since $\phi$ is conjugate linear. So, $U$ has a $\phi$ conjugate-linear symmetry. 

Consider the sets $\omega^0, \Omega^1, \omega^1, \Omega^2, \dots, \omega^{4m-1}, \Omega^{4m}=\Omega^0$. The sets in the collections $\{\Omega^k\}_k$ and $\{\omega^k\}_k$ are permuted by conjugation. More precisely, we have $\overline{\omega^k} = \omega^{-k+4m}$,  $\overline{\Omega^k} = \Omega^{-k+4m+1}$, $\overline{\Omega^k_0} = \Omega^{-k+4m+1}_0$, $\overline{\Omega^k_+} = \Omega^{-k+4m+1}_-$, $\overline{\Omega^k_-} = \Omega^{-k+4m+1}_+$. In particular, $\overline{\omega^0}=\omega^0$ and $\overline{\omega^{2m}}=\omega^{2m}$.

We claim that we only need to show that conditions (1.), (2.), and (3.) are satisfied when instead of using 
\[F_{\overline{\Omega_0}, \overline{\Omega}}, \;\;
F_{\overline{\Omega_0}, \overline{\Omega},-},\;\; F_{\overline{\Omega_0}, \overline{\Omega},+},\]
we use the projections 
\[\phi(F_{\Omega_0, \Omega}),\;\; \phi(F_{\Omega_0, \Omega, +}),\;\; \phi(F_{\Omega_0, \Omega, -}),\] 
respectively. Notice the interchange in $\pm$ due to the reflection's change in orientation. 

This assertion is true because for $\Omega_0, \Omega$ in the upper half-plane we apply the assumption to get $F_{k} = F_{\Omega^k_0, \Omega^k}$ for $k = 1, \dots, 2m$ as in the previous lemma. Then define the rest of the projections by 
\begin{align}\label{(i) prof def}
F_{-k+4m+1} = \phi(F_k),\; F_{-k+4m+1,-}= \phi(F_{k,+}),\; F_{-k+4m+1,+}= \phi(F_{k,-})
\end{align} 
for $k = 1, \dots, 2m$.
If (1.), (2.), and (3.) hold then the these projections are just as good as any projections satisfying these assumptions so the construction given previously works exactly as it did before.

Notice that because $\phi$ has order $2$ and so does the index change $k\mapsto -k+4m+1$, we see that Equation (\ref{(i) prof def}) holds for all $k$.

So, what we will show is that these conditions hold. We then obtain $U'', B''$ with the properties of the previous theorem. We then show that $U''$ is $\phi$-symmetric by our choice of the projections. We then check that $B''$ satisfies the $\eta_0$ phase symmetry from $\phi$.

\begin{enumerate}
\item By Proposition \ref{properties}\ref{spectral proj transformation} and \ref{proj inequality}, if we apply the conjugate-linear symmetry map $\phi$ to (1.) for $F_{\Omega_0, \Omega}$ then we get 
\[E_{\overline{\Omega_0}}(U)=\phi(E_{\Omega_0}(U))\leq \phi(F) \leq \phi(E_{\Omega}(U))=E_{\overline{\Omega}}(U).\]

\item Applying $\phi$ to Condition (2.) for $F$, we get $\phi(F)$ can be decomposed as
$\phi(F) = \phi(F_+) + \phi(E_{\Omega_0}(U)) + \phi(F_-)$,
where $\phi(F_-) \leq \phi(E_{\Omega_-}(U))=E_{\overline{\Omega}_+}(U)$ and $\phi(F_+) \leq E_{\overline{\Omega}_-}(U)$.

Because $\phi$ commutes with every symmetry of ${S}$, $\phi(F)$, $\phi(F_-)$, $\phi(F_+)$ are ${S}$-symmetric as well.

\item Note that $\phi$ does not change norms. If $\phi(B) = \eta_0 B$ with $|\eta_0|=1$, we have that 
\[\|[\phi(F),B]\|=\|\phi([F,B])\|=\|[F,B]\| \leq \epsilon(L).\]
For the situation that $\phi_{\ast}(B) = \eta_0 B$ with $|\eta_0|=1$ then $\phi(B) = \overline{\eta_0}B^\ast$ so because $\phi(F)^\ast = \phi(F)$, 
\begin{align*}
\|[\phi(F),B]\|&=\|[\phi(F),B]^\ast\|=\|[\phi(F),B^\ast]\|=\|\phi([F,B])\|\\
&=\|[F,B]\| \leq\epsilon(L).
\end{align*}
\end{enumerate}
So, we showed the three conditions.

The projections $F_k$ for $k=1, \dots, 4m$ are (in this order):
\[F_1, F_2, \dots, F_{m-1}, F_{2m}, \phi(F_{2m}), \phi(F_{2m-1}), \dots, \phi(F_2), \phi(F_1).\]
The expression $G_k = F_{k,+}^c+E_{\omega^k}(U)+F_{k+1,-}^c$ holds for all $k$. For $k=2m+1, \dots, 4m$, this looks like 
\[G_k = F_{k,+}^c+E_{\omega^k}(U)+F_{k+1,-}^c= \phi(F_{-k-1+4m,-}^c)+E_{\overline{\omega^{-k+4m}}}(U)+\phi(F_{-k+4m,+}^c).\]
The projections $F_1, G_1, \dots, F_{4m}, G_{4m}=G_0$ are (in this order):
\begin{align*}
&F_1,\; F_{1,+}^c+E_{\omega^1}(U)+F_{2,-}^c,\; F_2,\; F_{2,+}^c+E_{\omega^2}(U)+F_{3,-}^c,\; \dots, \\
&F_{2m-1,-}^c+E_{\omega^{2m-1}}(U)+F_{2m,-}^c,\; F_{2m},\; F_{2m,+}^c+E_{\omega^{2m}}(U)+\phi(F_{2m,+}^c),\; \phi(F_{2m}), \\
&\phi(F_{2m,-}^c)+E_{\overline{\omega^{2m-1}}}(U)+\phi(F_{2m-1,+}^c),\; \phi(F_{2m-1}),\; \dots, \phi(F_2), \\
&\phi(F_{1,-}^c)+E_{\overline{\omega^{2}}}(U)+\phi(F_{2,+}^c),\; \phi(F_1),\; \phi(F_{1,-}^c)+E_{\omega^0}(U)+F_{1,-}^c.
\end{align*}
Recall that \[U'' = \sum_k \left(a_{\Omega^k}F_k+a_{\omega^k}G_k\right).\]
Note that 
\[\overline{a_{\Omega^k}}=a_{\overline{\Omega^k}}=a_{\Omega^{-k+4m+1}}\] and
$\overline{a_{\omega^k}} = a_{\omega^{-k+4m}}$.
Also,
\begin{align*}
\phi(G_k)&=\phi(F_{k,
+}^c+E_{\omega^k}(U)+F_{k+1,-}^c) \\
&= F_{-k+4m+1,
-}^c+E_{\omega^{-k+4m}}(U)+F_{-k+4m,+}^c=G_{-k+4m}.
\end{align*}
So,
\begin{align*}
\phi(U'') &=  \sum_k \left(\overline{a_{\Omega^k}}\phi(F_k)+\overline{a_{\omega^k}}\phi(G_k)\right)\\
&\sum_k\left(a_{\Omega^{-k+4m+1}}F_{-k+4m+1}+a_{\omega^{-k+4m}}G_{-k+4m}\right)=U'',
\end{align*}
because this is a permutation of the terms in the sum defining $U''$. Also,
\begin{align*}
\phi(B'') &= \phi\left(\sum_k (F_k\L(B)F_k + G_k\L(B)G_k)\right) \\
&= \sum_k (F_{4m+1-k}\L(\phi(B))F_{4m+1-k} + G_{4m-k}\L(\phi(B))G_{4m-k}).
\end{align*}
So, $\phi$ commutes with the positive linear map $B \mapsto B''$. So, $B''$  satisfies the required phase symmetries.

This is what we wanted to show. Observe the symmetries were preserved simply because $\phi$ permuted the $F_k$, permuted the $G_k$, permuted the $a_{\Omega^k}F_k$, and permuted the $a_{\omega^k}G_k$.

\vspace{0.05in}
\noindent \underline{Case (ii)}: In this case, we also need to use the form of the localization operator that respects the phase symmetries of $U$ in the case $n \neq 2$. ($n=2$ phase symmetries are just antisymmetries which are respected by the localization operator used in the original construction.) So, when $n \geq 3$, we define $\tilde B = \L_n(B)$ where $\L_n$ is defined in Lemma \ref{phase normal localization lemma}. 

In order to obtain the required phase symmetry we need the sets $\Omega^k, \omega^k$ to be invariant, up to reindexing, under the rotation by $2\pi/n$. When $n = 2, 4$ this is already the case based on how we originally defined the $\omega^k$ to contain $\pm1, \pm i$. 

However, for general $n \geq 3$, we will need to add additional requirements to these sets. Let $\xi_n = e^{2\pi i/n}$. We require the $n$th roots of unity $1, \xi_n, \xi_n^2, \dots, \xi_n^{n-1}$ to be the center of some of the intervals $\omega^k$. For consistency with the construction earlier, we require also $\pm 1, \pm i$ to be the center of some of the $\omega^k$.

We will need $2\pi$ to be  $n_0$ times the length of $\Omega^k\cup\omega^k$, where $n_0$ is a multiple of $n$. For consistency with our construction in the other cases, we assume that $n_0$ is also a multiple of $2$. This can be satisfied when $L$ is sufficiently small.

Let $p = n_0/n\in \N$. Because the order of $\zeta$ is $n$, we see that the points $1, \zeta, \dots, \zeta^{n-1}$ are distinct. So, there is an $s \in \{1, \dots, n-1\}$ such that $\zeta^s = \xi_n$. Therefore, $\varphi^s(U) = \xi_n U$ and $A\in \mathcal A$ has a $\zeta$-phase symmetry with respect to $\varphi$ if and only if it has a $\xi_n$-phase symmetry with respect to the symmetry map $\varphi^s$.

If $\Omega$ is an arc and $\varphi(U) = \zeta U$ then since $\varphi$ is linear,
\[\varphi(E_{\Omega}(U)) = E_{\Omega}(\varphi(U)) = E_{\Omega}(\zeta U) = E_{\zeta^{-1}\Omega}(U).\]
So, we need  the sets to be permuted by rotations by $\zeta^{-1}$. However,  we see that it suffices to show that the sets are invariant under a rotation by $\xi_n$.

We now proceed as in Case (i). 
Consider the sets $\omega^0, \Omega^1, \omega^1, \Omega^2, \dots, \omega^{np-1}, \Omega^{np}$. Because the collection of sets $\{\Omega^k\}_k, \{\omega^k\}_k$ are invariant under rotation by $\xi_n$, we have 
\[\xi_n\cdot\omega^k = \omega^{k+p},  \xi_n\cdot\Omega^k = \Omega^{k+p}, \xi_n\cdot\Omega^k_0 = \Omega^{k+p}_0, \xi_n\cdot\Omega^k_- = \Omega^{k+p}_-, \xi_n\cdot\Omega^k_+ = \Omega^{k+p}_+.\]
So, if $R_k$ is one of these sets then 
\[\varphi^s(E_{R_k}(X)) = E_{\zeta^{-s}R_k}(X) = E_{R_{k-p}}(X).\]
Consider the arc $I_r=\{e^{i\theta}\in\bS: r\frac{2\pi}n< \theta < (r+1)\frac{2\pi}n\}$. So, except for the $n$th roots of unity, the arcs $I_0, \dots, I_{n-1}$ cover $\bS$. Moreover, each $\Omega^k$ is a subset of one of the $I_r$.

We claim that we only need to show that conditions (1.), (2.), and (3.) are satisfied when instead of using 
$F_{\xi_n^{-r}\cdot\Omega_0, \xi_n^{-r}\cdot\Omega}$ and associated $F_{\xi_n^{-r}\cdot\Omega_0, \xi_n^{-r}\cdot\Omega,-}$, $F_{\xi_n^{-r}\cdot\Omega_0, \xi_n^{-r}\cdot\Omega,+}$,
we use $\varphi^{rs}(F_{\Omega_0, \Omega}), \varphi^{rs}(F_{\Omega_0, \Omega, -}), \varphi^{rs}(F_{\Omega_0, \Omega, +})$, respectively. 
Notice that there is no interchange in $\pm$ because rotations preserve orientation. 
Verifying all these conditions is essentially identical to what we did before. 

We now explain how we use this. For the $\Omega^k$ in the fundamental arc $I_0$, we use projections $F_k, F_{k,-}, F_{k,+}$ guaranteed by conditions 1-3. For $\Omega^k \in I_r$ in general, there is a unique $r$ in $\{0, 1, \dots, n-1\}$ (equal to $r$ modulo $p$) so that $\Omega^{k-rsp}\subset I_0$. We define $F_k = \varphi^{-rs}(F_{k-rp})$, $F_{k,-} = \varphi^{-rs}(F_{k-rp,-})$, $F_{k,+} = \varphi^{-rs}(F_{k-rp,+})$.
Notice that because $\varphi^s$ has order $n$ and so does the index change $k\mapsto k-p$, we see that $\varphi^s(F_k) = F_{k-p}$ holds for all $k$, where the index $k$ is interpreted modulo $n_0$.

The projections $F_k$ for $k = 1, \dots, n_0=np$  are (in this order):
\begin{align*}
&F_1, \dots, F_{p},\;\; \\
&\varphi^{-s}(F_{1}), \dots, \varphi^{-s}(F_{p}), \\
&\varphi^{-2s}(F_{1}), \dots, \varphi^{-2s}(F_{p}),\\ 
&\;\;\;\;\;\dots, \\
&\varphi^{-(n-1)s}(F_{1}), \dots, \varphi^{-(n-1)s}(F_{p}). 
\end{align*}
Note that we could assert that the number of compositions of $\varphi$ in the above belong to $\{0, \dots, n-1\}$ by reducing each of the terms $0, -s, -2s, \dots, -(n-1)s$ modulo $n$. However, we see no need to do this.

The expression $G_k = F_{k,+}^c+E_{\omega^k}(U)+F_{k+1,-}^c$ holds for all $k$. For $k=rp+1, \dots, (r+1)p$, this looks like 
\[G_k = F_{k,+}^c+E_{\omega^k}(U)+F_{k+1,-}^c= \varphi^{-rs}(F_{{k-rp},+}^c)+E_{\omega^k}(U)+\varphi^{-rs}(F_{k+1-rp,-}^c).\]
The projections $F_1, G_1, \dots, F_{4m}, G_{4m}=G_0$ are (in this order):
\begin{align*}
&F_1,\; F_{1,+}^c+E_{\omega^1}(U)+F_{2,-}^c,\; \dots, F_p, F_{p,+}^c+E_{\omega^p}(U)+\varphi^{-s}(F_{1,-}^c),   \; \\
&\varphi^{-s}(F_1),\; \varphi^{-s}(F_{1,+}^c)+E_{\omega^{p+1}}(U)+\varphi^{-s}(F_{2,-}^c),\; \dots, \varphi^{-s}(F_p), \\
&\;\;\;\;\;\varphi^{-s}(F_{p,+}^c)+E_{\omega^{2p}}(U)+\varphi^{-2s}(F_{1,-}^c),\;\; \\
&\dots\\
&\varphi^{-(n-1)s}(F_1),\; \varphi^{-(n-1)s}(F_{1,+}^c)+E_{\omega^{(n-1)p+1}}(U)+\varphi^{-(n-1)s}(F_{2,-}^c),\; \dots, \varphi^{-(n-1)s}(F_p), \\
&\;\;\;\;\;\varphi^{-(n-1)s}(F_{p,+}^c)+E_{\omega^{np}}(U)+F_{1,-}^c. 
\end{align*}
Recall that \[U'' = \sum_k \left(a_{\Omega^k}F_k+a_{\omega^k}G_k\right).\]
Note that 
\[\xi_n{a_{\Omega^k}}=a_{\xi_n\Omega^k}=a_{\Omega^{k+p}}\] and
$\xi_n{a_{\omega^k}} = a_{\omega^{k+p}}$. 
Also,
\begin{align*}
\varphi^s(G_k)&=\varphi^s(F_{k,
+}^c+E_{\omega^k}(U)+F_{k+1,-}^c) \\
&= F_{k-p,
+}^c+E_{\omega^{k-p}}(U)+F_{k+1-p,-}^c=G_{k-p}.
\end{align*}
So,
\begin{align*}
\varphi^s(U'') &= \sum_k\left(a_{\Omega^{k}}F_{k-p}+a_{\omega^{k}}G_{k-p}\right)\\
&=\xi_n\sum_k\left(a_{\Omega^{k-p}}F_{k-sp}+a_{\omega^{k-p}}G_{k-p}\right)= \xi_n U''.
\end{align*}
Iterating this provides $\varphi(U'') = \zeta U''$.
Also,
\begin{align*}
\varphi^s(B'') &= \sum_k \left(F_{k-p}\L_n(\varphi(B))F_{k-p} + G_{k-p}\L_n(\varphi(B))G_{k-p}\right).
\end{align*}
This guarantees that $B \mapsto B''$ commutes with $\varphi^s$, and hence also with $\varphi$. Hence, $B''$ satisfies the required phase symmetries.

\vspace{0.05in}
\noindent \underline{Case (iii)}: As in Case (i), we may assume that $\zeta_0 = 1$ and write $\zeta$ in place of $\zeta_1$. We define $\tilde B = \L_n(B)$ where $\L_n$ is defined in Lemma \ref{phase normal localization lemma}. 
We will require that $2\pi$ be $n_0$ times the length of $\Omega^k\cup\omega^k$, where $n_0$ is a multiple of $2n$.  This can be satisfied when $L$ is sufficiently small.

In order to obtain the required phase symmetries we need the sets $\Omega^k, \omega^k$ to be invariant, up to reindexing, under the rotation by $2\pi/n$ and by the reflection about the real-axis (i.e. under the $2n$ element dihedral group). 
Let $\xi_n = e^{2\pi i/n}$. We require the $2n$th roots of unity $1, e^{\pi i/(2n)}, \xi_1, e^{3\pi i/(2n)}, \xi_2, \dots, e^{(2n-1)\pi/(2n)}$ and their conjugates to be the center of some of the intervals $\omega^{k}$. 
Let $p = n_0/n\in \N$. As in Case (ii), let $\zeta^s = \xi_n$.

Consider the sets $\omega^0, \Omega^1, \omega^1, \Omega^2, \dots, \omega^{2np-1}, \Omega^{2np}$.
Consider the arc $I_r=\{e^{i\theta}\in\bS: r\frac{2\pi}{2n}< \theta < (r+1)\frac{2\pi}{2n}\}$. So, except for the $2n$th roots of unity, the arcs $I_0, \dots, I_{2n-1}$ cover $\bS$. Moreover, each $\Omega^k$ is a subset of one of the $I_r$. 

Let $R_\varphi$ denote the rotation of the plane around the origin by $\zeta$ and $R_\phi$ the reflection of the complex plane across the real axis (i.e. conjugation). The $2n$ transformations of the plane composing the dihedral group $D_n$ are generated by these actions and can be expressed as $R_\varphi^aR_\phi^b$ for $a= 0, \dots, n-1$, $b=0, 1$. 
Let $\pi^{a,b} \in D_n$ be defined as $\pi^{a,b}= R_\varphi^{-a}R_\phi^{b}$ so that
\[\varphi^a\circ\phi^b(E_{\Omega}(U)) = \varphi^a(E_{R_\phi^{b}(\Omega)}(U)) = E_{R_\varphi^{-a}R_\phi^{b}(\Omega)}(U) = E_{\pi^{a,b}(\Omega)}(U)\]
and since $R_\phi R_\varphi = R_\varphi^{-1}R_\phi$,
\[\pi^{a,b} \pi^{c,d} = R_\varphi^{-a}\left(R_\phi^{b}R_\varphi^{-c}\right)R_\phi^{d} = R_\varphi^{-a}\left(R_\varphi^{(1-2b)c}R_\phi^{b}\right)R_\phi^{d} = \pi^{a+(1-2b)c, b+d}.\]
By assumption, $\varphi, \phi$ and $R_\varphi, R_\phi$ satisfy the same dihedral relations.

We claim that we only need to show that conditions (1.), (2.), and (3.) are satisfied when instead of using 
$F_{\pi^{a,b}(\Omega_0), \pi^{a,b}(\Omega)}$ and associated $F_{\pi^{a,b}(\Omega_0), \pi^{a,b}(\Omega),-}$, $F_{\pi^{a,b}(\Omega_0), \pi^{a,b}(\Omega),+}$,
we use $\varphi^a\circ\phi^b(F_{\Omega_0, \Omega}), \varphi^a\circ\phi^b(F_{\Omega_0, \Omega, -(1-2b)}), \varphi^a\circ\phi^b(F_{\Omega_0, \Omega}, +(1-2b))$, respectively. 
Notice that there is a interchange in $\pm$ exactly when $b=1$ due to the reflection $R_\phi$ reversing orientation. 
Verifying all these conditions is essentially identical to what we did before. 
The sets Consider the sets $\Omega^1, \omega^1, \Omega^2, \dots, \omega^{2np-1}, \Omega^{2np}, \omega^0$ are (in this order):
\begin{align*}
&\Omega^1, \omega^1, \dots, \omega^{p/2-1}, \Omega^{p/2}, \omega^{p/2},
\\
&R_\varphi R_\phi(\Omega^{p/2}), R_\varphi R_\phi(\omega^{p/2-1}),  \dots, R_\varphi R_\phi(\omega^1), R_\varphi R_\phi(\Omega^1), R_\varphi R_\phi(\omega^0)
\\
&R_\varphi(\Omega^1), R_\varphi(\omega^1), \dots, R_\varphi(\omega^{p/2-1}), R_\varphi(\Omega^{p/2}), R_\varphi(\omega^{p/2}), 
\\
&R_\varphi^2 R_\phi(\Omega^{p/2}), R_\varphi^2 R_\phi(\omega^{p/2-1}),  \dots, R_\varphi^2 R_\phi(\omega^1), R_\varphi^2 R_\phi(\Omega^1), R_\varphi^2R_\phi(\omega^0), 
\\
&R_\varphi^2(\Omega^1), R_\varphi^2(\omega^1), \dots, R_\varphi^2(\omega^{p/2-1}), R_\varphi^2(\Omega^{p/2}), R_\varphi^2(\omega^{p/2}),
\\
&\;\;\;\;\;\;\;\;\;\dots
\\
&R_\varphi^{n-1}(\Omega^1), R_\varphi^{n-1}(\omega^1), \dots, R_\varphi^{n-1}(\omega^{p/2-1}), R_\varphi^{n-1}(\Omega^{p/2}), R_\varphi^{n-1}(\omega^{p/2}),
\\
&R_\phi(\Omega^{p/2}), R_\phi(\omega^{p/2-1}),  \dots, R_\phi(\omega^1), R_\phi(\Omega^1), \omega^0.
\end{align*}
The first line corresponds to the arcs in the fundamental domain $I_0$. To get the arcs in the arc $I_1$, we reflect the arcs in $I_0$ across the line passing through $e^{2\pi/(2n)}$. Then to get all the other arcs we rotate the arcs in $I_0, I_1$ by applying $R_\varphi$ which multiplies by $\xi_n$. Note that $R_\phi(\omega^0) = \omega^0$ and likewise if $\omega^k$ is an arc containing $e^{2\pi r/(2n)}$ then the reflection $R_\varphi^rR_\phi$ across the line that passes through that point will preserve $\omega^k$ (but reverse its orientation).

We now explain how we use this. For the $\Omega^k$ in the fundamental arc $I_0$, we use projections $F_k, F_{k,-}, F_{k,+}$ guaranteed by conditions 1-3. For $\Omega^k$ in general, there are unique $a\in \{0, \dots, n-1\}$ and $b \in \{0, 1\}$ so that $\pi^{a,b}(\Omega^k)\subset I_0$. Moreover, there is a unique $k_{a,b}$ such that $\pi^{a,b}(\Omega^k) = \Omega^{k_{a,b}}$.
So, we define $F_k = \varphi^a\circ\phi^b(F_{k_{a.b}})$, $F_{k,-} = \varphi^a\circ\phi^b(F_{k_{a,b},-(1-2b)})$, $F_{k,+} = \varphi^a\circ\phi^b(F_{k_{a,b},+(1-2b)})$.

The projections $F_k$ for $k = 1, \dots, n_0=np$  are (in this order):
\begin{align*}
&F_1, \dots, F_{p/2},\\
&\varphi^{-s}\circ\phi(F_{p/2}), \dots, \varphi^{-s}\circ\phi(F_{1}),
\\
&\varphi^{-s}(F_1), \dots, \varphi^{-s}(F_{p/2}),\\
&\varphi^{-2s}\circ\phi(F_{p/2}), \dots, \varphi^{-2s}\circ\phi(F_{1}),\\
&\varphi^{-2s}(F_1), \dots, \varphi^{-2s}(F_{p/2}),\\
&\;\;\;\;\;\dots. \\
&\varphi^{-(n-1)s}(F_1), \dots, \varphi^{-(n-1)s}(F_{p/2}),\\
&\phi(F_{p/2}), \dots, \phi(F_{1}).
\end{align*}
The transformations $R_\varphi, R_\phi$ induce index changes as follows:
\[R_\varphi(\omega^k) = \omega^{k+p}, \;\;\; R_\phi(\omega^k) = \omega^{np-k},\]
\[R_\varphi(\Omega^k) = \Omega^{k+p}, \;\;\; R_\phi(\Omega^k) = \Omega^{np+1-k}.\]
The maps $\varphi, \phi$ induce index changes as follows:
\[\varphi^{-s}(F_k) = F_{k+p}, \;\;\; \phi(F_k) = F_{np+1-k}.\]

The expression $G_k = F_{k,+}^c+E_{\omega^k}(U)+F_{k+1,-}^c$ holds for all $k$. It is clear from our list of projections and sets above that $\varphi^{-s}(G_k) = G_{k+p}$. The dihedral commutation relation implies that 
\[\phi \circ\varphi^{-rs} = \varphi^{rs}\circ \phi = \varphi^{rs-ns}\circ \phi = \varphi^{-(n-r)s}\circ \phi\]
hence
\begin{align*}
\phi(G_k)&=\phi(F_{k,
+}^c+E_{\omega^k}(U)+F_{k+1,-}^c) \\
&= F_{np+1-k,
-}^c+E_{\omega^{np-k}}(U)+F_{np-k,+}^c=G_{np-k}.
\end{align*}
So,
\begin{align*}
\varphi^{-s}(U'') &= \sum_k\left(a_{\Omega^{k}}F_{k+p}+a_{\omega^{k}}G_{k+p}\right)\\
&=\xi_n^{-1}\sum_k\left(a_{\Omega^{k+p}}F_{k+p}+a_{\omega^{k+p}}G_{k+p}\right)= \xi_n^{-1} U'',
\end{align*}
\begin{align*}
\phi(U'') &=  \sum_k \left(\overline{a_{\Omega^k}}\phi(F_k)+\overline{a_{\omega^k}}\phi(G_k)\right)\\
&\sum_k\left(a_{\Omega^{np+1-k}}F_{np+1-k}+a_{\omega^{np-k}}G_{np-k}\right)=U'',
\end{align*}
Iterating this gives $\varphi(U'') = \zeta U''$.

As in the cases (i), (ii), the positive linear map $B \mapsto B''$ commutes with $\varphi$ and $\phi$ so $B''$ satisfies the required phase symmetries.
This completes the proof.
\end{proof}

We now state these results to the case when $U$ is self-adjoint instead of being unitary. We will instead refer to this operator as $X$. The only non-degenerate phase symmetries for self-adjoint operators are symmetries and antisymmetries, so we could simply apply what we proved to the Cayley transform of $X$ to obtain the following. We however will just skim over the details of how to modify the proofs given above so these results hold so as to obtain better estimates.
\begin{lemma}\label{proj requirement sa} Let $\A$ be a von Neumann algebra and ${S}$ some collection of weakly continuous linear symmetry maps on $\A$. Suppose that $X, B \in \A$ with $X$ self-adjoint and ${S}$-symmetric. 

Let
$\Omega_0\subset\subset \Omega$ be intervals in $\R$.
Let $\Omega_-, \Omega_+$ be the left and right disjoint intervals whose union is $\Omega\setminus \Omega_0$. 

Suppose that there exist commuting $X', B' \in \A$ with $X'$ self-adjoint and ${S}$-symmetric.
Then there is an ${S}$-symmetric projection $F=F_{\Omega_0, \Omega}\in \A$ with the following properties:
\begin{enumerate}[label=(\roman*)]
\item $
E_{\Omega_0}(X) \leq F \leq E_{\Omega}(X),
$ 
\item $F$ has the decomposition $F=F_{-} + E_{\Omega_0}(X) + F_+$, where $F_-, F_+$ are ${S}$-symmetric projections with \begin{align}
F_- \leq  E_{\Omega_-}(X),\; F_+ \leq E_{\Omega_+}(X).
\end{align}
\item 
$\displaystyle\|\,[F,B]\,\| \leq \left(\frac{24}{d(\R\setminus \Omega, \Omega_0)}+\max\left(\frac{16}{d(\R\setminus \Omega, \Omega_0)}, \frac{32}{d(\Omega_-, \Omega_+)}\right)\right)\|X'-X\|\|B\|$\\ 
\textcolor{white}{\_} \hspace{4in} $+ \|B'-B\|.$
\end{enumerate}
Also, the projections do not depend on $B$ or $B'$.
\end{lemma}
\begin{proof}
The proof of Lemma \ref{proj requirement} word-for-word proves this result except that arcs in $\bS$ are replaced with intervals in $\R$ and in this case $c = 1$.
\end{proof}

\begin{lemma}
\label{proj-converse sa}
Let $\A$ be a unital $C^\ast$-algebra and ${S}$ some collection of weakly continuous linear symmetry maps on $\A$.  Let $X, B \in \A$ with $X$ being self-adjoint and ${S}$-symmetric.

Suppose that there is a constant $\epsilon(L)$ (explicitly depending on some parameter $L$ and implicitly depending on $X, B$) so that for all intervals $\Omega_0 \subset\subset \Omega \subset\subset \R$ with the left and right subintervals $\Omega_-, \Omega_+$ of $\Omega\setminus \Omega_0$, there is an ${S}$-symmetric projection $F=F_{\Omega_0, \Omega}\in \A$ with the following properties.  
\begin{enumerate}
\item $E_{\Omega_0}(X) \leq F\leq E_{\Omega}(X)$.
\item $F$ can be decomposed as $F = F_- + E_{\Omega_0}(X) + F_+$, where $F_- \leq E_{\Omega_-}(X)$, $F_+ \leq E_{\Omega_+}(X)$ are ${S}$-symmetric projections.
 \item $F$ satisfies the commutator estimate
\begin{align}\label{commutator-leqconv sa}
\|\,[F,B]\,\| \leq \epsilon(L),
\end{align}
whenever $\dist(\R\setminus \Omega, \Omega_0)$ and $\dist(\Omega_-, \Omega_+)$ are at least $L$ and at most $L_0$.
\end{enumerate}
Then for any $L\in (0, L_0]$, there are commuting $X'', B''\in \A$ with $X''$ self-adjoint and ${S}$-symmetric satisfying
\begin{align}\label{ABestimates sa}
\|X''-X\| \leq 5L, \; \|B''-B\| \leq \frac{6C_\rho}{L}\|[X,B]\|+2\|B\|\epsilon(L).
\end{align}
The construction of $X''$ only depends on $X, L,$ and the projections. The map defining $B \mapsto B''$ is a completely positive linear map of norm $1$ that commutes with each symmetry in ${S}$.

\vspace{0.05in}

In each of the following cases that we also have that $X$ has an antisymmetry and $B$ has a phase symmetry, $X''$ and $B''$ have additional symmetry as well:
\begin{enumerate}[label=(\roman*)]
\item Suppose that $\phi$  is a weakly continuous conjugate-linear symmetry map on $\A$ with order $2$ that commutes with each symmetry of ${S}$. If $\phi(X)=-X$  and $\phi(B)=\eta B$ (resp. $\phi_{\ast}(B)=\eta B$), then $X'', B''$ can be chosen so that $\phi(X'')=- X''$ and $\phi(B'')=\eta B''$ (resp. $\phi_{\ast}(B'')=\eta B''$).

\item Suppose that  $\varphi$ is a weakly continuous linear symmetry map on $\A$ with order $2$ that commutes with each symmetry of ${S}$. If $\varphi(X)=-X$ and $\varphi(B)=\eta B$ (resp. $\varphi_{\ast}(B)=\eta B$), then $X'', B''$ can be chosen so that $\varphi(X'')=-X''$ and $\varphi(B'')=\eta B''$ (resp. $\varphi_{\ast}(B'')=\eta B''$).
\end{enumerate}
\end{lemma}
\begin{proof}
The proof of this result follows from arguments similar to that of the proof of Lemma \ref{proj-converse} and Lemma \ref{proj-converse extra sym}. We outline the changes. 

We let $\Omega^k$ be closed intervals in $\R$ of length $4L$ separated by open intervals $\omega^k$ of length $2L$. 
We require the center of $\omega^0$ to contain $0$. We let $k$ range over $\Z$, although only finitely many of the relevant spectral projections will be non-zero for any particular $X$. So, we have intervals
$\dots, \Omega^{-1}, \omega^{-1}, \Omega^0, \omega^0, \Omega^1, \omega^1, \Omega^2, \dots$ partition $\R$ into intervals.

Note that there is no additional limitation on the size of $L$ and there is no range of values that the length of the $\omega^k$ must satisfy because arc-length is identical to distance in $\R$.

The definitions of the subintervals of $\Omega^k$ and the projections $F_k$ and $G_k$ are identical to that given previously and we get projections: $E_j, j \in \Z$:
\[\dots, F_{-1}, G_{-1}, F_0, G_0, F_1, G_1, F_2, \dots\] that form a resolution of the identity. Again, only finitely many of these are non-zero.

The definitions of $U''$ and $B''$ are the same, except that we call the first operator $X''$. The only difference in the estimates are the following.
First, we instead apply the localization operator of Lemma \ref{localization lemma}. This introduces the numerical constant $2C_\rho$ instead of $4\sqrt2C_\rho$ into our estimates.

The second difference in the estimates is based on the distance between the center of an interval to its endpoint. The distance from the center of $\omega^k$ to its endpoints is $L$ and the distance of the center of $\Omega^k$ to its endpoints is $2L$.
So, $\|X''-X\| \leq 2L+(L+2L)=5L$.

So, the same arguments for the second estimate give for $\delta = \|[X,B]\|$:
\begin{align*}
    \|X''-X\| \leq 5L, \; \|B''-B\| \leq \frac{6C_\rho}{L}\delta +2\|B\|\epsilon(L).
\end{align*}

We now discuss the symmetries. As before, all the linear symmetries of $X$ pass to $X''$ and the map $B \mapsto B$ commutes with all these linear symmetries.

We now address the antisymmetry in the construction. By letting $\varphi=\phi_\ast$, we may reduce Case (i) to Case (ii). So, suppose that we are in Case (ii). The intervals are symmetric about $0$ up to a relabeling: with $-\omega^k = \omega^{-k}$ and $-\Omega^k = \Omega^{-k+1}$. Notice that both these index changes have order $2$.

As before, it is easy to show that we can use $\varphi(F_{\Omega_0, \Omega})$, $\varphi(F_{\Omega_0, \Omega,+})$, $\varphi(F_{\Omega_0, \Omega,-})$ instead of $F_{-\Omega_0, -\Omega}$, $F_{-\Omega_0, -\Omega, -}$, $F_{-\Omega_0, -\Omega, +}$, respectively.

Let $F_k = F_{\Omega_0^k, \Omega^k}, F_{k,-}, F_{k,+}$ be the projections guaranteed to exist for $k \geq 1$. These correspond to projections in $\R_{>0}$. We know that $\varphi$ acting on a spectral projection of $X$ satisfies:
\[\varphi(E_{\Omega}(X)) = E_{\Omega}(-X) = E_{-\Omega}(X).\]
So, the projections $F_1$, $F_2$, \dots are orthogonal $F_k \leq E_{(0, \infty)}(X)$ for $k \geq 1$. So, $\varphi(F_1)$, $\varphi(F_2)$, \dots are orthogonal and $\varphi(F_k) \leq E_{(-\infty,0)}(X)$ for $k \geq 1$. This shows that the projections $F_{k_1}, \varphi(F_{k_2})$ for $k_1, k_2 \geq 1$ are orthogonal.

We then define $F_k, F_{k,-}, F_{k,+}$ for $k \leq 0$ by $F_k = \varphi(F_{-k+1})$, $F_{k,-} = \varphi(F_{-k+1,+})$, $F_{k,+} = \varphi(F_{-k+1,-})$. 
Because $\varphi$ has order $2$, it follows that the identities $\varphi(F_k)= F_{-k+1}$, $\varphi(F_{k,-})= F_{-k+1,+}$, $\varphi(F_{k,+})= F_{-k+1,-}$ hold for all $k$.

The projections: 
\begin{align*}
\dots, F_{-1},\;\; &G_{-1},\;\; F_0, 
\\ &\;\;G_0, \\
F_1&,\;\; G_1,\;\; F_2, \dots\end{align*} 
are (in this order):
\begin{align*}
\dots\;\;\; \varphi(F_2),\;\;\; \varphi(F_{2,-})+E_{\omega^{-1}}&(X)+\varphi(F_{1,+}),\;\;\; \varphi(F_1),\\ 
&\varphi(F_{1,-})+E_{\omega^{0}}(X)+F_{1,-},\\
&\;\;\;\;\;\;\;\;\;\;\;F_1, \;\;\; F_{1,+}+E_{\omega^1}(X)+F_{2,-}, \;\;\; F_2,\;\;\; \dots.
\end{align*}
Then because \[-a_{\Omega^{k}} = a_{-\Omega^{k}}= a_{\Omega^{-k+1}}\]
and $-a_{\omega^k} = a_{\omega^{-k}}$, we have that
\begin{align*}
\varphi(X'') &= \sum_k (a_{\Omega^k}F_{-k+1}+a_{\omega^k}G_{-k})\\
&= -\sum_k (a_{\Omega^{-k+1}}F_{-k+1}+a_{\omega^{-k}}G_{-k}) = -X''.
\end{align*}
Also, 
\[\varphi(B'') = \sum_k (F_{-k+1}\L(\varphi(B))F_{-k+1}+G_{-k}\L(\varphi(B))G_{-k})\]
so $B\mapsto B''$ commutes with $\varphi$. It is clear that this map also satisfies the remaining desired properties as well. This completes the proof.
\end{proof}

\section{The Symmetry Bootstrap}
\label{SymmBootstrap}

In this section we will use the results of the previous two sections to bootstrap symmetry for almost commuting operators.
We prove the self-adjoint results first:
\begin{thm}\label{self-adjoint bootstrap}
Let $\A$ be a von Neumann algebra and ${S}$ some collection of weakly continuous linear symmetry maps on $\A$. Let $\varphi$ be a weakly continuous linear symmetry map on $\A$ with order $2$ that commutes with each symmetry of ${S}$. 
Let $X, B \in \A$ with $X$ being self-adjoint, ${S}$-symmetric and with $\|X\|, \|B\| \leq 1$. Suppose further that $\varphi(X)=-X$ and $\varphi(B)=\eta B$ (resp. $\varphi_{\ast}(B)=\eta B$) for a phase $\eta \in \C$. 

If there are commuting $X', B' \in \A$ with $X'$ self-adjoint and ${S}$-symmetric and with
\[\varepsilon = \max(\|X'-X\|, \|B'-B\|),\]
then there are commuting $X'', B'' \in \A$ with $X''$ self-adjoint and ${S}$-symmetric,
\[\|X''- X\|, \|B''-B\| \leq 45\sqrt{\varepsilon},\]
$\varphi(X'')=-X''$, and $\varphi(B'')=\eta B''$ (resp. $\varphi_{\ast}(B'')=\eta B''$).
The operator $X''$ constructed only depends on $X, X', \varphi, \varepsilon$. The map defining $B \mapsto B''$ is a completely positive linear map of norm $1$ that commutes with each symmetry in ${S}$.
\end{thm}
\begin{proof}
Let $\delta = \|[X,B]\|$.
Note that 
\[\delta \leq \|[X',B']\| + 2\|X'-X\|+2\|B'-B\| \leq 4\varepsilon.\]

Then we can apply Lemma \ref{proj requirement sa} to obtain certain projections $F$ that satisfy conditions (i), (ii), and (iii). The last of which is:
\begin{align*}
\|\,[F,B]\,\| &\leq \left(\frac{24}{d(\R\setminus \Omega, \Omega_0)}+\max\left(\frac{16}{d(\R\setminus \Omega, \Omega_0)}, \frac{32}{d(\Omega_-, \Omega_+)}\right)\right)\varepsilon+\varepsilon
\end{align*}
Note that if $\dist(\R\setminus \Omega, \Omega_0), \dist(\Omega_-, \Omega_+)\geq L$ and $L \leq 1$ then
\begin{align*}
\|\,[F,B]\,\| &\leq \left(\frac{24}{L} + \frac{32}{L} \right)\varepsilon + \varepsilon= \frac{56}{L}\varepsilon+\varepsilon
\end{align*}

We then see that this provides the assumed projections for Lemma \ref{proj-converse sa} with $\epsilon(L)= 56\varepsilon/L+\varepsilon$. So, applying Lemma \ref{proj-converse sa} to $X, B$, we obtain commuting $X'', B''$ with the required symmetries and from (\ref{ABestimates sa}) we obtain
\begin{equation}\label{bootstrap sa}
\|X''-X\| \leq 5L, \;   \|B''-B\| \leq \frac{6C_\rho}{L}\delta +\frac{112}{L}\varepsilon+2\varepsilon.
\end{equation}
We use (\ref{C_rho}) and $\delta \leq 4 \varepsilon$ to obtain
\[\|X''-X\| \leq 5L, \; \|B''-B\| \leq \frac{402}{L}\varepsilon.\]
Choosing $L = \sqrt{402\varepsilon/5}$, we obtain 
\[\|X'-X\|, \|B'-B\| \leq 45\sqrt{\varepsilon}.\]

This inequality is valid when $402\varepsilon/5 \leq  1$ which happens when
$\varepsilon\leq 5/402$. 
Note that $45\sqrt{5/402}>1$, so when $\varepsilon > 5/402$, we can choose $X'' = 0$ to obtain the desired estimate.
\end{proof}
\begin{remark}
The process we just outlined is what we refer to as ``projection symmetry bootstrapping''.
What we learn from this is that we got different commuting operators $X'', B''$ that satisfy an estimate with a different constant and half the exponent. 

We will not take the constant in the inequality too seriously for two reasons. The first is that we demonstrated that it is not \emph{too} large for the self-adjoint case. We expect similar types of estimates for the unitary case.  The second reason is that we see that this bootstrapping worsens the exponent so we do not expect the estimates gotten to be optimal anyway. There are good reasons to believe that the optimal exponent for many almost commuting operator problems with symmetries should be $\alpha = 1/2$ as in \cite{kachkovskiy2016distance}.

Note that this loss in exponent is similar to how there is a loss in the optimal asymptotic estimate when using Davidson's projection reformulation of Lin's theorem as discussed in the introduction of \cite{herrera2020hastings}. 
\end{remark}

As a consequence we obtain Lin's theorem with linear symmetries and a conjugate-linear symmetry and Lin's theorem with linear symmetries and a linear antisymmetry:
\begin{thm}\label{Symmetry Lin}
Suppose that $\A$ is a von Neumann algebra and ${S}$ is an admissible collection of weakly continuous linear symmetry maps on $\A$ so that $(\A, S)$ has TR rank 1. Let $\epsilon_\ell(\delta)$ be as in Theorem \ref{Linear Lin}. 

Let $\phi$ be a weakly continuous conjugate-linear symmetry map on $\A$ with order 2 that commutes with every symmetry map in ${S}$.
(Resp. let $\varphi$ be a weakly continuous linear symmetry map on $\A$ with order 2 that commutes with every symmetry map in ${S}$.)

Then if $A \in \A$ is an $S\cup\{\phi\}$-symmetric (resp. ${S}$-symmetric and $\varphi$-antisymmetric) contraction, there is a normal $A''\in \A$ that is $S\cup\{\phi\}$-symmetric (resp. ${S}$-symmetric and $\varphi$-antisymmetric) and satisfies
\[\|A''-A\| \leq 90\sqrt{\epsilon_\ell\left(\|[A^\ast,A]\|\right)}.\]
\end{thm}
\begin{proof}
By Proposition \ref{S0 Reduction Prop} and Proposition \ref{properties}\ref{B subalg}, we may replace $\A$ with $\A_{S}$ so that there is a set of symmetries $S_1$ containing at most two symmetries, each of order $2$, one $\phi \in S_1$, conjugate linear. There may also be a linear anti-multiplicative symmetry $\varphi\in S_1$. We know that $(\A_{S}, S_1)$ has TR rank 1 by Theorem \ref{finite TR rank 1}. Given $A \in \A_{S}$ that is $S_1$-symmetric, we need to find $A' \in \A_{S}$ normal and $S_1$-symmetric. 

Let $X = \Im(A), B = \Re(A)$. Then $X, B \in \A_{S}$ are both $S_\ast = S_1 \setminus \{\phi\}$-symmetric, $B$ is $\phi$-symmetric, and $X$ is $\phi$-antisymmetric.

By Theorem \ref{Linear Lin}, there is a normal $A' \in \A$ with $\|A'-A\| \leq \varepsilon_\ell(\|[A^\ast, A]\|)$ that is $S_\ast$-symmetric. Because $S_\ast$ is either empty or contain a linear symmetry, we know that $X'=\Im(A')$, $B' = \Re(A')$ are $S_\ast$-symmetric.

We now use Theorem \ref{self-adjoint bootstrap}(i), to conclude that there are commute self-adjoint $X'', B''$ that $S_\ast$-symmetric, $X''$ is $\phi$-antisymmetric, $B''$ is $\phi$-symmetric, and
\begin{align*}
\|X''-X\|, \|B''-B\| &\leq 45\sqrt{\max(\|X'-X\|, \|B'-B\|)} \\
&\leq 45 \sqrt{\epsilon_\ell(\|[A^\ast, A]\|)}.
\end{align*} 
Then $A'' = B'' + iX''$ is normal, is $S_\ast$-symmetric, is $\phi$-symmetric, and satisfies
\[\|A''-A\| \leq  \|X''-X\|+ \|B''-B\|.\] Hence $A''$ is the desired operator.
\end{proof}
We also have the analogous result for finite dimensional $C^\ast$-algebras:
\begin{thm}\label{Symmetry Lin fin dim}
Suppose that $\A$ is a finite dimensional $C^\ast$-algebra and ${S}$ is an admissible collection of linear symmetry maps on $\A$. Let $\epsilon_\ell(\delta)$ be as in Theorem \ref{Linear Lin}. 

Let $\phi$ be a conjugate-linear symmetry map on $\A$ that commutes with every symmetry map in ${S}$.
(Resp. let $\varphi$ be a linear symmetry map on $\A$ that commutes with every symmetry map in ${S}$.)

Then if $A \in \A$ is an $S\cup\{\phi\}$-symmetric (resp. ${S}$-symmetric and $\varphi$-antisymmetric) contraction, there is a normal $A''\in \A$ that is $S\cup\{\phi\}$-symmetric (resp. ${S}$-symmetric and $\varphi$-antisymmetric) and satisfies
\[\|A''-A\| \leq 90\sqrt{\epsilon_\ell\left(\|[A^\ast,A]\|\right)}.\]
\end{thm}
\begin{proof}
Since $\mathcal A$ is finite-dimensional, $(\mathcal A, S \cup \{\phi^2\})$ (resp. $(\mathcal A, S \cup \{\varphi^2\})$) has TR rank 1. Therefore, the result follows from Theorem \ref{Symmetry Lin} using Proposition \ref{S0 Reduction Prop} and Theorem \ref{finite TR rank 1}.   
\end{proof}

We have the following:
\begin{corollary}\label{sa normal sym}
Suppose that $\A$ is a finite dimensional $C^\ast$-algebra and ${S}$ is an admissible collection of linear symmetry maps on $\A$. Let $\varphi$ be a linear symmetry map on $\A$.

Let $X, B \in \A$ be ${S}$-symmetric contractions, where $X$ is self-adjoint and $\varphi$-antisymmetric and $B$ is $\varphi$-symmetric (resp. $\varphi$-antisymmetric, resp. $\varphi_\ast$-symmetric). Suppose that that there exist commuting ${S}$-symmetric $X', B' \in \A$, where $X'$ is self-adjoint and 
\[\varepsilon = \max(\|X'-X\|, \|B'-B\|).\]

Then there exist commuting ${S}$-symmetric $X'', B''' \in \A$ so that $X''$ is self-adjoint and $\varphi$-antisymmetric, $B'''$ is normal and is $\varphi$-symmetric (resp. $\varphi$-antisymmetric, resp. $\varphi_\ast$-symmetric) and
\begin{align}\nonumber
\|X''-X\|&\leq 45\sqrt{\varepsilon} \\ \|B'''-B\| &\leq 45\sqrt{\varepsilon} + \epsilon_{\ell}\left(\|[B^\ast, B]\|+180\sqrt{\varepsilon}\right).\label{sa normal ineq}
\end{align}
\end{corollary}
\begin{proof}
We first construct $X''$ and $B''$ as in Theorem \ref{self-adjoint bootstrap}. 
Let $\B$ be the unital $C^\ast$-subalgebra of all elements that commute with $X''$. Note that $B'' \in \B$. Because $X''$ is ${S}$-symmetric, each symmetry $\tilde\varphi \in {S}$ has $\B$ as an invariant set:
\[[C, X''] = 0 \Leftrightarrow 0=\tilde\varphi([C,X'']) = \pm [\tilde\varphi(C), X''],\]
where the sign is determined by whether $\varphi$ is multiplicative or anti-multiplicative. Likewise, $\varphi$ and $\varphi_\ast$ also have $\B$ as invariant set. So, we may restrict all relevant symmetries to $\mathcal B$. 

Because the map $B \mapsto B''$ has norm $1$, $B''$ is a contraction. Also, $B''$ is $S$-symmetric and $\varphi$-symmetric (resp. $\varphi$-antisymmetric, resp. $\varphi_\ast$-symmetric) by Theorem \ref{self-adjoint bootstrap}.
We then obtain $B'''$ by applying Theorem \ref{Symmetry Lin fin dim} to $B''$.
We conclude with estimating
\[\|[(B'')^\ast, B'']\| \leq \|[B^\ast, B]\| + 4\|B''-B\| \leq \|[B^\ast, B]\|+180\sqrt{\varepsilon}.\]
\end{proof}

See the introduction for consequences of these results for two and three almost commuting self-adjoint matrices.

We now prove the following theorem, which is the symmetry bootstrap for a unitary operator. We do not provide numerical estimates, but as in the self-adjoint case we do not expect them to be too large.
\begin{thm}\label{unitary bootstrap}
Let $\A$ be a von Neumann algebra and ${S}$ some collection of weakly continuous linear symmetry maps on $\A$. Let $U, B \in \A$ with $U$ being unitary and ${S}$-symmetric and $\|B\| \leq 1$. Suppose that there are commuting $U', B' \in \A$ with $U'$ unitary and ${S}$-symmetric with
\[\varepsilon = \max(\|U'-U\|, \|B'-B\|).\]
Suppose further that we have the additional phase symmetries for $U$ and $B$ in one of three cases given below. 

Then for $\|[U,B]\|$ small enough there are commuting $U'', B'' \in \A$ with ${S}$-symmetric $U''$ unitary, $U''$ constructed only depends on $U, U',\varepsilon, \varphi, \phi$,  and 
\[\|U''- U\|, \|B''-B\| \leq Const. \sqrt{\varepsilon},\]
where Const. is a universal constant. The map defining $B \mapsto B''$ is a completely positive linear map of norm $1$ that commutes with each symmetry in ${S}$.
Additional symmetries are also had in each case:
\begin{enumerate}[label=(\roman*)]
\item\label{conj-lin sym addition boot} Suppose that $\phi$ is a weakly continuous conjugate-linear symmetry map on $\A$ with order $2$ that commutes with each symmetry of ${S}$. Suppose that $\phi(U)=\zeta_0 U$ for $\zeta_0$ in the unit circle. Suppose that $\phi(B)=\eta_0 B$ (resp. $\phi_{\ast}(B)=\eta_0 B$) for $\eta_0\in \C$.

Then $U'', B''$ can be chosen so that $\phi(U'')=\zeta_0 U''$ and $\phi(B'')=\eta_0 B''$ (resp. $\phi_{\ast}(B'')=\eta_0 B''$).

\item\label{lin sym addition boot} Suppose that $\varphi$ is a weakly continuous linear symmetry map on $\A$ with finite order $n\geq 2$ that commutes with each symmetry of ${S}$. Suppose that $\varphi(U)=\zeta_1 U$ for $\zeta_1$ in the unit circle having order $n$. Suppose that $\varphi(B)=\eta_1 B$ (resp. $\varphi_{\ast}(B)=\eta_1 B$) for $\eta_1\in \C$.

Then $U'', B''$ can be chosen so that $\varphi(U'')=\zeta_1 U''$ and $\varphi(B'')=\eta_1 B''$ (resp. $\varphi_{\ast}(B'')=\eta_1 B''$), where Const. and how small $\|[U,B]\|$ is required to be depend on $n$.

\item Suppose that there are $\phi, \varphi$ as in both \ref{conj-lin sym addition boot} and \ref{lin sym addition boot} that additionally have the properties that they satisfy the dihedral commutation relation: 
$\phi\circ\varphi\circ \phi = \varphi^{-1}$. 

Then $U'', B''$ can be chosen to satisfy both of the phase symmetries of \ref{conj-lin sym addition boot} and \ref{lin sym addition boot}.

\end{enumerate}
\end{thm}

\begin{proof}
Let $\delta = \|[U,B]\|$.

We apply Lemma \ref{proj requirement} to obtain the existence of certain projections $F$ that satisfy conditions (i), (ii), and (iii) of that lemma. The last of which implies 
\begin{align*}
\|\,[F,B]\,\| &\leq \frac{Const.}{L}\varepsilon+\varepsilon
\end{align*}
when $\dist(\ell\setminus \Omega, \Omega_0), \dist(\Omega_-, \Omega_+)$ are between $L$ and $2L$ for $L \leq L_0$.

So we apply Lemmas \ref{proj-converse} and
\ref{proj-converse extra sym}. This, with the restriction that $L_0 \leq 1$, we obtain provides $\epsilon(L) =  Const.\varepsilon/L$.
Hence there are commuting $U'', B''$ with the desired symmetries so that
\begin{align*}
\|U''-U\| \leq Const.L, \; \|B''-B\| \leq \frac{Const.}{L}\delta +\frac{Const}{L}\varepsilon.
\end{align*}

Note that 
\[\delta = \|[U,B]\| \leq \|[U',B']\| + 2\|U'-U\|+2\|B'-B\| \leq 4\varepsilon.\] So, $\|B''-B\| \leq \frac{Const}{L}\varepsilon$.
We choose $L = \sqrt{Const. \varepsilon}$ to obtain
\begin{align*}
\|U''-U\|, \|B''-B\| \leq Const. \sqrt{\varepsilon}.
\end{align*}
\end{proof}

\section{Lin's Theorem with Rotational and Dihedral Symmetries}\label{rot-dih Lin section}

We now will extend Lin's theorem to include rotational and dihedral symmetries. We will first provide a generalization of the statement and proof of Lemma 6.2 of \cite{loring2016almost} concerning the symmetries of the polar decomposition.
\begin{lemma}\label{polarLemma}
Let $\varphi$ be a symmetry map on a unital $C^\ast$-algebra $\A$. Let $A \in \A$ be invertible and $A = U|A|$ be its polar decomposition.
\begin{enumerate}[label=(\roman*)] 
\item Suppose that $\varphi$ is multiplicative. Then $\varphi(A) = \zeta A$ if and only if $\varphi(U)=\zeta U$ and $\varphi(|A|)=|A|$.
\item Suppose that $\varphi$ is anti-multiplicative. Then if $\varphi(A) = \zeta A$, it follows that $\varphi(U)=\zeta U$ and $P = \frac12(|A|+|A^\ast|)$ is $\varphi$-symmetric with
\[\|A-UP\| \leq \frac12\sqrt{\|[A^\ast, A]\|},\]
\[\|[U,P]\| \leq \sqrt{\|[A^\ast, A]\|}.\]
\end{enumerate}
\end{lemma}
\begin{proof}
The polar decomposition satisfies the property that $U = A|A|^{-1}$ is unique and because $A(A^\ast A) = (AA^\ast)A$, it follows that $U = |A^\ast|^{-1}A$. The first part of the lemma follows from 
\[\varphi(|A|)= \varphi(\sqrt{A^\ast A}) = \sqrt{\varphi(A)^\ast\varphi(A)}=|A|\]
and
\[\varphi(U) = \varphi(A|A|^{-1}) = \varphi(A)\varphi(|A|)^{-1} = \zeta U.\]

For the second part, 
\[\varphi(|A|) = \sqrt{\varphi(A)\varphi(A)^\ast}=|A^\ast|\]   
so
\[\varphi(U) = \varphi(A|A|^{-1})=\varphi(|A|)^{-1}\varphi(A) = \zeta |A^\ast|^{-1}A = \zeta U\]
and $\varphi(P)=P$. 

We then use the square root estimate in (\cite{ando1988comparison}) to see that
\[\|P-|A|\|=\frac12\|\sqrt{A^\ast A}-\sqrt{AA^\ast}\| \leq \frac12\sqrt{\|A^\ast A - AA^\ast\|}.\]
From $A = U|A|$, we see that
\[\|A^\ast A - AA^\ast\| = \||A|^2-U|A|^2U^\ast\| = \|[U, |A|^2]\|.\]
Using the same reasoning with $A = |A^\ast|U$, we obtain $\|A^\ast A - AA^\ast\|=\|[U, |A^\ast|^2]\|$.
Then using the square root commutator with a unitary inequality proven in 
(\cite{pedersen1993commutator}), we see that
\[\|[U, |A|]\|=\|[U, \sqrt{|A|^2}]\|\leq \sqrt{\|[U, |A|^2]\|}=\sqrt{\|[A^\ast , A]\|} .\]
For the same reason, $\|[U, |A^\ast|]\|\leq \sqrt{\|[A^\ast , A]\|}$. So, we obtain the desired estimate for $\|[U,P]\|$.
\end{proof}
We now can prove Lin's theorem with rotational or dihedral symmetries for invertible elements:
\begin{thm}\label{Rotational-Dihedral Lin}
Suppose that $\A$ is a von Neumann algebra and ${S}$ is an admissible collection of weakly continuous linear symmetry maps on $\A$ so that $(\A, S)$ has TR rank 1. Let $\epsilon_\ell(\delta)$ be as in Theorem \ref{Linear Lin}. 

Then if $A \in \A$ is an $S$-symmetric contraction that is invertible, there is a normal $S$-symmetric $A''\in \A$ that satisfies
\[\|A''-A\| \leq \frac12\|[A^\ast, A]\|^{1/2}+Const.\sqrt{\epsilon_\ell\left(\|[A,A^\ast]\|^{1/2}\right)}\]
with additional symmetry as in one of the following two cases:
\begin{enumerate}[label=(\roman*)]
\item Suppose that $\varphi$ is a weakly continuous linear symmetry map on $\A$ with finite order $n\geq 2$ that commutes with each symmetry of ${S}$. Suppose that $A$ also satisfies $\varphi(A)=\zeta A$ for $\zeta$ in the unit circle having order $n$. 

Then $A''$ can be chosen so that $\varphi(A'')=\zeta A''$, where the Const. in the above estimate depends only on $n$.

\item Suppose that $\varphi$ is a weakly continuous linear symmetry map on $\A$ with finite order $n\geq 2$ that commutes with each symmetry of ${S}$. Suppose that $\phi$ is a weakly continuous conjugate-linear symmetry map on $\A$ with order $2$ that commutes with each symmetry of ${S}$. Suppose that $\varphi, \phi$ satisfy the dihedral commutation relation: 
$\phi\circ\varphi\circ \phi = \varphi^{-1}$.

Suppose that $A$ also satisfies $\varphi(A)=\zeta_1 A$ and $\phi(A) = \zeta_2 A$ for $\zeta_1, \zeta_2$ in the unit circle with $\zeta_1$ having order $n$. 

Then $A''$ can be chosen so that $\varphi(A'')=\zeta_1 A''$ and $\phi(A'')=\zeta_2A''$, where the Const. in the above estimate depends only on $n$.

\end{enumerate}
This result applies also to any non-invertible $A_0$ that can be approximated by invertible elements $A$ of $\mathcal A$ satisfying the above assumptions.
\end{thm}
\begin{proof}
By replacing $A$ with $\zeta_2^{1/2}A$, we may assume that $\zeta_2=1$. We write $\zeta=\zeta_1$.
Let $U, P \in \A$ be as in Lemma \ref{polarLemma}. In particular, $U, P$ are $S$-symmetric and $U, P$ satisfy phase symmetries due to $A$ satisfying phase symmetries in the respective case. Namely, in Case (i), $\varphi(U)=\zeta U$, $\varphi(P)=P$ and in Case (ii), $\varphi(U)=\zeta U$, $\varphi(P)=P$, $\phi(U)=U$, $\phi(P)=P$.

Applying the construction and estimates obtained in the proofs of Theorem 2.4 and 2.3 of \cite{hastings2010almost} and using our Lin's theorem with symmetries in place of Theorem 2.2 of \cite{hastings2010almost}, we see that there are commuting $U', P'$ which satisfy
\[\|U'-U\|, \|P'-P\|\leq 6\epsilon_\ell\left(\frac5{16}\|[U,P]\|\right)\leq 6\epsilon_\ell\left(\frac5{16}\|[A,A^\ast]\|^{1/2}\right).\]
and are additionally $S$-symmetric.

We now apply Theorem \ref{unitary bootstrap} with $B=P, B'=P'$ to obtain commuting unitary $U''$ and positive contraction $P''$ satisfying
\[\|U''-U\|, \|P''-P\|\leq Const.\sqrt{\epsilon_\ell\left(\frac5{16}\|[A,A^\ast]\|^{1/2}\right)}.\]

Moreover, $U'', P''$ are $S$-symmetric and satisfy the respective phase symmetries inherited from $U, P$ in each case. So, $A''= U''P''$ is a normal $S$-symmetric element of $\A$ that satisfies the appropriate phase symmetries: $\varphi(A'')=\zeta A''$, $\phi(A'')=A''$.
The desired estimate then follows from
\[\|A''-A\| \leq \|U''-U\|+\|P''-P\|+\|UP-A\|.\]
\end{proof}
\begin{remark}
If all the symmetry maps in $S$ are multiplicative then we can instead use Theorem \ref{Sym Lin KS} to obtain the refined estimate
\[\|A''-A\|\leq Const.\|[A^\ast, A]\|^{1/8}.\]
\end{remark}
\begin{remark}
In general, if we have any collection of linear phase symmetries: $\varphi_\alpha(A) = \zeta_\alpha A$ and any collection of conjugate-linear phase symmetries $\phi_\beta(A) = \zeta_\beta A$ then we can use a method similar to that of prior sections to reduce to the case of the theorem above provided that each of the $\zeta_\alpha$, $\zeta_{\beta_1}\overline{\zeta_{\beta_2}}$ have finite order and the appropriate order and conjugation relations between the $\varphi_\alpha, \phi_\beta$ hold. The value of Const. in the estimate of $\|A''-A\|$ will depend on the gcd of the order of these complex numbers. We omit a careful statement of this result.

As an example, if $\phi_1(A)=\zeta_1A, \phi_2(A)=\zeta_2A$ then if we define $\varphi = \phi_2\circ \phi_1$ then $\varphi$ is a linear symmetry. The appropriate commutation relations between $\phi_1, \phi_2$ will allow us to consider $\varphi, \phi_1$ phase symmetries instead of $\phi_!, \phi_2$ phase symmetries.

In particular, if two of the $\zeta_\beta$ phases do not have a rational angle between them then the associated reflections on $\C$ will generate an infinite subgroup of the orthogonal group $O(2)$. So, we should not expect the reduction to work in this case.
\end{remark}

\begin{example}
We now include an example to illustrate how it is necessary for the estimate to depend on $n$. Let $A\in M_n(\C)$ and $W=\diag(1, \zeta, \zeta^2, \dots, \zeta^{n-1})$ for $\zeta=e^{2\pi i/n}$. Then $W^{-1}AW = \zeta A$ if and only if $A$ is a bilateral weighted shift matrix, meaning that $Ae_k = a_ke_{k+1}$ with indices labeled cyclically so that $Ae_n = a_ne_1$.

It can be easily shown that when $A$ has this special form, 
\[\|[A^\ast, A]\| = \max\left||a_{k+1}|^2-|a_k|^2\right|\] and $A$ is normal if and only if the absolute values of the weights are constant. (See Section 5 of \cite{herrera2022constructing} for the details.)

Let $A$ be a bilateral weighted shift matrix with $a_k \geq 0$, $a_1=0$, $|a_{k+1}^2-a_k^2|\leq 1$, with $\|A\|=\max_k|a_k|$ maximized. Then $\|A\|\geq Const.\sqrt{n}$. 
Because $A$ is a bilateral weighted shift matrix, $W^{-1}AW = \zeta A$ and any normal bilateral weighted shift matrix $A'$ must satisfy $\|A'-A\| \geq \|A\|/2 \geq Const.\sqrt{n}$.

Expressing this inequality in a scaling invariant form, this shows that it is possible for the following inequality to hold
\[\inf_{A' \mbox{ normal}:\, W^{-1}A'W=\zeta A'}\|A'-A\| \geq Const.\sqrt{n \|[A^\ast, A]\|}.\]

In the infinite dimensional setting of $\A = B(\mathcal H)$ with $\mathcal H$ separable spanned by $e_k, k \in \Z$ and with $\zeta$ having infinite order, we can define $We_k = \zeta^k e_{k}$ to be a diagonal unitary operator and $A$ to satisfy $W^{-1}AW = \zeta A$. Then $A$ is a weighted shift operator: $Ae_k = a_ke_{k+1}$. Suppose that $a_0=0$ and $\|A\|=\sup_k|a_k|=1$.

By \cite[Theorem 2]{berg1975approximation}, if $\lim_{k\to -\infty}a_k = \lim_{k\to\infty}a_k$ then there is guaranteed to exist nearby normal operators in $B(\mathcal H)$.
However, the above arguments show that it is possible for $\|[A^\ast, A]\|$ to be arbitrarily small while any normal $A'$ that satisfies $W^{-1}A'W = \zeta A'$ will be at least a distance of $1/2$. 
\end{example}

\newpage

\vspace{1in}

\textbf{ACKNOWLEDGEMENTS}. The author would like to thank Eric A. Carlen for support during the writing of this paper and Terry A. Loring for a helpful discussions on his work on this topic and the status of the relevant open problem.

This research was partially supported by NSF grant DMS-1764254.

\renewcommand{\biblistfont}{%
	\normalfont
	\normalsize
}

\phantomsection

\addcontentsline{toc}{chapter}{Bibliography}

\bibliography{almostcommuting.bib}

\end{document}